\documentclass{article}
\usepackage{graphicx} % Required for inserting images
\usepackage[alphabetic, abbrev, msc-links, nobysame]{amsrefs}

\usepackage{comment}
\usepackage{physics}
\usepackage{amsmath,amsthm}
\usepackage[all,cmtip]{xy}
\usepackage{amssymb}
\usepackage{geometry}
\usepackage{fancyhdr}
\usepackage{lipsum}
\usepackage{mathtools}
\usepackage{relsize}
\usepackage{physics}
\allowdisplaybreaks
\usepackage{braket}
\usepackage{bm}
\usepackage{esint}
\usepackage{enumitem}
\usepackage{mathrsfs}
\usepackage{xcolor}
\usepackage{bbold}
\usepackage{comment}
\usepackage{hyperref}
\usepackage{amsthm}
\usepackage[alphabetic, abbrev, msc-links, nobysame]{amsrefs}
\usepackage{stmaryrd}
\usepackage{float}
\usepackage[normalem]{ulem}
\newenvironment{issue}{\color{red}}{\ignorespacesafterend}
\numberwithin{equation}{section}
\newtheorem{thm}{Theorem}[section]
\newtheorem{theorem}[thm]{Theorem}
\newtheorem{cor}[thm]{Corollary}
\newtheorem{definition}[thm]{Definition}
\newtheorem{assumption}[thm]{Assumption}
\newtheorem{lemma}[thm]{Lemma}

\newtheorem{remark}[thm]{Remark}
\newtheorem{prop}[thm]{Proposition}
\newtheorem{corollary}[thm]{Corollary}
\newtheorem{example}[thm]{Example}

\def\bi{
\begin{itemize}}
    \def\ei{
  \end{itemize}}
\def\bem{
\begin{emuerate}}
  \def\eem{
\end{enumerate}}
\def\beq{
\begin{eqnarray*}}
\def\eeq{
\end{eqnarray*}}

\def\N{\mathbb{N}} %naturals
\def\Z{\mathbb{Z}} %integers
\def\R{\mathbb{R}} %reals
\def\P{\mathbb{P}} %probability
\def\E{\mathbb{E}} %expectation
\def\T{\mathbb{T}} %Torus
\def\1{\mathbb{1}}
\DeclarePairedDelimiter{\Parentheses}{(}{)} \newcommand{\pr}{\Parentheses*}

\newcommand{\Besovp}[2]{\mathcal B^{#1}_{#2}}
\newcommand{\Binf}[1]{\mathcal B^{#1}_\infty}
\newcommand{\prk}[1]{{(#1)}}

\newcommand{\Rnum}[1]{\uppercase\expandafter{\romannumeral #1}}

\title{Wick Renormalized Parabolic Stochastic Quantization Equations on Rough Metric Measure Spaces}
\author{Hongyi Chen \footnote{Department of Mathematics Aarhus University, Denmark} \and Yifan (Johnny) Yang \footnote{School of Mathematics and Statistical Sciences, Arizona State University, Tempe, U.S.A}}
\allowdisplaybreaks

\begin{document}

\maketitle

\begin{abstract}
\noindent On metric measure spaces with sub-Gaussian heat kernel behavior in small time, we obtain sufficient conditions for the well-posedness of Wick renormalized stochastic quantization equations with polynomial interaction. Given the power of the nonlinearity, the local solution condition depends on the Hausdorff dimension $d_h$, the walk dimension $d_w$, and the maximal spatial H\"older regularity  $\Theta$ of the heat kernel. Global solutions exist under a slightly more restrictive condition on the same parameters, corresponding to the full range in which the a priori energy argument of \cite{MW17,TW18} can be carried out. For all global solutions, we construct an invariant measure for the Markov process defined by the solution.  Our results apply to many rough spaces such as Barlow--Kigami type fractals as well as their Cartesian products and open up the possibility of making rigorous various structures in quantum field theory and statistical mechanics in non-integer dimensions. In the process, we build entirely from the short-time heat semigroup the necessary analytic framework that accommodates the issues which come with allowing rough local geometry.
\end{abstract}
{ \tableofcontents}

\section{Introduction}\label{sec: Intro}

For $n\geq 3$ odd, we seek to develop a solution theory for the  $\Phi^{n+1}$ equations, formally written as
\begin{equation}\label{eq:phi n formal}
\partial_t \phi = -L\phi - \phi - \phi^n + \xi,
\end{equation}
where $\xi$ is a Gaussian space-time noise, $L$ is a  self-adjoint, non-negative operator, in a general setting where renormalization is required. In the Euclidean setting, $L$ is some diffusion operator, usually the Laplacian $-\Delta$, and $\xi$ is space-time white noise. Both of these can be defined on certain metric measure spaces via the theory of Dirichlet forms pioneered by Fukushima\cite{FukushimaDF}. This paper will provide some sufficient conditions, in terms of heat kernel properties, such that a renormalized (local/global) solution theory exists for \eqref{eq:phi n formal} as well as its invariant measures on those spaces (see Theorem \ref{thm: main}). 

Of course, \eqref{eq:phi n formal} is only formal as written and giving rigorous meaning to it can be challenging. Its solution theory, even locally-in-time, is non-trivial and requires ``renormalization" when ``dimension" is bigger than or equal to 2, due to the singularity of the non-linear term $\phi^n$. Stochastic equations which require renormalization are known as singular stochastic partial differential equations (singular SPDEs).

Equation \eqref{eq:phi n formal} has connections to quantum field theory and statistical mechanics. In quantum field theory, equation \eqref{eq:phi n formal} is the Langevin dynamic of the formal Gibbs measure 
\[
\mu(d\phi) \propto \exp\Bigl(-\int_M \Bigl[\tfrac12 (\phi L \phi+\phi^2) + \frac{1}{n+1}\phi^{n+1}\Bigr] \, d\mu \Bigr) D\phi,
\]
where $``D\phi"$ is the non-existent Lebesgue measure over some infinite-dimensional function space (see \cite{PW81}). In this context, \eqref{eq:phi n formal} is known as the ``stochastic quantization equation". The rigorous construction of this Gibbs measure in dimensions 2 and 3 were crowning achievements of constructive quantum field theory in the 1970s, and they still remain the source of fruitful theoretical exploration. For this reason, the study of the solution of \eqref{eq:phi n formal} and its invariant measure remains an active topic in low dimensional Euclidean spaces\cite{bauerschmidt2024log,GH21,HairerSteele21,MW17,SZZ23,CLTphi4,TW18}. In recent years, stochastic quantization has been used to study ergodic problems involving the defocusing Schrodinger equation\cite{bringmannStaffilani}. The invariant measure has also been identified as the semiclassical limit of Gibbs states of some many-body quantum systems (see \cite{nam2025phi43theorymanybodyquantum} and the references therein). In statistical mechanics, the solution to \eqref{eq:phi n formal} in the case $n=3$ has also been identified as scaling limits of finite range interaction Ising models at critical temperature \cite{IsingKac1D,isingkacLine,IsingKacConjecture,IsingKac2D,IsingKac3D,ADC}. This is believed to be one manifestation of a not fully explored universality class. In infinite volume $d\in\set{2,3}$, the fact that $\Phi^4$ measures exist only for sufficiently small coupling corresponds to the uniqueness of equilibrium Gibbs measures of the nearest neighbor Ising model only at sufficiently high temperature\cite{simon1974p}.

Calculations which ``treat dimension as a continuum'' \cite{FW72,Wilson1973QuantumFM} in the physics renormalization group literature predict phase transitions not visible in integer dimensions. Recently, this has motivated the stochastic quantization community to study a version of \eqref{eq:phi n formal} where $L$ is a fractional Laplacian on Euclidean spaces \cite{BCCH20,DGRfracPhi4,chandra2025non}. It is believed that since the underlying space operator is no longer local, this equation and its invariant measure belong to the universality class of certain long-range Ising models with infinite interaction radius on lattices.

This paper seeks to  start a program to realize these phenomena directly on spaces with non-integer dimensions using stochastic quantization. Motivated by how physics predicts spectral dimension to govern renormalization \cite{FW72,Wilson1973QuantumFM,ADRfracphoton,QFTonCRT}, we look to the setting of metric measure spaces where spectral dimension enters naturally via short time heat kernel behavior and strongly local Dirichlet forms. This naturally includes many rough spaces, such as fractals studied by Barlow and Kigami\cite{barlow2003heat,kigami2001analysis}. 
Coincidentally, independent motivations from physics has led to rich theoretical explorations of constructing quantum field theories on fractal spacetimes\cite{CraneSmolin,QFTonCRT}. With the exception of \cite{Eyink1989QuantumFM}, who worked using hierarchical models in low spectral dimension, there is no non-perturbative literature in this direction. Attempts to study the Ising model on fractal-like graphs have also been made, mostly in the physics literature\cite{IsingFractalPhysics}. These works suggest a broader program of understanding interacting field theories and statistical mechanics on irregular geometries is yet to be developed, and the present work provides a first step in this direction via stochastic quantization.

Since our central objects of study are  singular SPDEs, we give a brief overview of the literature in this area. On the $d-$dimensional torus $\T^d$ where $L=-\Delta$, \eqref{eq:phi n formal} requires renormalization in dimensions $d\geq 2$ for any nontrivial solution theory. It is regarded as one of the simplest singular SPDEs, as it was the focus of \cite{DPD} which provided a probabilistically strong solution theory $d=2$. The field experienced breakthroughs with the advent of regularity structures\cite{Hai14} and Paracontrolled Calculus\cite{GIP15}, which provided a rigorous solution theory for  \eqref{eq:phi n formal} in $d=3$ for $n=3$. Since then, further techniques involving renormalization group concepts \cite{duch2025flow,Polchinski,SarahAndMax} to study \eqref{eq:phi n formal} and/or its invariant measure have also emerged.

Regularity structures\cite{Hai14} is a local expansion-based theory and has been lauded for its wide applicability in locally Euclidean spaces. One of its most celebrated recent results is the local theory for the non-Abelian Stochastic Yang-Mills Higgs equations on $\T^3$\cite{YMH3d}, which is closely connected to the Yang-Mills millennium problem. In recent years, the theory has been extended to manifold domains in a series of works \cite{DDD19,MSLieGrp,HairerSingh}. Notably, all of these seem tied to the locally Euclidean structure of the space (for example, \cite{HairerSingh} relies on the concept of jet bundles). This is natural, as regularity structures is heuristically a generalization of Taylor expansion. Since such structures do not exist on generic metric spaces, we believe that re-building this theory in our setting will need new ideas that are not in the current literature. 

On the other hand, paracontrolled calculus\cite{GIP15} is based on harmonic analysis\cite{BCD11}, and is thus readily applicable in the framework of existing function space theory, as opposed to Regularity structures, which must build all of the function spaces and estimates ``by hand". For \eqref{eq:phi n formal}, global well-posedness was shown for $\T^2$ first in \cite{DPD}, then without reliance on invariant measures \cite{MW17,TW18} (see \cite{Weber_2dMoyal} for the latest development in an exotic 2-dimensional setting). Then in \cite{GH21}, paracontrolled calculus was used to give a global solution on $\R^3$. Since then, paracontrolled calculus has been the tool to generate global solutions for many more complicated equations in $\R^d$ and $\T^d$, including vector-valued and gauge-theoretic equations\cite{Bcao1,BCao2,HairerRosatiSNSE,SZZ24,SingularKEZhu2}. 

In \cite{bailleul2016heat}, many of the key results in \cite{GIP15} were reproduced in the setting of a metric measure spaces with small-time Gaussian heat kernel upper bounds. Notably, these spaces include all Riemannian and sub-Riemannian manifolds, as well as all regular (metric) graphs and their Cartesian products. A notable aspect of this work is that Besov spaces were defined only using the short time behavior of the heat kernel and all of the results did not assume (as one familiar with \cite{FollandAbstractHA} may expect) a group structure on the underlying space or the need for some kind of Bernstein lemma (\cite[Section 2.1.1]{BCD11}, see \cite[Section 2.4]{SixPerson_Heisenberg} for a nontrivial geometric context). On many rough metric measure spaces, such as the Sierpinski gasket, one can obtain the precise upper and lower heat kernel estimates. In light of this, the heat semi-group approach of \cite{bailleul2016heat} seems to be a perfect setting to study singular (S)PDEs on general spaces from the paracontrolled perspective. One informal goal of this paper is to give evidence for the possibility of re-building paracontrolled calculus on certain rough metric spaces.

While singular SPDEs have not been studied on rough metric measure spaces, the literature on SPDEs without renormalization on general metric spaces has been growing for more than a decade. Besides the abstract functional analytic formulations \cite{da2014stochastic}, the earliest results trace back to the Dirichlet form approaches in \cite{hinz2013vector}, and the semi-group approaches using heat kernel estimates in \cite{hinz2012semigroups,issoglio2015regularity}. The well-posedness and regularity of solutions to parabolic-type SPDEs on post critically finite fractals were studied in \cite{hambly2017continuous, hambly2018existence,hinz2022approximation}, as well as stochastic wave equations in \cite{hambly2020damped}. For metric graphs, we refer to \cite{cerrai2017spdes,cerrai2019fast,cerrai2021incompressible,cerrai2024spdes,fan2021stochastic}. In \cite{baudoin2025parabolic}, whose setting is close to ours, the authors studied random field solutions to the multiplicative stochastic heat equation on spaces with sub-Gaussian heat kernels. In an upcoming work, Zhenyao Sun, Louis Fan and the second author have developed a framework for SPDEs with random field solutions on general metric measure spaces and studied their comparison principles.

To the best of our knowledge, there is currently no literature on equations which require renormalization on general metric measure spaces. This article is the first step in filling that gap, using the growing theory of heat kernel based Besov spaces\cite{triebel2006theory,grigor2015heat,liu2016besov,UConn1,UConn2,alonso2021besov,UConn4}. However, the existing literature on such spaces is mostly focused on the structure of the heat flow and function spaces themselves, and does not provide the nonlinear analytic tools required for the study of singular stochastic PDE. In particular, fundamental ingredients such as Schauder-type estimates, paraproduct decompositions, and precise connections between the Dirichlet form and useful notions of Besov-type norms are largely absent in this setting.

We discovered that reconstructing this analytic structure forces us to confront difficulties arising from allowing fractal geometry. Indeed, in Euclidean-type settings, smooth functions form an algebra and Fourier analytic tools provide fine localization. Such structures are known to not exist on general metric measure spaces. As a result, the formulation of nonlinear operations becomes delicate. These difficulties are closely related to phenomena observed in the Dirichlet forms and fractal analysis literature. In particular, the space of ``smooth" test functions typically are not algebras \cite{ben1999not}, and ``derivatives" are often singular with respect to the reference measure \cite{kajino2020singularity}, leading to a lack of standard harmonic analytic tools such as Bernstein-type inequalities. For the study of \eqref{eq:phi n formal}, these also make certain PDE type arguments, such as the maximum principle arguments of \cite{GH19}, difficult to implement here.

One of the main contributions of this paper is to show that, despite these limitations, it is possible to develop a sufficiently rich analytic framework to treat renormalized stochastic quantization equations via the Da Prato–Debussche method\cite{DPD} on such spaces. This is achieved by working entirely within a heat-semigroup-based Besov framework and exploiting the available H\"older regularity of the heat kernel. From the PDE viewpoint, it is also worth mentioning that we used an energy argument to obtain the global solution to a nonlinear equation, despite all the difficulties caused by rough geometry. In this way, we obtain a clearer understanding of how rough geometry gives rise to analytic difficulties and what structures can be used to overcome them in the study of nonlinear equations. We believe these difficulties will also prohibit a general extension of our results to Sobolev settings similar to the type discussed \cite[Appendix B]{bailleul2016heat}. For an exploration of the bad behavior of Sobolev spaces in a particular rough geometry when studying a PDE, we refer the reader to \cite{SGNLS}, which studied the nonlinear Schrodinger equation on the Sierpinski gasket.

\subsection*{Acknowledgments} 
We are grateful to SLmath (formerly MSRI) for hosting the summer school on stochastic quantization in July 2024, during which we learned most of the singular SPDE and mathematical physics material which made us believe this project is possible. We thank the summer school instructors Hao Shen, Massimiliano Gubinelli, and Lorenzo Zambotti as well as the teaching assistants Sky Cao, Sarah Jean-Meyer, and Evan Sorensen for their carefully crafted lessons and patience for our questions during the summer school. We also thank Ajay Chandra, Pawel Duch, Martin Hairer, Nicholas Perkowski,  Harprit Singh and Rongchan Zhu for helpful discussions on singular SPDEs and mathematical physics and comments on earlier versions of this project, as well as Fabrice Baudoin and Li Chen for discussions on Besov theory for spaces with sub-Gaussian heat kernels. H.C. is supported by the research grant (VIL73729) from Villum Fonden.
\subsection{Assumptions and Main Results}\label{sec: Setup}

Throughout the paper $(M,d)$ denotes a locally compact complete metric space
and $\mu$ a Radon measure on $M$ with full support. We refer to the triple
$(M,d,\mu)$ as a \emph{metric measure space}. Its diameter is
\[
\mathrm{diam}(M):=\max_{x,y\in M} d(x,y).
\]
For simplicity, and in order to avoid weighted function
spaces, we assume throughout that $(M,d)$ is \textbf{compact}. Without loss
of generality we also take $\mathrm{diam}(M)\ge 1$ and $\mu(M)\ge 1$.

We write $\mathcal C_M$ for the space of continuous functions on $M$ and, for
$p\ge 1$, we abbreviate $L^p_M:=L^p(M,\mu)$. For functions $f,g:M\to\R$, we
write $f\lesssim g$ (equivalently $f(x)\lesssim g(x)$ uniformly in $x\in M$)
if there is a constant $C>0$ such that $f(x)\le Cg(x)$ for every $x\in M$,
and $f\asymp g$ if $f\lesssim g$ and $g\lesssim f$. When used repeatedly the
implicit constants may change from line to line.  For $t,s \in \R$, we denote $t \wedge s:= \min\{t,s\}$ and $t\vee s:=\max\{t,s\}$. Finally we set
$\Z_+:=\{0,1,2,\dots\}$ and $\N:=\Z_+\setminus\{0\}$.

For~\eqref{eq:phi n formal} to formally make sense, we assume that $L$ is a
non-negative, densely defined, self-adjoint operator on $L^2_M$ with domain
$\mathcal D_2(L)$, and that $L$ generates a Markov semigroup
$(P_t)_{t\ge 0}$, written interchangeably as $e^{-tL}=P_t$. We further assume
that $P_t$ is absolutely continuous with respect to $\mu$ for every $t>0$,
so that there exists a symmetric, everywhere-defined continuous function
$p_t:M\times M\to\R_+$ satisfying
\[
P_tf(x)=\int_M p_t(x,y)f(y)\,\mu(\mathrm dy)
\]
for every non-negative measurable $f$ and every $x\in M$ for which the
integral converges. We call $p_t(x,y)$ the \emph{heat
kernel} of $L$.

By the classical theory of Dirichlet
forms (See \cite[Theorem 1.3.1]{FukushimaDF}, also see \cite{chen2012symmetric,ma2012introduction}), $L$ is
associated with a symmetric quadratic form: for $f,g \in \mathcal D_2(L)$
\begin{align}
    \mathcal E(f,g):=\int_M f(Lg)\,\mathrm d\mu,
\end{align}
whose domain $\mathcal F\supset \mathcal D_2(L)$ makes $(\mathcal E,\mathcal
F)$ a Dirichlet form on $L^2_M$. More details will be discussed in Section~\ref{sec:prelim}.

%\paragraph{The main assumption and consequences}
We assume the following for the rest of the paper.

\begin{assumption}\label{A:main}
With $(M,d,\mu)$, $L$, $(P_t)_{t\ge 0}$ and $(p_t(x,y),\ t>0,\ x,y\in M)$ as
above, there exist constants $d_h\ge d_w\ge 2$ and $\Theta\in(0,1]$ such that
the following hold.
\begin{itemize}
  \item \label{cond:sgu} \emph{Sub-Gaussian heat-kernel bounds.} Set
  $\Phi(u):=\exp\!\bigl(-u^{d_w/(d_w-1)}\bigr)$ for $u\in[0,\infty)$. There
  are $C_1,C_2>0$ such that, uniformly in $(t,x,y)\in[0,1]\times M^2$,
    \begin{align}\label{ineq:sgu}
      t^{-\frac{d_h}{d_w}}\,\Phi\!\left(C_1\frac{d(x,y)}{t^{1/d_w}}\right)
      \;\lesssim\; p_t(x,y) \;\lesssim\;
      t^{-\frac{d_h}{d_w}}\,\Phi\!\left(C_2\frac{d(x,y)}{t^{1/d_w}}\right).
      \tag{$\mathrm{HKE_f}$}
    \end{align}
  \item \emph{Spatial H\"older regularity.} Uniformly in $t\in(0,1]$ and
  $x,x',y\in M$ with $d(x,x')\le t^{1/d_w}$,
    \begin{align}\label{ineq:holder_kernel}
      \bigl|p_t(x,y)-p_t(x',y)\bigr|
      \;\lesssim\;
      \pr{\frac{d(x,x')}{t^{1/d_w}}}^{\!\Theta}
      t^{-\frac{d_h}{d_w}}\,\Phi\!\pr{C\frac{d(x,y)}{t^{1/d_w}}}.
      \tag{$\mathrm{HC}_\Theta$}
    \end{align}
  \item \emph{Stochastic completeness.} For every $t>0$ and $x\in M$,
    \begin{align}\label{eq:s_complete}
      \int_M p_t(x,y)\,\mu(\mathrm dy)=1.\tag{SC}
    \end{align}
\end{itemize}
\end{assumption}

\noindent Examples and implications of Assumption~\ref{A:main} are numerous
and are deferred to Section~\ref{sec:example and implications}.

Assumption~\ref{A:main} implies the \emph{chain condition} on the metric
space $(M,d)$ (see~\cite{murugan2020length}); as a result any $\alpha$-H\"older
continuous function with $\alpha>1$ must be constant. By~\cite[Theorem
4.1]{grigor2008dichotomy}, it also entails volume regularity: uniformly in
$x\in M$ and $r\in(0,\mathrm{diam}(M)]$,
    \begin{align}\label{ineq:a_reg}
      \mu(B(x,r))\asymp r^{d_h}.\tag{$V_{h}$}
    \end{align}
In particular $d_h$ is the Hausdorff dimension of $(M,d,\mu)$. Moreover, the
semigroup $(P_t)_{t\ge 0}$ satisfies the Feller and strong Feller properties
(see~\cite{lierl2015scale}): it is strongly continuous on $\mathcal C_M$,
and $P_tf\in\mathcal C_M$ for every $t>0$ and every bounded Borel measurable
$f$ on $M$.

Combining \eqref{ineq:a_reg} with Lemma~\ref{lem:long time heat kernel
bounds}, the heat-kernel estimates~\eqref{ineq:sgu} can be recast in the
more compact volume-based form: there exist $C_1,C_2>0$ such that, uniformly
in $x,y\in M$ and $t>0$,
\begin{align}\label{ineq:hk_subg_compact_form}
    \frac{1}{\mu\!\left(B(x,t^{1/d_w})\right)}
      \Phi\!\pr{C_1\frac{d(x,y)}{t^{1/d_w}}}
    \;\lesssim\; p_t(x,y) \;\lesssim\;
    \frac{1}{\mu\!\left(B(x,t^{1/d_w})\right)}
      \Phi\!\pr{C_2\frac{d(x,y)}{t^{1/d_w}}}.
\end{align}

%\paragraph{Test functions, distributions and Besov spaces.}
We now turn to the function spaces used to state our main results. Write
$\norm{\cdot}_{\mathcal C_M}$ for the supremum norm on $\mathcal{C}_M$ and let $\mathcal D_\infty(A)$ be the domain of an operator $A$ in $\mathcal{C}_M$.

\begin{definition}\label{def:test_functions}
The space of \emph{test functions} is
\[
  \mathcal S
  := \bigl\{f\in\mathcal C_M:\ f\in \mathcal D_\infty(L^n)\ \text{for every }
  n\in\Z_+\bigr\}.
\]
The space of \emph{distributions} $\mathcal S'$ is the topological dual of $\mathcal S$.
\end{definition}

\noindent The semigroup-based building blocks of our Littlewood--Paley
decomposition are
\begin{equation}\label{eq:pq operators}
    \begin{aligned}
    Q^{(k)}_t &:= \left(tL\right)^{k}e^{-tL},
      \qquad k\in\Z_+,\\
    P_t^{(k)}&:= \sum_{m=0}^{k-1}\frac{1}{m!}\,L^m e^{-Lt},
      \qquad k\in\N.
    \end{aligned}
\end{equation}

\noindent We now introduce the Besov spaces that will serve as solution
spaces for~\eqref{eq:phi n formal} and the natural setting for
paraproducts. Our definition follows~\cite{liu2016besov} (see
also~\cite{grigor2015heat}).

\begin{definition}\label{def:besov_space}
For $\alpha\in\R$, $p,q\in(0,\infty]$ and $k\in\Z_+$ with $k>\alpha/d_w$, set
\begin{align}\label{eq:besov norm def pq}
      \norm{f}_{\dot{\mathcal B}_{p,q}^\alpha}
  := \left(\int_0^1\!\left[t^{-\alpha/d_w}\norm{Q^{(k)}_t f}_{L^p_M}
  \right]^q\frac{\mathrm dt}{t}\right)^{\!1/q},
\qquad
  \norm{f}_{\mathcal B^\alpha_{p,q}}
  := \norm{P_1 f}_{L^p_M} + \norm{f}_{\dot{\mathcal B}_{p,q}^\alpha},
\end{align}
for $f\in\mathcal S$, with the usual modification when $q=\infty$. The
Besov space $\mathcal B^\alpha_{p,q}$ is the closure of $\mathcal S$ under
$\norm{\cdot}_{\mathcal B^{\alpha}_{p,q}}$. When $q=\infty$ we abbreviate
$\mathcal B^{\alpha}_p:=\mathcal B^\alpha_{p,\infty}$  for $\alpha\in\R$,
$p\in[1,\infty]$.
\end{definition}

\noindent The norms in~\eqref{eq:besov norm def pq} are equivalent for all
sufficiently large $k\in\Z_+$; see~\cite[Proposition 4.4]{bui2012calderon}.

Let $\dot{\xi}$ be space–time white noise on $\R\times M$. Motivated by \cite{DPD}, we decompose \eqref{eq:phi n formal} into the stochastic linear part 
\begin{align}\label{eq:she mass}
  \partial_t Y = -LY  - Y +\dot{\xi},\quad  \text{on }\R_+\times M
\end{align}
and the smoother remainder equation
\begin{equation}\label{eq:formal eq for v intro}
\begin{aligned}
  \partial_t v = -L v- v + \Psi(Y,v)\quad \text{ on } \R_+\times M,
  \end{aligned}
\end{equation}
with $\Psi(Y,v):=-\sum_{j = 0}^n \binom{n}{j}v^j Y^{:n-j:}$, where  $Y^{:k:}$ is the $k$-Wick power of $Y$ for $k \in \N$ (see Section \ref{sec: EW and Wick} for more details on the re-normalizing counter terms). 

Given the tuple $(d_h,d_w,\Theta,n)$ as in Assumption \ref{A:main}, we say equation \eqref{eq:phi n formal} is in the \textbf{Da Prato-Debussche (DPD)} regime if for each $\phi_0$ in the natural Besov space of $Y$ in \eqref{eq:she mass} (see Theorem \ref{thm:Wick Power Convergence}), there exist appropriate $(Y_0,v_0)$ with $Y_0+v_0 = \phi_0$, so that \eqref{eq:she mass} has a solution global in $t \ge 0$ and \eqref{eq:formal eq for v intro} has a local (in time) solution with initial conditions $Y_0$ and $v_0$ respectively.

\noindent Our main result  gives a sufficient condition for the DPD regime with the convention that $a/0 = \infty$ for $a>0$. For the precise statement, see Lemma \ref{lem:local solution to DPD and convergence of approximation} and Corollary \ref{cor:global_wp}.
\begin{theorem}\label{thm: main}
    The tuple $(d_h,d_w,\Theta,n)$ is in the DPD regime when 
    \begin{align}\label{ineq: dbd condition}
        n < \frac{1}{2} \frac{d_h+d_w}{d_h - d_w},\qquad \text{ and } \qquad n < \frac{2\Theta}{d_h - d_w} + 1.
    \end{align}
    If furthermore $n \ge 3$ is odd, and 
    \begin{align}\label{ineq:global solution condition}
        n < \frac{d_w}{d_h - d_w}.
    \end{align}
    Then for appropriate initial condition, the remainder equation \eqref{eq:formal eq for v intro} admits a function-valued global in time solution,  and \eqref{eq:phi n formal} defines a  time-homogeneous Markov process on $\mathcal B^{-\alpha}_{\infty}$, for any $\alpha>\frac{d_h-d_w}{2}$. In addition, it admits at least one invariant probability measure. 
\end{theorem}
\begin{remark}
    The first inequality in \eqref{ineq: dbd condition} comes from the Wick power of $Y$ defined as a function of time rather than space-time distribution. The second inequality in \eqref{ineq: dbd condition} is due to the restriction for product estimates (see Theorem \ref{thm:besov_product}). Inequality \eqref{ineq:global solution condition} is due to the need to control the Besov norm of the solution of \eqref{eq:formal eq for v intro} in terms of the Dirichlet energy (see Proposition \ref{prop:CDI_mollified} and Proposition \ref{lemma: Chain rule workaround}). This should be understood as the full regime of the energy argument of \cite{MW17,TW18}.
\end{remark}

\begin{corollary}\label{cor: phi4 DPD}
The $\Phi^4$ equation is in the DPD regime whenever
\begin{align*}
  \frac{d_h+d_w}{d_h -d_w} > 6,\quad \text{ and } \quad \frac{\Theta}{d_h - d_w}> 1.
\end{align*}
\end{corollary}

\noindent The appearance  of the H\"older regularity $\Theta$ of heat kernel in \eqref{ineq: dbd condition} is due to the following product estimates. See Section \ref{sec:pp_estimates} for details.
\begin{thm}[Multiplicative inequality]\label{thm:besov_product}
There exists a decomposition of product $f \cdot g$ that agrees with the point-wise multiplication when $f,g \in \mathcal C_M$, so that 
$\alpha>0$,  $\beta \in \left(-\Theta, \Theta\right)\backslash\{0\}$ and  $\alpha+ \beta>0$
\begin{align*}
  \norm{f\cdot g}_{\mathcal B^{\alpha \wedge \beta}_{\infty}} \lesssim \norm{f}_{\mathcal B^\alpha_{\infty}} \cdot \norm{g}_{\mathcal B^\beta_{\infty}},
\end{align*}
 uniformly in $f \in \mathcal  B^\alpha_{\infty},\, g \in \mathcal B^\beta_{\infty}$
\end{thm}
\begin{remark}[Spectral dimension and admissible parameters]\label{remark:spec dim and theta}
The first condition in Theorem~\ref{thm: main} admits a reformulation purely in terms of the spectral dimension
\[
d_s = \frac{2d_h}{d_w}.
\]
A direct computation yields
\[
\frac{1}{2}\frac{d_h + d_w}{d_h - d_w}
=
\frac{1}{2}\frac{d_s+2}{d_s-2},
\]
so that the corresponding restriction on $n$ depends only on $d_s$. This is consistent with physics heuristics, where the spectral dimension governs the divergence of the linear stochastic convolution and therefore determines the renormalizability threshold.

In contrast, the second condition in Theorem~\ref{thm: main},
\[
n < \frac{2\Theta}{d_h - d_w} + 1,
\]
cannot, in general, be expressed solely in terms of the spectral dimension, as it depends explicitly on the H\"older regularity exponent $\Theta$ of the heat kernel. This shows that the Da Prato--Debussche regime is not determined purely by spectral dimension, but also by finer regularity properties of the underlying space.

Such a phenomenon does not arise in locally Euclidean Laplacian or fractional-Laplacian settings. For variable coefficient operators, similar in spirit restrictions have appeared, particularly in connection with homogenisation(see\cite[Definition 2.1, Proposition 2.9]{chen2025operatordependentparacontrolledcalculus}, see also\cite{singh2025canonical,hairer2025homogenisationsingularspdes}). However, we will see that the restriction in this setting is an effect of the rough local geometry.
\end{remark}
\subsection{Interpretations and examples}\label{sec:example and implications}

\textit{Scope of the assumption.}
Assumption~\ref{A:main} is standard in the analysis of Dirichlet spaces and
covers many familiar examples---tori, compact Riemannian manifolds, metric
graphs, and a large class of fractals. Recall that $d_w$ is the walk
dimension, which encodes the space--time scaling of the associated diffusion,
while $\Theta$ quantifies the spatial H\"older regularity of the heat kernel.
The class of admissible geometries is in fact quite broad: for every pair
$(d_h,d_w)$ satisfying
\[
2\le d_w\le 1+d_h,
\]
there exists a compact geodesic metric measure space, equipped with a
Dirichlet form, whose heat kernel obeys the sub-Gaussian
estimate~\eqref{ineq:sgu} (c.f. \cite{murugan2024diffusions}, \cite{murugan2024heat}, \cite{barlow2003heat}).

\textit{Spectral dimension and the singular regime.}
A basic consequence of~\eqref{ineq:sgu} is the short-time on-diagonal estimate
\[
p_t(x,x)\asymp t^{-d_h/d_w},
\]
which connects naturally to the spectral dimension
\[
d_s:=\frac{2d_h}{d_w}.
\]
The spectral dimension governs the regularity of the solution of \eqref{eq:she mass}: by the classic theory of Walsh\cite{Walsh86}, a random field solution of \eqref{eq:she mass} must satisfy
\[
\mathbb E[Y(t,x)^2]= \int_0^t e^{-s}p_s(x,x)\,ds<\infty
\qquad\Longleftrightarrow\qquad d_s<2.
\]
The value $d_s=2$ thus marks the threshold between the random field and the singular regimes, and the assumption $d_h\ge d_w$ is equivalent to the singular regime  $d_s\ge 2$, where renormalization is unavoidable. In the case $d_s<2$, the conclusions of the main results will hold except that \eqref{eq:phi n formal} admits a random field solution in the sense of Walsh \cite{Walsh86} and has no renormalization conditions, and it would almost surely belong to a space of positive Besov spatial regularity.

\textit{Spatial regularity heuristics.}
One notable aspect of Theorem \ref{thm: main} is that in addition to the dimensional quantities $d_h$ and $d_w$, the maximal spatial H\"older regularity $\Theta$ of the heat kernel must also be considered. As noted before, this is due to the appearance of $\Theta$ in Theorem \ref{thm:besov_product}, but notably, this is absent in the Fourier analytic formulation of the paraproduct decomposition in $\R^d$ and $\T^d$(c.f. \cite[Chapter 2]{BCD11}, see also \cite[Proposition 3.11]{baudoin2025parabolic}). Also worth pointint out is that, \cite[Proposition 3.3]{bailleul2016heat} has a similar restriction, but with $\Theta=1$ (which they had assumed). There is evidence that this restriction is not an accident. First, unlike $\mathcal C^\infty(\T^d)$ or $\mathcal C_c^\infty(\R^d)$, our test function space $\mathcal{S}$ needs not be an algebra: in fact, we know it cannot be if $d_w>2$\cite{ben1999not}. Thus, the way of defining arbitrary products of distributions with any test function fails in this setting. However, H\"older spaces are algebras, and it would make sense that we can define products of the H\"older space elements and spaces contained in their dual spaces. The fact that our Besov spaces only agree with H\"older spaces up to $\Theta$(Proposition \ref{prop:holder=besov}) is also reminiscent of recent triviality results for other notions of Besov spaces \cite{baudoin2024korevaarschoensobolevspacescriticalexponents,Kumagai_2025}.

The parameters not entirely independent: by~\cite[Proof of Lemma
4.6]{barlow2012equivalence} and~\cite[Lemma 3.4]{alonso2021besov}, the
sub-Gaussian bound~\eqref{ineq:sgu} already implies~\eqref{ineq:holder_kernel}
for \emph{some} exponent $\Theta\in(0,1]$. What is generally \emph{not} known
is the optimal value of $\Theta$. This becomes decisive for identifying concrete Da Prato--Debussche regimes. For this reason,
although the lower bound in~\eqref{ineq:sgu} is not essential for most of our
analytic arguments, we retain the full two-sided estimate in our assumption. In addition, it keeps the link with the spectral dimension transparent, and for many examples the necessary and sufficient conditions for such two-sided bounds are now well understood.

\textit{Quantitative estimates for $\Theta$.}
In certain cases $\Theta$ can be controlled explicitly in terms of
$(d_h,d_w)$. It is known that spaces with heat kernels satisfying \eqref{ineq:sgu} but with $d_h<d_w$, we have
\[
\Theta\ge d_w-d_h \qquad \text{whenever } 1 \le d_h<d_w\le d_h+1,
\]
with equality (in other words, optimality) if $M$ is a post-critically finite self-similar
fractals~\cite{barlow2006diffusions,alonso2021besov}. In the complementary
regime $d_w\le d_h$, no general lower bound on $\Theta$ is available. This is why the second inequality in Theorem~\ref{thm: main} is
substantive: even when the spectral-dimension restriction is favourable,
insufficient heat-kernel regularity may still rule out the Da
Prato--Debussche argument.

\textit{A classical benchmark.}
Suppose that the heat kernel satisfies two-sided Gaussian bounds, i.e.\
$d_w=2$, and is spatially Lipschitz, i.e.\ $\Theta=1$. Corollary~\ref{cor:
phi4 DPD} then places $\Phi^4$ in the Da Prato--Debussche regime if
\[
d_h<\frac{14}{5},
\]
in agreement with the fractional-dimensional heuristic
of~\cite[Section~2.8.2]{BCCH20}. Whether one can actually \emph{realize}
$d_w=2$, $\Theta=1$, and an arbitrary Hausdorff dimension $d_h$
simultaneously appears to be open:
Laakso's construction\cite{laakso2000ahlfors} yields spaces with arbitrary $d_h$ and $d_w=2$, but no quantitative control of $\Theta$ is known there.

\textit{Construction of Examples.}
Figure~\ref{fig:phi4-parameter-regions} presents two complementary views of
Corollary~\ref{cor: phi4 DPD}; the two panels should be read differently.

The \emph{left} panel is built from a product-space construction. Starting
from a space $X$ for which $\Theta\ge d_w-d_h$, $d_h(X)<d_w(X)$, one considers the product
\[
M=X\times X.
\]
Under this operation the walk dimension is preserved while the Hausdorff
dimension doubles,
\[
d_w(M)=d_w(X),\qquad d_h(M)=2d_h(X),
\]
and the H\"older exponent descends from the single factor\cite[Theorems 3.7 and 3.8]{alonso2021besov}. The left panel
therefore records precisely those product constructions $M=X\times X$ for
which the minimal H\"older regularity possible still allows $M$ to be in the $\Phi^4$ Da Prato--Debussche regime; it is \emph{not}
a plot for arbitrary metric measure spaces $M$. In this picture the Vicsek
product lies inside the admissible region, while the product of two
Sierpi\'nski gaskets lies outside. Notably,  for spaces $M = X\times X$, the second inequality in \eqref{ineq: dbd condition} coincides with \eqref{ineq:global solution condition}. In other words, they 
also all admit global solutions (and thus invariant measures).  There may be a deeper reason behind this, but we shall not explore it further in this project.

The \emph{right} panel considers a different scenaio. If one sets aside the
difficulty that $\Theta$ is typically unknown and simply imposes the
benchmark assumption $\Theta=1$, then Corollary~\ref{cor: phi4 DPD} reduces,
for $\Phi^4$, to
\[
d_h<\tfrac75\,d_w,
\qquad
d_h<d_w+1.
\]
The right panel is therefore a benchmark picture: the maximal region
predicted by our criterion under the strongest realistic assumption on the
spatial regularity of the heat kernel.

\begin{figure}[H]
    \centering
    \includegraphics[width=1.\linewidth]{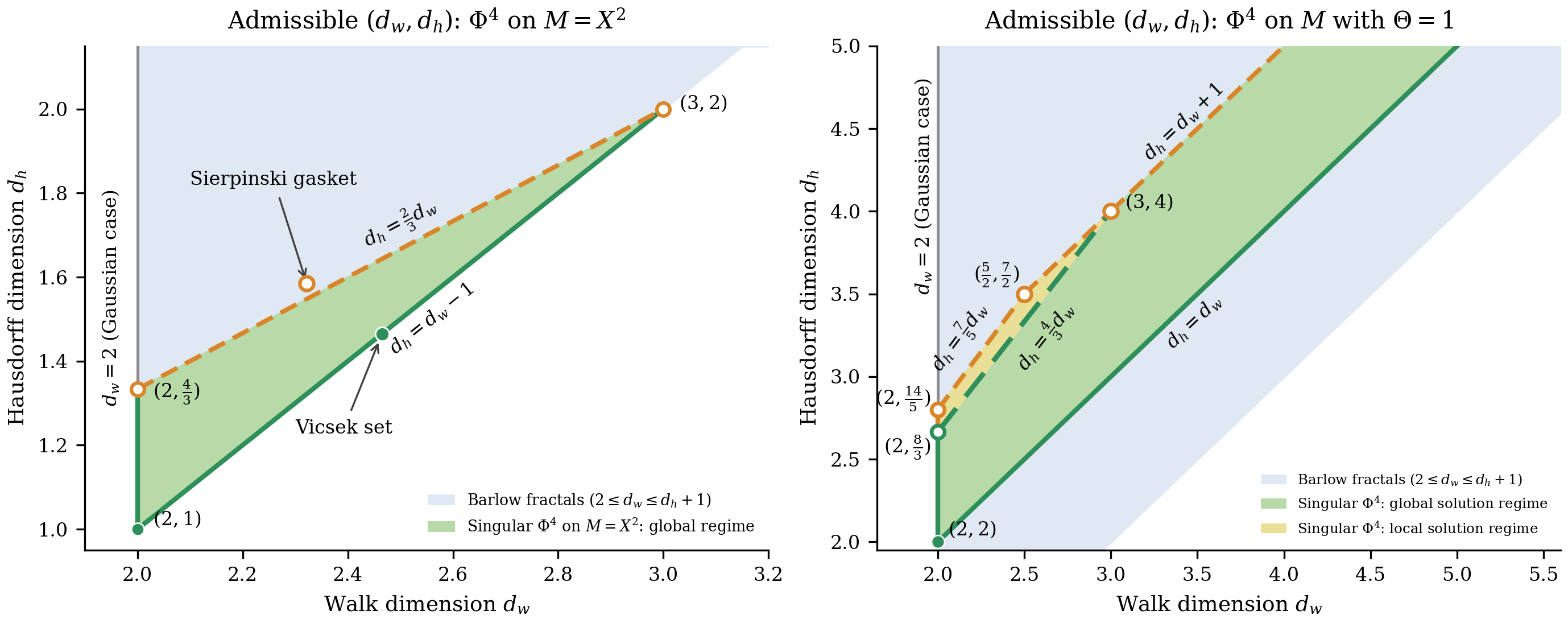}
    \hfill
    \caption{Admissible parameter regimes in the $(d_w,d_h)$-plane. Left: the product construction $M=X^2$, starting from a one-factor Barlow fractals $X$ with $d_h(X)<d_w(X)$ and quantitatively known heat-kernel H\"older exponent. The pale-blue background is the Barlow fractal range $2\le d_w\le d_h+1$, while the green region is the singular $\Phi^4$ global-solution regime on $M=X^2$.  Right: the benchmark case $\Theta=1$, corresponding to Lipschitz spatial regularity of the heat kernel. The green region is the global-solution regime, whereas the yellow region satisfies the local theory but lies outside the global condition.}
    \label{fig:phi4-parameter-regions}
\end{figure}
\noindent 

\begin{remark}
    The number $\frac{8}{3}$ appearing in both figures should have a deeper meaning, as it is known that in the Gaussian case $(d_w = 2)$, the invariant measure of the $\Phi^4_{d_h}$ equation(interpreted using $(-\Delta)^{-\alpha}$ on $\T^2$) is absolutely continuous with respect to the Gaussian free field for $d_h <\frac{8}{3}$ \cite[Corollary 4.5]{hairer2026singularitysolutionssingularspdes}. We conjecture that adapting the content of \cite{hairer2026singularitysolutionssingularspdes}, one can show that the invariant measures of \eqref{eq:phi n formal} constructed in this paper are all absolutely continuous with respect to the Gaussian free field, defined in \eqref{def:GFF measure}.
\end{remark}

\textit{Two fractal examples.}
We list two well-known examples from the above to show that the $\Theta$-constraint is relevant: it can separate solvable from non-solvable regimes even when the scaling condition alone points in the same direction.

\begin{example}[A product space in the $\Phi^4$ Da Prato--Debussche regime]
Let $V$ be the finite Vicsek fractal. By~\cite[Theorem
3.7]{alonso2021besov}, $V$ satisfies Assumption~\ref{A:main} with
\[
d_h'=\frac{\ln 5}{\ln 3},\qquad d_w=\frac{\ln 5}{\ln 3}+1,\qquad \Theta=1.
\]
Equip $M:=V\times V$ with the product metric and product measure. Then
Assumption~\ref{A:main} holds on $M$ with the same $d_w$ and $\Theta$ and
with $d_h=2d_h'$. A direct computation gives
\[
\frac{d_h+d_w}{d_h-d_w}\approx 11.6>6,
\qquad
\frac{\Theta}{d_h-d_w}\approx 2.15>2,
\]
so Corollary~\ref{cor: phi4 DPD} applies: $\Phi^4$ on $M$ lies in the Da
Prato--Debussche regime.
\end{example}

\begin{example}[A space excluded by the $\Theta$-condition]
Let $M=T\times T$, where $T$ is the Sierpi\'nski gasket. Again
by~\cite[Theorem 3.7]{alonso2021besov},
\[
d_h=2\,\frac{\ln 3}{\ln 2}\approx 3.170,
\qquad
d_w=\frac{\ln 5}{\ln 2}\approx 2.322,
\qquad
\Theta=d_w-\frac{d_h}{2}\approx 0.737.
\]
Then
\[
\frac{d_h+d_w}{d_h-d_w}\approx \frac{5.492}{0.848}>6,
\qquad
\frac{\Theta}{d_h-d_w}\approx 0.87<1.
\]
The scaling condition is satisfied, but the H\"older-regularity condition
fails, and $\Phi^4$ on $M$ is not covered by Corollary~\ref{cor: phi4 DPD}.
\end{example}

The one-factor spaces underlying the two examples are depicted in Figure~\ref{fig:fractal-models}.

\begin{figure}[H]
    \centering
    \includegraphics[width=0.4\linewidth]{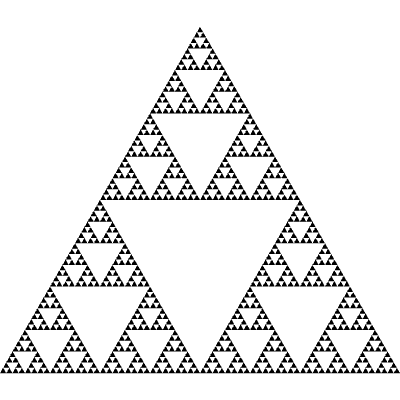}
    \hfill
    \includegraphics[width=0.4\linewidth]{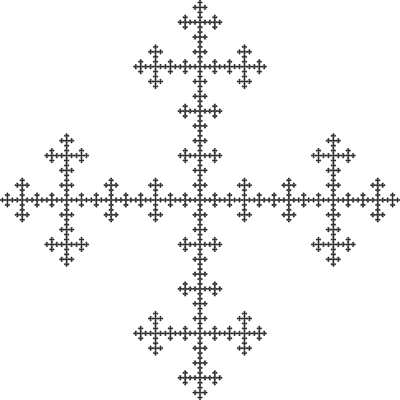}
    \caption{Left: the Sierpi\'nski gasket. Right: the Vicsek fractal.
    These are the one-factor spaces used in the product examples above.}
    \label{fig:fractal-models}
\end{figure}

\textit{Future directions.}
Below are some interesting potential future directions motivated by the existing literature and the current work.
\begin{enumerate}
    \item Here, we do not obtain the uniqueness of the invariant measure. Uniqueness typically requires additional mixing properties (for instance irreducibility and an \emph{(asymptotic) strong Feller} property), which are not proved here. We conjecture that one should be able to use the argument of \cite{TW18} to produce a spectral gap for $\mathbf{P}_t$, the semigroup on probability measures on the relevant Besov space defined by the solution. We also conjecture that the argument of \cite{chandra2025non} or \cite{HairerSteele21} can be adapted to show that the unique invariant measure is non-Gaussian.
    \item The discrepancy between the local and global solution regimes in our setting is possibly due to the lack of sharpness of the energy method we adapted. More precisely, a careful examination of Section \ref{sec:global sol} reveals that this gap is due to the need to control a higher regularity Besov norm than provided by the local solution condition (except in the case $d_h=d_w$, which generalizes $\T^2$ exactly). It is possible that a more optimal exploitation of the coercive nature of the backward superlinear forcing will close this gap, as has been done in Euclidean settings for $\Phi^4$ \cite{MW17c,GH19,CG25}. This idea could also be explored in combination with paracontrolled calculus.
    \item It is also natural to wonder if the machinery we built up here can be extended to other singular SPDEs. For instance, \cite{bailleul2016heat} originally addressed some parabolic Anderson models, and \cite{JP23} has introduced a inverse Cole-Hopf transform which removed the need to use paracontrolled calculus for the $\Phi^4$ equation. The latter strategy has been generalized to Hermitian vector bundles on 3-manifolds in \cite{BDFT23}. It is also known that the dynamical sine-Gordon equation has a Da Prato-Debussche regime on $\T^2$ \cite{HS16,sineGordon24}, and one may expect to find the same in the analogous generalization in our setting: $d_h=d_w$. However, the literature of all of these equations indicate that the analytic difficulties we encounter here manifest as different problems for each equation.
    \item More intricate sub-Gaussian heat-kernel estimates arise on more irregular  geometries, such as continuum random trees~\cite{CroydonCRT} and Liouville quantum gravity surfaces~\cite{LQGHeatKerAIHP,AndresKajinoLQG,LQGDingZeitouniZhang}. Given the interest of quantum field theory in some of these spaces~\cite{QFTonCRT}, deriving an equation of the form~\eqref{eq:phi n formal} and extending the results of the present paper to this setting would constitute a genuinely novel non-perturbative approach in these settings.
\end{enumerate}

\subsection{ Paper Organization}

We first recollect some basic Dirichlet form preliminaries in Section \ref{sec:prelim} and prove some elementary functional inequalities involving the heat-kernel based Besov norms in Section \ref{sec:besov}.
In Section \ref{sec:pp_estimates}, we construct the paraproduct and prove 
%Proposition \ref{prop: schauder} and 
Theorem \ref{thm:besov_product} and a Schauder-type estimate (Proposition \ref{prop: schauder}) using the strategy of \cite{bailleul2016heat}. In section \ref{sec: EW and Wick}, we construct Wick powers of $Y$ and compute their regularities. Then in section \ref{sec:Solution Theory}, we will develop local and global solutions for $v$ using the methods of \cite{DPD,MW17,TW18}. Finally, we employ the machinery developed in \cite{chandra2025non} to view the solution $\phi$ as a time-homogeneous Markov process on a proper state space $\mathfrak C$, to which we can apply the Krylov-Bogliubov method and obtain an invariant measure.

\section{Preliminaries}\label{sec:prelim}

This section collects the basic facts about Dirichlet forms, the associated semigroup and generator, and the energy measure that will be used freely throughout the paper, and then records several preliminary estimates. The material is standard but scattered across the literature; for the convenience of the reader, proofs that are not conveniently available in the literature are collected in Appendix~\ref{sec:appendix}. We claim no novelty for the results in this section. For a systematic treatment of Dirichlet forms we refer to \cite{FukushimaDF,chen2012symmetric,ma2012introduction}, and for the general semigroup theory to \cite{yosida2012functional}.

Let $(\mathcal E, \mathcal F)$ be a densely defined, closed, non-negative definite symmetric bilinear form on $L^2_M$, where $\mathcal F := \mathcal D(\mathcal E)$ is a Hilbert space under the inner product
\begin{align*}
    \mathcal E_1(u,v) := \mathcal E(u,v) + \langle u,v\rangle_{L^2_M}.
\end{align*}
A map $\phi : \R \to \R$ is a \emph{normal contraction} if $\phi(0) = 0$ and $\abs{\phi(s) - \phi(t)} \le \abs{s-t}$ for all $s,t \in \R$. The form $(\mathcal E, \mathcal F)$ is called a \emph{Dirichlet form} if it is \emph{Markovian}, i.e., $\phi \circ u \in \mathcal F$ and $\mathcal E(\phi\circ u, \phi\circ u) \le \mathcal E(u,u)$ for every $u \in \mathcal F$ and every normal contraction $\phi$; see \cite[Theorem~1.4.1]{FukushimaDF}.

A Dirichlet form $(\mathcal E, \mathcal F)$ is called \emph{regular} if $\mathcal F \cap \mathcal C_M$ is simultaneously dense in $\mathcal C_M$ in the supremum norm and dense in $\mathcal F$ in the $\mathcal E_1^{1/2}$-norm; any such dense subalgebra is called a \emph{core} of $(\mathcal E, \mathcal F)$. It is called \emph{strongly local} if $\mathcal E(u,v) = 0$ whenever $u,v \in \mathcal F$ and $u$ is $\mu$-a.e.\ constant on an open neighborhood of the support of $v$. Equivalently, the Beurling--Deny decomposition of $\mathcal E$ contains neither jump nor killing part \cite[Theorem~3.2.1]{FukushimaDF}. Under Assumption~\ref{A:main}, the Dirichlet form $(\mathcal E, \mathcal F)$ associated with $L$ is both regular and strongly local; see \cite{baudoin2024heat}.

Every regular Dirichlet form $(\mathcal E, \mathcal F)$ determines a unique non-negative definite self-adjoint operator $L$ on $L^2_M$ with domain $\mathcal D_2(L) \subset \mathcal F$, characterized by
\begin{align}\label{eq:dform and generator}
    \mathcal E(u, v) = \langle L u, v\rangle_{L^2_M}, \qquad u \in \mathcal D_2(L), \ v \in \mathcal F,
\end{align}
and $-L$ generates a strongly continuous Markovian semigroup $\set{P_t}_{t \ge 0}$ on $L^2_M$ (see \cite[Theorem~1.3.1]{FukushimaDF} and \cite[Chapter~IX]{yosida2012functional}). %Under Assumption~\ref{A:main}, the semigroup admits a jointly continuous heat kernel $p_t(x,y)$ satisfying the sub-Gaussian bounds \eqref{ineq:hk_subg_compact_form}.

For every $u, v \in \mathcal F \cap L^\infty_M$, there exists a unique finite signed Radon measure $\Gamma(u,v)$ on $M$, called the \emph{energy measure} of $u$ and $v$, determined by the identity
\begin{align}\label{eq:energy measure def}
    \int_M \varphi \, \mathrm d \Gamma(u,v)
    = \frac{1}{2}\pr{\mathcal E(\varphi u, v) + \mathcal E(\varphi v, u) - \mathcal E(\varphi, uv)},
    \qquad \varphi \in \mathcal F \cap \mathcal C_M.
\end{align}
The construction extends by truncation to arbitrary $u,v \in \mathcal F$; see \cite[Section~3.2]{FukushimaDF}. By a slight abuse of notation, for $u,v \in \mathcal D_2(L) \cap L^\infty_M$ we also write
\begin{equation}\label{eq:energy measure L form}
    \Gamma(u,v) = \tfrac{1}{2}\pr{u L v + v L u - L(uv)},
\end{equation}
which is to be read as the identity~\eqref{eq:energy measure def} upon integration against a test function in $\mathcal F \cap \mathcal C_M$. By \cite[Lemma 5.4.3]{FukushimaDF}, we have the following Cauchy-Schwarz inequality for energy measures: for $f,g \in \mathcal F \cap \mathcal C_M$ and $u,v \in \mathcal F$, 
\begin{align}\label{ineq:cs for energy measure}
    \int_M \abs{fg} \mathrm d\abs{\Gamma\pr{u,v}} \le \pr{\int_M f^2 \mathrm d\Gamma(u,u)}^\frac{1}{2}\pr{\int_M g^2 \mathrm d\Gamma(v,v)}^\frac{1}{2}.
\end{align}

The form $\mathcal E$ and its energy measure are related by $\mathcal E(u,v) = \int_M \mathrm d\Gamma(u,v)$ whenever $1 \in \mathcal F$, which holds in particular under Assumption~\ref{A:main} since $M$ is compact. Under strong locality, $\Gamma$ satisfies the Leibniz rule
\begin{align*}
    \mathrm d \Gamma(u v, w) = u\, \mathrm d\Gamma(v,w) + v\, \mathrm d\Gamma(u,w),
    \qquad u,v,w \in \mathcal F \cap L^\infty_M,
\end{align*}
and the chain rule $\mathrm d\Gamma(g(u), v) = g'(u)\, \mathrm d\Gamma(u,v)$ for any $g \in \mathcal C^1(\R)$ with bounded derivative; see \cite[Theorem~3.2.2]{FukushimaDF} or \cite[Chapter~4]{chen2012symmetric}. Strong locality and the Leibniz rule will be used repeatedly, and will be essential to the construction of a global solution of~\eqref{eq:phi n formal}.

\begin{lemma}\label{lem:long time semi-group norm bound}
Let $a>0$ and set $c_a:= \inf_{x,y \in M}p_a(x,y) \in (0,1)$. Suppose $p \in[1,\infty]$ and $g \in L^p_M$ with $\int_M g\mathrm d\mu = 0$, then
\begin{align*}
  \norm{e^{-tL}g}_{L^p_M} \le (1-c_a)^{\left\lfloor \frac{t}{a} \right\rfloor} \norm{g}_{L^p_M}. 
\end{align*}
\end{lemma}
\noindent Under two-sided heat kernel estimates, \eqref{ineq:holder_kernel} can be strengthened.
\begin{lemma}\label{lem:long time heat kernel bounds}
Uniformly in $t \ge 1$ and $x,y \in M$, $p_t(x,y) \asymp 1$.  In addition, there exists $c,C>0$, the following inequality holds uniformly in $t>0$ and $x,x',y \in M$ with $d(x,x') < t^{1/d_w}$
\begin{align*}
    \abs{p_t(x,y) - p_t(x',y)} \lesssim \pr{\frac{d(x,x')}{t^{1/d_w}}}^\Theta t^{-d_h/d_w}\exp \pr{-C\pr{\frac{d(x,y)}{t^{1/d_w}}}^\frac{d_w}{d_w-1}-ct}.
\end{align*}
\end{lemma}

\begin{prop}\label{prop:kernel_bound_dt}
Let $k\in \N$ and $t >0$, $Q^{(k)}_t$ admits a symmetric density function $q_{k,t}$  w.r.t $\mu$, so that 
\begin{align}\label{eq:locality integral q equals 0}
  t^k\partial_t^kp_t(x,y) = q_{k,t}(x,y),\qquad  \int_M q_{k,t}(x,y) \mu(\mathrm dy) = 0,\qquad (t,x,y) \in (0,\infty)\times M^2.
\end{align}
There exists $C,c>0$ so that 
\begin{itemize}
    \item uniformly in $t >0$ and $x,y \in M$,
        \begin{align}\label{ineq: q sub gaussian bound}
          |q_{k,t}(x,y)| \lesssim t^{-d_h/d_w}\exp \pr{-C\pr{\frac{d(x,y)}{t^{1/d_w}}}^\frac{d_w}{d_w-1}-ct}.
        \end{align}
        \item uniformly in $t >0$ and  $x,x',y \in M$ with $d(x,x') \le t^{1/d_w}$,
\begin{align}\label{ineq: q holder continuous}
  \abs{q_{k,t}(x,y) - q_{k,t}(x',y)} \lesssim \pr{\frac{d(x,x')}{t^{1/d_w}}}^\Theta t^{-d_h/d_w}\exp \pr{-C\pr{\frac{d(x,y)}{t^{1/d_w}}}^\frac{d_w}{d_w-1}-ct}.
\end{align}
\end{itemize}
Consequently, for each $k \in \N$ and $p \in [1,\infty]$, it holds uniformly in $t>0$ and $f \in L^p_M$ that 
\begin{align*}
    \norm{Q^{(k)}_tf}_{L^p_M} \lesssim e^{-ct}\norm{f}_{L^p_M}.
\end{align*}
\end{prop}

%The following lemma is an elementary, its proof can be found in the Appendix.

\begin{lemma}{\cite[Lemma 2.3]{alonso2021besov}}\label{lem:subg_moment}
For any $\alpha \in [0,\infty)$, there exists $c = c(\alpha)>0$ so that uniformly in $t >0$ and $x,y \in M$,
$$
d(x,y)^\alpha p_t(x,y) \lesssim \pr{t^\frac{\alpha}{d_w}\wedge 1} p_{ct}(x,y).
$$
\end{lemma}

\begin{lemma}\label{lem:hk LpLq interpolation}
For $1 \le p \le q \le \infty$, uniformly in $t>0$,
\begin{align*}
  \norm{e^{-tL}f}_{L^q_M} \lesssim \pr{  1+t^{-\frac{d_h}{d_w}\pr{\frac{1}{p} - \frac{1}{q}}}} \norm{f}_{L^p_M}.
\end{align*}
\end{lemma}

\section{Heat semi-group based Besov Spaces}\label{sec:besov}

In this section, we explore heat semi-group based Besov spaces and their relevant properties. We will first list some elementary properties of Besov norms. Next, we study how they relate to H\"{o}lder spaces and with each other via duality, interpolation, various embedding properties, and the smoothing effect of heat flow. We will also study how Besov norms are related to Dirichlet forms and state some useful inequalities that will help us analyze $L^p_M$-norms of solutions to nonlinear (S)PDEs. We define $\alpha$-H\"older norm $\norm{\cdot}_{\mathcal C^\alpha_M}$ for $\alpha \in (0,1]$:
\begin{align*}
\norm{f}_{\mathcal C^\alpha_M}
:=
\norm{f}_{\mathcal C_M}
+
\sup_{0<d(x,y)\le 1}\frac{|f(x)-f(y)|}{d(x,y)^\alpha}.
\end{align*}

\subsection{Elementary properties of Besov norms}
The following almost orthogonality relation will be useful: for each $p \in[1,\infty]$, $k_1,k_2 \ge 0$ and $t,s>0$,
\begin{equation}\label{ineq:almost orthogonal}
    \begin{aligned}
        \norm{Q^{(k_1)}_sQ^{(k_2)}_t}_{L^p_M \to L^p_M} &= \norm{e^{-t\wedge sL } \min\left\{ \pr{\frac{t}{s}}^{k_2}, \pr{\frac{s}{t}}^{k_1}\right\} Q^{(k_1+k_2)}_{t\vee s}}_{L^p_M \to L^p_M}\\
        &\le \min\left\{ \pr{\frac{t}{s}}^{k_2}, \pr{\frac{s}{t}}^{k_1}\right\}\norm{ Q^{(k_1+k_2)}_{t\vee s}}_{L^p_M \to L^p_M}.
    \end{aligned}
\end{equation}

\noindent We will need the following version of Calderon's reproducing formula. 
\begin{theorem}\label{thm:crf}
Let $p \in (1,\infty)$ and $f \in L^p_M$, then  for all $k\in \N$
\begin{align*}
  \lim_{t \downarrow 0} P_t^{(k)}f = f,\quad \text{ in }L^p_M
\end{align*}
and
\begin{align*}
  f =  \frac{1}{\gamma_k}\int_0^1 Q_t^{(k)}f \frac{\mathrm dt}{t} +P^{(k)}_1(f),\qquad \text{in    }L^p_M.
\end{align*}
If $f \in \mathcal C_M$, then the convergence is in $\mathcal C_M$. 
\end{theorem}
\noindent The proof of the case for $f \in L^p_M$ and $p \in (1,\infty)$ can be found in \cite[Proposition 2.7]{grigor2015heat}. For $f \in \mathcal C_M$, this proof applies \textit{mutatis mutandis} by Feller property of $(P_t)_t$. 

\noindent The Besov norms defined in \eqref{eq:besov norm def pq} are independent of the index $k$  for sufficiently large $k$. This is essentially shown in \cite[Proposition 2.9]{grigor2015heat}, and we will omit its proof.
\begin{prop}\label{prop:independent_b}
Suppose $\alpha  \in \R$ and $p,q \in [1,\infty]$, the norm  $\norm{\cdot}_{\mathcal B^\alpha_{p,q}}$ defined as in Definition \ref{def:besov_space}  is equivalent for $k\in \Z_+$ with $k> \alpha$. 
%In particular, $k$ does not need to be an integer.
\end{prop}

\begin{prop}\label{prop:Pa1_besov_bound}
For $\sigma \in (-\infty,d_w)$ and $k \in \N$, we have
\begin{align*}
  \norm{P^{(k)}_1f}_{L^\infty_M} \lesssim \norm{f}_{\mathcal B^\sigma_{\infty}}.
\end{align*}
\end{prop}
\begin{proof}
    \cite[Proposition 2.7.]{bailleul2016heat} is this proposition with $d_w=2$. The proof there can be adapted \textit{mutatis mutandis}, so we omit it.
\end{proof}
\noindent The following proposition shows that the Besov norm is also independent of the time horizon.
\begin{prop}\label{prop: norm equivalent time horizon}
Suppose $\alpha \in \R$ and $p,q \in [1,\infty]$, then for $k \in \Z_+$ with $k > \alpha$ and $T \in [1,\infty]$, denote for $f \in \mathcal B^\alpha_{p,q}$ that
\begin{align*}
  \norm{f}_{\mathcal B^\alpha_{p,q}(T)}:=\norm{P_1f}_{L^p} + \pr{\int_0^T \pr{t^{-\frac{\alpha}{d_w}} \norm{Q^{(k)}_t f}_{L^p_M}}^q \frac{\mathrm{d}s    }{s} }^\frac{1}{q}
\end{align*}
with obvious modifications when $q = \infty$. Then for any $T,T' \in [1,\infty]$, it holds uniformly in $f \in \mathcal B_{p,q}^\alpha$ that
\begin{align*}
  \norm{f}_{\mathcal B^\alpha_{p,q}(T)} \asymp\norm{f}_{\mathcal B^\alpha_{p,q}(T')}
\end{align*}
\end{prop}
\begin{proof}
Note that it is enough to take $T = 1$ and $T' = \infty$ and prove the  $\gtrsim$ direction. Observe that by stochastic completeness and Proposition \ref{prop:Pa1_besov_bound}, we have for any $k >0$ and $s>0$ \[\int_M L^k e^{-sL}f(x)\mu(\mathrm d x) = 0.\]
Thus by Lemma \ref{lem:long time semi-group norm bound}, we have for $t \in [1,2]$,
\begin{align*}
  \norm{Q^{(k)}_t f}_{L^p_M}  &= \norm{e^{-(t - 1)L} (tL)^kP_1f}_{L^p_M} \le t^k \norm{L^k P_1 f}_{L^p_M},\qquad t \in [1,2],\\
  \norm{Q^{(k)}_t f}_{L^p_M}  &= \norm{e^{-(t - 1)L} (tL)^kP_1f}_{L^p_M} \le t^k(1-c_1)^{\left \lfloor t-1 \right \rfloor} \norm{L^k P_1 f}_{L^p_M},\qquad t \in [2,\infty),
\end{align*}
where $c_1 = \inf_{x,y} p_1(x,y) \in (0,1)$  by the lower bound in \eqref{ineq:sgu}. Therefore, for any $p\in [1,\infty]$ and $q \in [1,\infty)$,
\begin{align*}
  \pr{\int_1^{\infty}\pr{t^{-\frac{\alpha}{d_w}} \norm{Q^{(k)}_t f}_{L^p_M}}^q \frac{\mathrm{d}s}{s} }^\frac{1}{q} \lesssim \norm{L^k P_1f}_{L^p_M},
\end{align*}
and $\sup_{t \ge 1} t^{-\frac{\alpha}{d_w}} \norm{Q_t^{(k)} f  }_{L^p_M} \lesssim \norm{L^k P_1f}_{L^p_M}.$ When $q = \infty$, we see
\begin{align*}
  \sup_{t > 0} t^{-\frac{\alpha}{d_w}} \norm{Q_t^{(k)} f  }_{L^p_M} \lesssim \sup_{t \in (0,1]} t^{-\frac{\alpha}{d_w}} \norm{Q_t^{(k)} f  }_{L^p_M}+\norm{L^k P_1f}_{L^p_M} \lesssim \sup_{t \in (0,1]} t^{-\frac{\alpha}{d_w}} \norm{Q_t^{(k)} f  }_{L^p_M}.
\end{align*}
For $q < \infty$, observe that uniformly in  $s \in [1/2,1],$
\begin{align*}
  \norm{L^k P_1f}_{L^p_M} = s^{-k}\norm{e^{-(1-s)L}(sL)^k e^{-sL}f}_{L^p_M} \lesssim \norm{Q^{(k)}_s f}_{L^p_M}.
\end{align*}
Hence
\begin{align*}
  \norm{L^k P_1f}_{L^p_M} &\asymp \pr{\int_{1/2}^1 \pr{s^{-\frac{\alpha}{d_w}}  \norm{Q^{(k)}_t f}_{L^p_M}}^q\frac{\mathrm ds}{s}}^\frac{1}{q}\\
  &\lesssim \pr{\int_{1/2}^1 \pr{s^{-\frac{\alpha}{d_w}}  \norm{Q^{(k)}_s f}_{L^p_M}}^q\frac{\mathrm ds}{s}}^\frac{1}{q},
\end{align*}
which implies the desired inequality.
\end{proof}

\subsection{Regularizing effect of heat flow}
The following results demonstrates the regularization effects of the heat semi-group.

\begin{lemma}\label{lem:q_regulariz}
Suppose $f \in  \mathcal B^\sigma_{p,\infty}$ for some $\sigma \in \R$ and $p \in [1,\infty]$.
Let $b \in \Z_+$ and $\sigma\le\alpha $, then for   $a \in \mathbb N$ with $a \ge \alpha/d_w$ and  $s >0$ that
\begin{align*}
  \norm{Q_s^{(b)} f}_{\dot{\mathcal B}^\alpha_{p,\infty}} \le s^\frac{\sigma - \alpha}{d_w} \norm{f}_{\dot{\mathcal B}^\sigma_{p,\infty}}.
\end{align*}
Furthermore, if $\sigma \le 0$, then uniformly in $t \in (0,1]$
\begin{align*}
  \norm{P_t f}_{L^p_M} \lesssim
  \begin{cases}
    t^{\frac{\sigma}{d_w}}\norm{f}_{\dot{\mathcal B}^\sigma_{p,\infty}} + \norm{P_1f}_{L^p_M},&\text{if }\sigma <0,\\
    \log(1/t)\norm{f}_{\dot{\mathcal B}^0_{p,\infty}} + \norm{P_1f}_{L^p_M},&\text{if }\sigma =0,
  \end{cases}
\end{align*}
In particular, for each $s >0$,  $Q^{(b)}_sf \in \mathcal B^\alpha_{p,\infty}$ for all $\alpha >0$.
\end{lemma}
\begin{proof}
    A simple calculation gives 
    \begin{align*}
       t^{-\frac{\alpha}{d_w}} \norm{Q_t^{(a)} Q_s^{(b)}f}_{L^p_M} &\overset{\eqref{ineq:almost orthogonal}}{\le }t^{-\frac{\alpha}{d_w}}\pr{\frac{t}{s}}^a\wedge \pr{\frac{s}{t}}^b \norm{Q^{(a+b)}_{t\vee s}f}_{L^p_M}\\
       &\le t^{-\frac{\alpha}{d_w}} \pr{\frac{t}{s} }^a\wedge \pr{\frac{s}{t}}^b \pr{t\vee s}^{\frac{\sigma}{d_w}} \norm{f}_{\dot{\mathcal B}_{p,\infty}^\sigma}.
    \end{align*}
    Thus, 
    \begin{align*}
         t^{-\frac{\alpha}{d_w}} \norm{Q_t^{(a)} Q_s^{(b)}f}_{L^p_M} \le \begin{cases}
             t^{-\frac{\alpha}{d_w}+a}s^{\frac{\sigma}{d_w} - a} \norm{f}_{\dot{\mathcal B}_{p,\infty}^\sigma},& t\le s\\
             t^{ -\frac{\alpha}{d_w}-b+\frac{\sigma}{d_w}} s^{b} \norm{f}_{\dot{\mathcal B}_{p,\infty}^\sigma},& t \ge s.
         \end{cases}
    \end{align*}
    Since $a> \alpha/d_w$, $\sigma\le \alpha $ and $b \ge 0$, we see 
    \begin{align*}
        \sup_{t \in (0,s]} t^{-\frac{\alpha}{d_w }+a}s^{\frac{\sigma}{d_w}-a} = s^\frac{\sigma - \alpha}{d_w}=\sup_{t \ge s} t^{-\frac{\alpha}{d_w}-b +\frac{\sigma}{d_w}} s^b,
    \end{align*}
    which proves the first inequality. For the second inequality, we have by the reverse triangle inequality
    \begin{align*}
        \norm{P_tf}_{L^p_M} - \norm{P_1f}_{L^p_M}  \le \norm{P_tf - P_1f}_{L^p_M}  \le \int_t^1 \norm{Q^{(1)}_sf}_{L^p_M}  \frac{\mathrm ds}{s} \le \int_t^1 s^\frac{\sigma}{d_w} \frac{\mathrm ds}{s} \norm{f}_{\dot{\mathcal B}_{p,\infty}^\sigma},
    \end{align*}
    from which the desired inequalities follow by single variable calculus. 
\end{proof}

Next we show strong continuity of the heat semi-group acting on Besov spaces. 
\begin{lemma}\label{lem:strong continuity of heat semi group}
Let $\sigma \in \R$, $\delta\in(0,1)$. Then uniformly in  $f \in \Binf{\sigma}$ and $t \in (0,1]$,
\begin{align*}
  \norm{P_t f - f}_{\Binf{\sigma - \delta}} \lesssim t^\frac{\delta}{d_w}\norm{f}_{\Binf{\sigma}}.
\end{align*}
For any $\lambda>0$,
\begin{align*}
  \norm{P_tf - f}_{L^\infty_M} \lesssim t^\frac{\lambda}{d_w}\norm{f}_{\Binf{\lambda}}.
\end{align*}
As a consequence, $P_t$ is strongly continuous on $\Binf{\sigma}$ for every $\sigma \in \R$.
\end{lemma}

\begin{proof}

Let $k \ge 0$ be sufficiently large, and consider for $0<\tau< t$ that
\begin{align*}
  \tau^{-\frac{\sigma-\delta}{d_w}}
  \int_0^t \norm{Q_\tau^{(k)} Q_s^{(k)} f}_{L^\infty_M}\,\frac{\mathrm ds}{s}
  &\lesssim
  \tau^{-\frac{\sigma-\delta}{d_w}}
  \int_0^t
  \left(\frac{\min\{\tau,s\}}{\max\{\tau,s\}}\right)^k
  \norm{Q_{\max\{\tau,s\}}^{(2k)} f}_{L^\infty_M}\,\frac{\mathrm ds}{s}
  \\
  &\lesssim
  \tau^{-\frac{\sigma-\delta}{d_w}}
  \int_0^t
  \left(\frac{\min\{\tau,s\}}{\max\{\tau,s\}}\right)^k
  \max\{\tau,s\}^{\frac{\sigma}{d_w}}
  \frac{\mathrm ds}{s}\,
  \norm{f}_{\Binf{\sigma}}
  \\
  &=
  \tau^{\frac{\delta}{d_w}}
  \int_0^t
  \left(\frac{\min\{\tau,s\}}{\max\{\tau,s\}}\right)^k
  \left(\frac{\max\{\tau,s\}}{\tau}\right)^{\frac{\sigma}{d_w}}
  \frac{\mathrm ds}{s}\,
  \norm{f}_{\Binf{\sigma}}\\
  &=
  \tau^{\frac{\delta}{d_w}}
  \left(
    \int_0^\tau \left(\frac{s}{\tau}\right)^k \frac{\mathrm ds}{s}
    +
    \int_\tau^t \left(\frac{\tau}{s}\right)^{k-\frac{\sigma}{d_w}} \frac{\mathrm ds}{s}
  \right)
  \norm{f}_{\Binf{\sigma}}
  \\
  &=
  \tau^{\frac{\delta}{d_w}}
  \left(
    \frac1k
    +
    \frac{1-\left(\frac{\tau}{t}\right)^{k-\frac{\sigma}{d_w}}}{\,k-\frac{\sigma}{d_w}\,}
  \right)
  \norm{f}_{\Binf{\sigma}}\lesssim t^\frac{\delta}{d_w} \norm{f}_{\Binf{\sigma}}.
\end{align*}
If $\tau \ge t$, then
\begin{align*}
  \tau^{\frac{\delta}{d_w}}
  \int_0^t \left(\frac{s}{\tau}\right)^k \frac{\mathrm ds}{s}
  =
  \frac{1}{k}\tau^{\frac{\delta}{d_w}}\left(\frac{t}{\tau}\right)^k
  \le \frac1k\, t^{\frac{\delta}{d_w}},
\end{align*}
which implies
\begin{align*}
  \tau^{-\frac{\sigma-\delta}{d_w}} \int_0^t \norm{Q^{(k)}_\tau Q^{(k)}_s f}_{L^\infty_M} \frac{\mathrm ds}{s} \lesssim t^\frac{\delta}{d_w} \norm{f}_{\Binf{\sigma}}.
\end{align*}
By taking supremum in $\tau \in (0,1]$, we see by Theorem \ref{thm:crf} that
\begin{align*}
  \norm{P_tf -f}_{\dot{\mathcal B}_\infty^{\sigma - \delta}}\lesssim t^\frac{\delta}{d_w}\norm{f}_{\Binf{\sigma} }.
\end{align*}
Similarly, consider
\begin{align*}
  \norm{P_1\pr{P_t f -f}}_{L^\infty_M}\lesssim \int_0^t \norm{Q^{(1)}_s P_1f}_{L^\infty_M} \frac{\mathrm ds}{s}\lesssim \int_0^t s^\frac{\delta}{d_w} \frac{\mathrm ds}{s} \norm{P_1f}_{\dot{\mathcal B}^\frac{\delta}{d_w}_\infty}\lesssim t^\frac{\delta}{d_w} \norm{f}_{\Binf{\sigma}},
\end{align*}
where we used the fact that $\delta \in (0,1)$ and Proposition \ref{prop:independent_b} in the second inequality and Lemma \ref{lem:q_regulariz} in the last inequality. Collecting the estimates above yields the first inequality.

Now consider for $\lambda>0$, we have
\begin{align*}
  \norm{P_tf - f}_{L^\infty_M}\lesssim \int_0^t \norm{Q^{(k)}_sf}_{L^\infty_M} \frac{\mathrm ds}{s}\lesssim \int_0^t s^\frac{\lambda}{d_w} \frac{\mathrm ds}{s}\norm{f}_{\Binf{\lambda}}\lesssim t^\frac{\lambda}{d_w}\norm{f}_{\Binf{\lambda}}.
\end{align*}

For the strong continuity of $P_t$, we let $\alpha \in \R$, $f \in \Binf{\alpha}$ and  choose $\{f_n\}_{n} \subset \Binf{1+\abs{\alpha}}$ that converges to $f$ in $\Binf{\alpha}$. Then by Lemma \ref{lem:q_regulariz},
\begin{align*}
    \lim_{t \downarrow 0}\norm{P_tf - f}_{\Binf{\alpha}} &\le \lim_{t \downarrow 0}\norm{P_t(f-f_n)}_{\Binf{\alpha}}+\lim_{t \downarrow 0}\norm{P_tf_n - f_n}_{\Binf{\alpha}    } + \norm{f_n - f}_{\Binf{\alpha}}\\
    &\lesssim \norm{f - f_n}_{\Binf{\alpha}}.
\end{align*}
Sending $n \to \infty$ yields the desired result.
\end{proof}

\subsection{Relation to H\"{o}lder space and dyadic formulation}
The following proposition determines the relation between the Besov space $\mathcal B^\sigma_{\infty}$ and H\"{o}lder space ${\mathcal C}^\sigma_M$.
\begin{prop}\label{prop:holder=besov}
Recall $\Theta\in (0,1]$ in \eqref{ineq:holder_kernel}. The following inequalities holds uniformly for $f\in \Binf{\sigma}$:
\begin{align}
  \norm{f}_{\mathcal B^\sigma_{\infty}}& \asymp \norm{f}_{\mathcal C^\sigma},\qquad \text{ if }\sigma \in (0,\Theta)\label{ineq:holder=besov}\\
      \norm{f}_{\mathcal B^\sigma_{\infty}}&\lesssim \norm{f}_{\mathcal C^\sigma},\qquad \text{ if }\sigma \in (0,1).\label{ineq:besov<holder}
\end{align}
\end{prop}

\begin{proof}%[Proof of Proposition \ref{prop:holder=besov}]
Let $\sigma \in (0,1)$ and  $f \in {\mathcal C}^\sigma$. By Proposition \ref{prop:kernel_bound_dt}, we see for any $t>0$, $a \in \mathbb N$ and $x \in M$,
\begin{align*}
  Q^{(a)}_tf(x) = Q^{(a)}_t\left( f(\cdot) - f(x)\right)(x) = \int_M \left(q_{a,t}(x,y) \left(f(y) - f(x) \right)\right) \mu(\mathrm dy).
\end{align*}
Therefore, by Lemma \ref{lem:subg_moment} and Proposition \ref{prop:kernel_bound_dt}, we see uniformly in $t \in (0,1]$ and $x \in M$ that
\begin{align*}
  \abs{Q^{(a)}_tf(x)} &=  \left|\int_{M} q_{a,t}(x,y) (f(z) - f(x)) \mu(dz) \right| \\
  &\le \norm{f}_{{\mathcal C}^\sigma} \int_M \left| q_{a,t}(x,y) \right| d(x,y)^\sigma \mu(dy)\\
  &\lesssim t ^\frac{\sigma}{d_w}\norm{f}_{{\mathcal C}^\sigma},
\end{align*}
which implies $\norm{f}_{\mathcal B^\sigma_{\infty}} \lesssim \norm{f}_{{\mathcal C}^\sigma}.$ This completes the proof of \eqref{ineq:besov<holder}. %Note that this means we need only prove the $\gtrsim$ side of \eqref{ineq:holder=besov}.

Suppose $f \in \mathcal B^\sigma_{\infty}$, for some $0 < \sigma < \Theta \le 1$. Since $\frac{\sigma}{d_w}<1$, by Calderon's reproducing formula
\begin{align*}
  \norm{f}_{L^\infty_M}& \le \norm{P_1f}_{L^\infty_M} + \int_0^1 \norm{Q^{(1)}_tf}_{L^\infty_M}  \frac{\mathrm dt}{t}\\
  &\lesssim \norm{f}_{\mathcal B^\sigma_{\infty}}\left(1+ \int_0^1 t^\frac{\sigma}{d_w} \frac{\mathrm dt}{t} \right)\lesssim \norm{f}_{\mathcal B^\sigma_{\infty}}.
\end{align*}
Moreover, for any two points $x,y \in M$ with $0<d(x,y) \le 1$, we have
\begin{align*}
  f(x) - f(y) &= P_1f(x) - P_1f(y)+Q^{(1)}_1f(x) -  Q^{(1)}_1f(y)\\
  &\quad \quad + \int_0^1 Q^{(2)}_tf(x) -  Q^{(2)}_tf(y) \frac{\mathrm dt}{t}.
\end{align*}
By \eqref{ineq:holder_kernel} and Proposition \ref{prop:kernel_bound_dt}, it holds uniformly in $x,y \in M$ that 
\begin{align*}
  \left|P_1f(x) - P_1f(y)\right| +\abs{Q_1^{(1)} f(x) - Q^{(1)}_1f(y)} \lesssim d(x,y)^\Theta \norm{f}_{L^\infty_M}.
\end{align*}

Now observe that $Q^{(2)}_t = 16Q^{(2)}_\frac{t}{2}Q^{(2)}_\frac{t}{2}$. So for $d(x,y) \le t^\frac{1}{d_w}$, by Proposition \ref{prop:kernel_bound_dt} again,
\begin{align*}
  \left|Q^{(2)}_tf(x) - Q^{(2)}_t f(y) \right|&\lesssim \int_M \left|q_{2,t/2}(x,z) - q_{2,t/2}(y,z) \right| \left| Q^{2}_{t/2} f(z)\right|\mu(dz)\\
  &\lesssim \left(\frac{d(x,y)}{t^\frac{1}{d_w}} \right)^\Theta \norm{Q^{(2)}_{t/2}f}_{L^\infty_M} \\
  &\lesssim \left(\frac{d(x,y)}{t^\frac{1}{d_w}} \right)^\Theta t^\frac{\sigma}{d_w} \norm{f}_{\mathcal B^\sigma_{\infty}}
\end{align*}
If $t^\frac{1}{d_w} \le d(x,y) \le 1$, we have
\begin{align*}
  \left|Q^{(2)}_tf(x) - Q^{(2)}_tf(y) \right| \lesssim \norm{Q^{(2)}_tf}_{L^\infty_M}  \lesssim t^\frac{\sigma}{d_w} \norm{f}_{\mathcal B^\sigma_{\infty}}.
\end{align*}
Therefore,
\begin{align*}
  \left|\int_0^1 Q_t^{(2)}f(x) - Q^{(2)}_tf(y) \right| \frac{\mathrm dt}{t} &\lesssim \left(\int_0^{d(x,y)^{d_w}} t^\frac{\sigma}{d_w} \frac{\mathrm dt}{t} + \int_{d(x,y)^{d_w} }^1 \left(\frac{d(x,y)}{t^\frac{1}{d_w}} \right)^\Theta t^\frac{\sigma}{d_w}\frac{\mathrm dt}{t}\right)\norm{f}_{\mathcal B^\sigma}\\
  &\lesssim \left(d(x,y)^\sigma +d(x,y)^\Theta \right) \norm{f}_{\mathcal B^\sigma_{\infty}}.
\end{align*}
Therefore, since $d(x,y) \le 1$,
\begin{align*}
  \left|f(x) - f(y) \right| &\lesssim \left( d(x,y)^\sigma + d(x,y)^\Theta\right)\norm{f}_{\mathcal B^\sigma}\\
  &\lesssim d(x,y)^{\min\{\sigma,\Theta\}}\norm{f}_{\mathcal B^\sigma_{\infty}}.
\end{align*}
So if $\sigma \le \Theta$, we have
\begin{align*}
  \norm{f}_{{\mathcal C}^\sigma} \lesssim \norm{f}_{\mathcal B^\sigma_{\infty}}.
\end{align*}
\end{proof}

\noindent Sometimes it is more convenient to work with a discrete version of the Besov norm, which resembles the Littlewood-Paley formulation. The following proposition gives such a version.
\begin{prop}\label{prop:dyadic norm}
For any $\alpha \in \mathbb R$ and $k \in \mathbb N$ with $k > \frac{\alpha}{d_w}$. Let $t_j := 2^{-d_wj}$ for $j \in \mathbb N$ and for any $f \in \mathcal S'$, $$A_j(f):= \pr{t_jL}^k P_{t_j} f.$$
Then for any $p \in [1,\infty]$ and $q \in [1,\infty)$,
\begin{align*}
  \norm{f}_{\mathcal B^\alpha_{p,q}} \asymp \pr{\sum_{j = 0}^\infty \pr{2^{j\alpha} \norm{A_j(f)}_{L^p_M} }^q }^\frac{1}{q}  + \norm{P_1f}_{L^p_M},
\end{align*}
with the constant depending only on $\alpha$, $k$ and $p,q \in \mathbb N$ but not $f$. When $q = \infty$,
\begin{align*}
  \norm{f}_{\mathcal B^\alpha_{p,\infty}}\asymp  \sup_{j \in \mathbb N} 2^{j\alpha} \norm{A_j f}_{L^p_M} + \norm{P_1f}_{L^p_M}.
\end{align*}
\end{prop}

\begin{proof}
Let $t \in (0,1]$, choose $j \in \mathbb N$ so that $2^{-d_w(j+1)} < t \le 2^{-d_wj}$, then
\begin{align}\label{ineq:dyadic dw decomposition of time}
  2^{-\alpha} \le t^{-\frac{\alpha}{d_w}} 2^{-\alpha(j+1)} \le 1\text{ for } \alpha \ge 0,\qquad 2^{-\alpha} \ge t^{-\frac{\alpha}{d_w}} 2^{-\alpha(j+1)} \ge 1 \text{ for } \alpha < 0.
\end{align}
Hence,
\begin{align*}
  t^{-\frac{\alpha}{d_w}}\norm{ \pr{tL}^k P_t f}_{L^p_M} &= \pr{\frac{t^{-\alpha /d_w}}{2^{\alpha d_w(j+1)/d_w}}}\pr{\frac{t}{2^{-d_w(j+1)}}}^k 2^{\alpha (j+1)}\norm{\pr{2^{-d_w{(j+1)}}L}^kP_{t-t_{j+1}}P_{t_{j+1}}f}_{L^p_M}\\
  &\overset{\eqref{ineq:dyadic dw decomposition of time}}{\lesssim }\norm{P_{t - t_{j+1}}}_{L^p_M \to L^p_M} 2^{\alpha(j+1)}\norm{\pr{2^{-d_w(j+1)}L}^kP_{t_{j+1}}f}_{L^p_M}\\
  &= 2^{\alpha (j+1)}\norm{A_{j+1}(f)}_{L^p_M}.
\end{align*}
On the other hand,
\begin{align*}
  2^{\alpha j}\norm{A_j(f)}_{L^p_M} &=2^{\alpha j}\norm{\pr{t_j L}^k P_{t_j} f}_{L^p_M}\\
  &= t^{\frac{\alpha}{d_w}}2^{\alpha j}\pr{t^{-1}2^{-jd_w}}^kt^{-\frac{\alpha}{d_w}} \norm{P_{t_j - t}(tL)^k P_tf}_{L^p_M}\\
  &\lesssim \norm{P_{t_j - t}}_{L^p_M \to L^p_M} t^{-\frac{\alpha}{d_w}}\norm{(tL)^k P_t f}_{L^p_M}\\
  &\le  t^{-\frac{\alpha}{d_w}}\norm{(tL)^k P_t f}_{L^p_M}.
\end{align*}
Hence we have
\begin{align}\label{ineq: comparible of time point and dyadic}
  2^{\alpha j}\norm{A_j(f)}_{L^p_M} \lesssim t^{-\frac{\alpha}{d_w}}\norm{(tL)^k P_t f}_{L^p_M}\lesssim 2^{\alpha(j+1)} \norm{A_{j+1} f}_{L^p_M}.
\end{align}
We now prove the inequality for $q\neq \infty$. For the upper bound, we have
\begin{align*}
  \int_0^1 \pr{t^{-\frac{\alpha}{d_w}} \norm{\pr{tL}^kP_t f}_{L^p_M}}^q \frac{\mathrm dt}{t} &= \sum_{j \in \mathbb N} \int_{2^{-d_w(j+1)}}^{2^{-d_w j}}\pr{t^{-\frac{\alpha}{d_w}} \norm{\pr{tL}^kP_t f}_{L^p_M}}^q \frac{\mathrm dt}{t} \\
  &\overset{\eqref{ineq: comparible of time point and dyadic}}{\lesssim} \sum_{j \in \mathbb N} 2^{d_w j}\pr{2^{-d_w j}-2^{-d_w(j+1)}} \pr{2^{\alpha (j+1)} \norm{A_{j+1} (f)}_{L^p_M}}^q\\
  &\lesssim \sum_{j \ge 1} \pr{2^{\alpha (j)} \norm{A_{j} (f)}_{L^p_M}}^q.
\end{align*}
For the lower bound, we have
\begin{align*}
  \int_0^1 \pr{t^{-\frac{\alpha}{d_w}} \norm{\pr{tL}^kP_t f}_{L^p_M}}^q \frac{\mathrm dt}{t} &= \sum_{j \in \mathbb N} \int_{2^{-d_w(j+1)}}^{2^{-d_w j}}\pr{t^{-\frac{\alpha}{d_w}} \norm{\pr{tL}^kP_t f}_{L^p_M}}^q \frac{\mathrm dt}{t} \\
  &\overset{\eqref{ineq: comparible of time point and dyadic}}{\gtrsim}\sum_{j \in \mathbb N} 2^{d_w (j+1)} \pr{2^{-d_w j}-2^{-d_w(j+1)}}\pr{2^{\alpha j}\norm{A_j(f)}_{L^p_M}}^q\\
  &\gtrsim \sum_{j \ge 0} \pr{2^{\alpha j}\norm{A_j(f)}_{L^p_M}}^q.
\end{align*}
Therefore,
\begin{align*}
  \norm{P_1 f}_{L^p_M}+ \pr{\sum_{j \ge 0} \pr{2^{\alpha j}\norm{A_j(f)}_{L^p_M}}^q}^\frac{1}{q} \lesssim \norm{f}_{\mathcal B_{p,q}^\alpha} \lesssim \norm{P_1f}_{L^p_M}+\pr{\sum_{j \ge 1} \pr{2^{\alpha (j)} \norm{A_{j} (f)}_{L^p_M}}^q}^\frac{1}{q},
\end{align*}
which implies the desired inequality.

The inequality for $q=\infty$ follows by taking supremum over $j \in \mathbb N$ in \eqref{ineq: comparible of time point and dyadic}.\end{proof}

\subsection{Duality and Interpolation Inequalities}
%Next, we show the duality relation between Besov spaces and interpolation inequalities of Besov norms is similar to that of the Euclidean case.
\begin{lemma}[Duality]\label{lem:Besov Duality}
Let $\alpha\in \R$, $p,p',q,q'\in[1,+\infty]$ satisfy $\frac{1}{p}+\frac{1}{p'}=\frac{1}{q}+\frac{1}{q'}=1$. Then for any smooth $f,g:M\to \R$, we have $$\abs{\braket{f,g}}:=\abs{\braket{f,g}_{L^2}}\lesssim \norm{f}_{\mathcal B^{\alpha}_{p,q}}\norm{g}_{\mathcal B^{-\alpha}_{p',q'}}.$$
Hence, the pairing $\langle\cdot, \cdot \rangle$ extends to $\mathcal B^\alpha_{p,q}\times \mathcal B^{-\alpha}_{p',q'}$ continuously.
\end{lemma}
In light of Lemma \ref{lem:Besov Duality}, for $\alpha \in \R$ and $p,\, q,\, p',\, q' \in [1,\infty]$ satisfies the condition in Lemma \ref{lem:Besov Duality}, we denote $\langle \cdot,\cdot\rangle$ as the pairing between elements in $\mathcal B_{p,q}^\alpha$ and $\mathcal B^{-\alpha}_{p',q'}$ when there is no confusion going forward.
\begin{proof}%[Proof of Lemma \ref{lem:Besov Duality}]
Recall the Carderon's reproducing formula from Theorem \ref{thm:crf}, let $k \in \mathbb N$ with $k > \alpha$, we can write
\begin{align*}
  \braket{f,g} &= \braket{ f,P_1^{(2k)}g} + \int_0^1 \braket{f,Q^{(2k)}_t g}\frac{\mathrm dt}{t} .
\end{align*}
By \cite[Lemma 3.3]{grigor2015heat}, we have
\begin{align*}
  \abs{\braket{f,P_1^{(2k)}g}} &\le   \sum_{\ell = 0}^{2k-1} p_\ell \abs{\braket{f, L^\ell  P_1g}}\\
  &\lesssim\sum_{\ell = 0}^{2k-1} \abs{\braket{e^{-\frac{1}{2}L}f, L^\ell e^{-\frac{1}{2}}g}}\\
  &\lesssim \sum_{\ell = 0}^{2k-1} \norm{e^{-\frac{1}{2}L}f}_{L^p_M} \norm{L^\ell e^{-\frac{1}{2}L}g}_{L^{p'}_M}\\
  &\lesssim \sum_{\ell = 0}^{2k-1} \norm{e^{-\frac{1}{2}L}f}_{L^p_M} \norm{L^\ell e^{-\frac{1}{4}L}}_{L^{p'}_M \to L^{p'}_M}\norm{ e^{-\frac{1}{4}L}g}_{L^{p'}_M}\\
  &\lesssim \norm{f}_{\mathcal B^\alpha_{p,q}} \norm{g}_{\mathcal B^{-\alpha}_{p',q'}}
\end{align*}
For the seminorm part, we have
\begin{align*}
  \abs{\int_0^1 \braket{f,Q_t^{(2k)}g} \frac{\mathrm dt}{t}} &=\abs{ \int_0^1\braket{f,t^{2k} L^{2k} e^{-tL} g} \frac{\mathrm dt}{t}}\\
  &=\abs{ \int_0^1\braket{(tL)^ke^{-\frac{t}{2}L}f,(tL)^k e^{-\frac{t}{2}L} g} \frac{\mathrm dt}{t}}\\
  &\lesssim \int_0^1 \pr{\frac{t}{2}}^{-\frac{\alpha}{d_w}}\norm{Q^{(k)}_{t/2} f}_{L^p_M} \pr{\frac{t}{2}}^\frac{\alpha}{d_w}\norm{Q^{(k)}_{t/2} g}_{L^{p'}_M} \frac{\mathrm dt}{t}\\
  &\le \pr{\int_0^1 \pr{\pr{\frac{t}{2}}^{-\frac{\alpha}{d_w}}\norm{Q^{(k)}_{t/2} f}_{L^p_M}}^q \frac{\mathrm dt}{t}}^\frac{1}{q} \pr{\int_0^1 \pr{\pr{\frac{t}{2}}^{\frac{\alpha}{d_w}}\norm{Q^{(k)}_{t/2} g}_{L^{p'}_M}}^{q'} \frac{\mathrm dt}{t}}^\frac{1}{q'} \\
  &\lesssim \norm{f}_{\mathcal B_{p,q}^\alpha} \norm{g}_{\mathcal B^{-\alpha}_{p',q'}}.
\end{align*}
Using both estimates finishes the proof.
\end{proof}

\begin{lemma}[Interpolation]\label{lem: Besov Interpolation}
Let $p,q\in[1,+\infty]$ and $\alpha_0,\alpha_1\in \R$. For $\eta\in (0,1)$, define $\alpha_\eta=\eta \alpha_0+(1-\eta)\alpha_1$. Then, we have $$\norm{f}_{\mathcal B^{\alpha_\eta}_{p_\eta,q_\eta}}\le \norm{f}_{\mathcal B^{\alpha_0}_{p_0,q_0}}^{\eta}\norm{f}_{\mathcal B^{\alpha_1}_{p_1,q_1}}^{1-\eta}$$
where $p_0,p_1,q_0,q_1\in [1,+\infty]$ satisfy $$\frac{1}{p_\eta}=\frac{\eta}{p_0}+\frac{1-\eta}{p_1},\quad \frac{1}{q_\eta}=\frac{\eta}{q_0}+\frac{1-\eta}{q_1}.$$
\end{lemma}

\begin{proof}%[Proof of Lemma \ref{lem: Besov Interpolation}]
Let $k >0$ be sufficiently large and let $z = 1+ \frac{(1-\eta)q_0}{\eta q_1}$, then $\eta z q_\eta = q_0$ and $\frac{z}{z-1}(1-\eta)q_\eta = q_1$. By the $L^p$ interpolation inequality and H\"{o}lder's inequality, we have for any $\eta \in (0,1)$
\begin{align*}
  \left(\int_0^1 \left[t^{-\frac{\alpha_\eta}{d_w}}\norm{Q^{(k)}_t f}_{L^{p_\eta}_M} \right]^{q_\eta} \frac{\mathrm dt}{t}\right)^{\frac{1}{q_\eta}} & \le \left(\int_0^1 \left[t^{-\frac{\alpha_0 \eta}{d_w}}\norm{Q^{(k)}_t f}_{L^{ p_0}_M}^\eta t^{-\frac{\alpha_1 (1-\eta)}{d_w}}\norm{Q^{(k)}_t f}_{L^{ p_1}_M}^{1-\eta} \right]^{q_\eta} \frac{\mathrm dt}{t}\right)^{\frac{1}{q_\eta}}\\
  &\le \left(\int_0^1 t^{-\frac{\alpha_0\eta z q_\eta}{d_w}}\norm{Q_tf}_{L^{p_0}_M}^{\eta q_\eta z}  \frac{\mathrm dt}{t}\right)^\frac{1}{zq_\eta}\\
  &\quad\quad\quad \times \left( \int_0^1 t^{-\frac{(1-\eta)\alpha_1 z q_\eta}{d_w (z-1)}}\norm{Q^{(k)}_t f}_{L^{p_1}_M}^{(1-\eta)q_\eta \frac{z}{z-1}}\right)^\frac{z-1}{z q_\eta}\\
  &=\left(\int_0^1 \left[t^{-\frac{\alpha_0 }{d_w}}\norm{Q^{(k)}_tf}_{L^{p_0}_M} \right]^{q_0}\frac{\mathrm dt}{t} \right)^\frac{\eta}{q_0}\\
  &\quad\quad\quad\times \left(\int_0^1 \left[t^{-\frac{\alpha_1 }{d_w}}\norm{Q^{(k)}_tf}_{L^{p_1}_M} \right]^{q_1}\frac{\mathrm dt}{t} \right)^\frac{1-\eta}{q_1}.
\end{align*}
This implies
\begin{align*}
  \norm{f}_{\mathcal B^{\alpha_\eta}_{p_\eta,q_\eta}} &= \norm{P^{(k)}_1 f}_{L^{p_\eta}_M} + \left(\int_0^1 \left[t^{-\frac{\alpha_\eta}{d_w}}\norm{Q^{(k)}_t f}_{L^{p_\eta}_M} \right]^{q_\eta} \frac{\mathrm dt}{t}\right)^{\frac{1}{q_\eta}} \\
  &\le \norm{f}_{\mathcal B^{\alpha_0}_{p_0,q_0}}^\eta \norm{f}_{\mathcal B^{\alpha_1}_{p_1,q_1}}^{1-\eta}.
\end{align*}

\end{proof}

\subsection{ Embedding Theorems}

We now collect several standard and useful embedding properties of Besov spaces. 
\begin{prop}{\cite[Section 4.2]{liu2016besov}}\label{prop:besov_embed}
Let $\alpha \in \R$ and $p,\, q\in (0,\infty]$,  then $$\mathcal B_{p,q}^\alpha \hookrightarrow \mathcal B^{\alpha - d_h/p}_{\infty}.$$
In addition, for any $\alpha \in \R$, $1 \le p \le q \le +\infty$, the following inequality holds uniformly in $f \in \mathcal B_{p,p,}^\alpha$, $$\norm{f}_{\mathcal B^{\alpha}_{p,p}}\lesssim \norm{f}_{\mathcal B^{\alpha}_{q,q}}.$$
\end{prop}

\begin{lemma}\label{lem: besov embedding q to infinity }
Let $\alpha \in \R$ and $p,q \in [1,\infty]$, then the following inequality holds uniformly for $f \in \mathcal B_{p, q}^\alpha$
\begin{align*}
  \norm{f}_{\mathcal B^\alpha_{p,\infty}}\lesssim\norm{f}_{\mathcal B^\alpha_{p,q}} .
\end{align*}
\end{lemma}
\begin{proof}
Suppose  $\alpha \in \R$ and let $k \in \mathbb Z$  so that $k > \alpha/d_w$. By Theorem \ref{thm:crf}, we see
\begin{align*}
  \sup_{t \in (0,1]} t^{-\frac{\alpha}{d_w}}\norm{Q^{(k)}_tf}_{L^p_M} \le \sup_{t \in (0,1]} t^{-\frac{\alpha}{d_w}}\norm{Q^{(k)}_tP^{(2k)}_1 f}_{L^p_M}+\sup_{t \in (0,1]} t^{-\frac{\alpha}{d_w}} \int_0^1 \norm{Q^{(k)}_tQ^{(2k)}_s f}_{L^p_M} \frac{\mathrm ds}{s}.
\end{align*}
We consider the first term on the right hand side. If $\alpha \le 0$, then
\begin{align*}
  \sup_{t \in (0,1]} t^{-\frac{\alpha}{d_w}}\norm{Q^{(k)}_tP^{(2k)}_1 f}_{L^p_M} \lesssim \norm{P_{1/2} f}_{L^p_M}.
\end{align*}
On the other hand, if $\alpha > 0$, then by Lemma \ref{lem:q_regulariz} that
\begin{align*}
  \sup_{t \in (0,1]} t^{-\frac{\alpha}{d_w}}\norm{Q^{(k)}_tP^{(2k)}_1 f}_{L^p_M} \lesssim \norm{P_{1/2} f}_{\mathcal B^0_{p,\infty}} \lesssim \norm{P_{1/2} f}_{L^p_M}.
\end{align*}

For the second term, consider
\begin{align*}
  \int_0^1 \norm{Q^{(k)}_tQ^{(2k)}_s f}_{L^p_M} \frac{\mathrm ds}{s}&= 2^{2k}\int_0^1 \norm{Q_t^{(k)} Q^{(k)}_{s/2}Q^{(k)}_{s/2}f}_{L^p_M} \frac{\mathrm ds}{s}\\
  &\overset{\eqref{ineq:almost orthogonal}}{\lesssim} \int_0^1 \frac{\min\{t,\frac{s}{2}\}^k}{\max\{t,\frac{s}{2}\}^k}\norm{ Q_{s/2}^{(k)} f}_{L^p_M} \frac{\mathrm ds}{s}\\
  &= \pr{\int_0^t + \int_t^1} \frac{\min\{t,\frac{s}{2}\}^k}{\max\{t,\frac{s}{2}\}^k}\norm{ Q_{s/2}^{(k)} f}_{L^p_M} \frac{\mathrm ds}{s}.
\end{align*}
Let us fix a pair of $q,\, q' \in [1,\infty]$ so that $\frac{1}{q} + \frac{1}{q'} = 1$, and  consider first for $\frac{s}{2} \le t$,
\begin{align*}
  \int_0^t\frac{\min\{t,\frac{s}{2}\}^k}{\max\{t,\frac{s}{2}\}^k}\norm{ Q_{s/2}^{(k)} f}_{L^p_M} \frac{\mathrm ds}{s}&\le \int_0^t\pr{\frac{s/2}{t}}^k \norm{Q_{s/2}^{(k)}f}_{L^p_M} \frac{\mathrm ds}{s}\\
  &=t^\frac{\alpha}{d_w}\int_0^t \pr{\frac{s/2}{t}}^{k+\frac{\alpha}{d_w}} s^{-\frac{\alpha}{d_w}} \norm{Q^{(k)}_{s/2}f}_{L^p_M} \frac{\mathrm ds}{s}\\
  &\le t^\frac{\alpha}{d_w} \pr{\int_0^t \pr{\frac{s/2}{t}}^{q'k+q'\frac{\alpha}{d_w}}\frac{\mathrm ds}{s}}^\frac{1}{q'} \pr{\int_0^t  \pr{s^{-\frac{\alpha}{d_w}}\norm{Q_{s/2}^{(k)}}_{L^p_M}}^q\frac{\mathrm ds}{s}}^\frac{1}{q}\\
  &\lesssim t^{\frac{\alpha}{d_w}}\pr{\int_0^t \pr{s^{-\frac{\alpha}{d_w}}\norm{Q_{s/2}^{(k)}}_{L^p_M}}^q\frac{\mathrm ds}{s}}^\frac{1}{q},
\end{align*}
where we used H\"{o}lder's inequality in the third line and the fact that $k +\frac{\alpha}{d_w} > 0$ in the last line.

Similarly, consider for $\frac{s}{2}\ge t$,
\begin{align*}
  \int_t^1 \frac{\min\{t,\frac{s}{2}\}^k}{\max\{t,\frac{s}{2}\}^k}\norm{ Q_{s/2}^{(k)} f}_{L^p_M} \frac{\mathrm ds}{s}&= t^{\frac{\alpha}{d_w}}\int_t^1 \pr{\frac{t}{s/2}}^{k - \frac{\alpha}{d_w}} (s/2)^{-\frac{\alpha}{d_w}}\norm{Q_{s/2}^{(k)}f}_{L^p_M} \frac{\mathrm ds}{s}\\
  &\le t^\frac{\alpha}{d_w} \pr{\int_t^1 \pr{\frac{t}{s/2}}^{q'k-q'\frac{\alpha}{d_w}}\frac{\mathrm ds}{s}}^\frac{1}{q'} \pr{\int_t^1\pr{s^{-\frac{\alpha}{d_w}} \norm{Q^{(k)}_{s/2} f}_{L^p_M}}^q \frac{\mathrm ds}{s}}^\frac{1}{q}\\
  &\lesssim t^\frac{\alpha}{d_w} \pr{\int_t^1\pr{s^{-\frac{\alpha}{d_w}} \norm{Q^{(k)}_{s/2} f}_{L^p_M}}^q \frac{\mathrm ds}{s}}^\frac{1}{q}.
\end{align*}
Therefore,
\begin{align*}
  \sup_{t \in (0,1]} t^{-\frac{\alpha}{d_w}} \int_0^1 \norm{Q^{(k)}_tQ^{(2k)}_s f}_{L^p_M} \frac{\mathrm ds}{s} \lesssim \pr{\int_0^1 \pr{s^{- \frac{\alpha}{d_w}}\norm{Q_{s/2}^{(k)} f}_{L^p_M}}^q \frac{\mathrm ds}{s}}^\frac{1}{q} + \norm{P_{1/2} f}_{L^p_M},
\end{align*}
which implies the desired result.
\end{proof}

From the definition of the Besov norm in Definition \ref{def:besov_space}, it is clear $\mathcal B^\alpha_{p,q} \hookrightarrow \mathcal B^\beta_{p,q}$ for $\beta > \alpha$. When $M=\T^n$, this embedding is compact. The corresponding compact embedding result is not known in our setting, and the following theorem fills this gap.
\begin{thm}\label{thm:compact embedding}
For $p,q \in [1,\infty]$ and $\alpha \in \R$, the embedding $\mathcal B^\beta_{p,q} \hookrightarrow \mathcal B^\alpha_{p,q}$ for $\beta >\alpha$ is compact.
\end{thm}
\begin{proof}
It is enough to show that any bounded sequence in $\mathcal B^\beta_{p,q}$ admits a Cauchy subsequence in $\mathcal B^\alpha_{p,q}$. Take $J \in \mathbb N$ and recall operator $A_j$ from Proposition \ref{prop:dyadic norm} for $j \in \mathbb N$.  If $q = \infty$, then by Proposition \ref{prop:dyadic norm}, it holds for any $f\in \mathcal B^\beta_{p,\infty}$ that
\begin{equation}\label{ineq:compact sequential infty}
  \begin{aligned}
    \sup_{j \ge J} 2^{\alpha j} \norm{A_j(f)}_{L^p_M} &= \sup_{j \ge J} 2^{j\beta} 2^{j(\alpha-\beta)} \norm{A_j(f) }_{L^p_M}\\
    &\le  2^{J(\alpha-\beta)} \sup_{j \ge 0} 2^{j\beta} \norm{A_j(f)}_{L^p_M} \\
    &\lesssim 2^{J(\alpha -\beta)} \norm{f}_{\mathcal B^{\beta}_{p,\infty}}.
  \end{aligned}
\end{equation}
Similarly, if $q <\infty$, then for any $f \in \mathcal B^\beta_{p,q}$, we have
\begin{equation}\label{ineq:compact sequential p}
  \begin{aligned}
    \sum_{j \ge J} 2^{\alpha jq} \pr{\norm{A_j f}_{L^p_M}}^q& = \sum_{j \ge J} 2^{j(\alpha - \beta)q}2^{\beta jq} \pr{\norm{A_j f}_{L^p_M}}^q \\
    &\le 2^{J(\alpha - \beta)q} \sum_{j \ge J}2^{\beta jq} \norm{A_jf}_{L^p_M}^q\\
    &\lesssim 2^{J(\alpha - \beta)q} \norm{f}_{\mathcal B^\beta_{p,q}}^q.
  \end{aligned}
\end{equation}
Now, take any bounded sequence $\{f_n\}_n\subset \mathcal B^\beta_{p,q}$, which we assume without loss of generality that $$\sup_{n \in \mathbb N} \norm{f_n}_{\mathcal B^\beta_{p,q}} \le 1.$$

By Proposition \ref{prop:besov_embed}, we see $\norm{f}_{\mathcal B^{\beta - d_h/p}_{\infty}}< \infty$, hence by Lemma \ref{lem:q_regulariz}, $\norm{P_t f}_{\mathcal B^a_{\infty} }< \infty$ for any $a >0$ and $t >0$. Recall $\Theta >0$ in \eqref{ineq:holder_kernel}. By proposition \ref{prop:holder=besov}, we see for some $0 < \delta < \Theta$,
\begin{align*}
  \sup_{n \in \mathbb N} \norm{Q^{(k)}_tf_n}_{\mathcal C^\delta_M}  \asymp \sup_{n \in \mathbb N}\norm{Q^{(k)}_tf_n}_{\mathcal B^\delta_{\infty}} < \infty.
\end{align*}
Hence, by Arzela-Ascoli (c.f. \cite[Theorem 4.43]{folland1999real}), for each $j \in \mathbb N$, the sequence $\{A_j(f_n)\}_{n \in \mathbb N}$ admits a convergent subsequence in $\mathcal C_M$, hence in $L^p_M$ for any $p \in [1,\infty]$.

Therefore, one may use the diagonalization argument to see there exists a subsequence $\{g_n\}_{n \in \mathbb N} \subset \{ f_n\}_{n \in \mathbb N}$ so that $A_j(g_n)$ converges in $L^p_M$ (in $\mathcal C_b(M)$ if $p = \infty$) for each $j \in \mathbb N$. In addition,  by either \eqref{ineq:compact sequential infty} or \eqref{ineq:compact sequential p}, for any small $\varepsilon >0$, we may take $J \ge 1$ to be sufficiently large so that $2^{J(\alpha - \beta)} < \varepsilon$, since $\alpha < \beta$. Then for $q \le \infty$, we have
\begin{align*}
  \norm{g_n - g_m}_{\mathcal B^\alpha_{p,q}} &\lesssim \max_{0 \le j \le J} 2^{\alpha j}\norm{A_j(g_n )- A_j(g_m)}_{L^p_M} + \sup_{n,m \in \mathbb N}\pr{\sum_{j \ge J+1} 2^{\alpha j q} \norm{A_j(g_n) - A_j(g_m)}_{L^p_M}^q}^\frac{1}{q}\\
  &\lesssim \max_{0 \le j \le J} 2^{\alpha j}\norm{A_j(g_n )- A_j(g_m)}_{L^p_M}+2^{J(\alpha - \beta)}2\sup_{n\in \mathbb N}  \norm{g_n}_{\mathcal B^\beta_{p,q}}\\
  &\lesssim \max_{0 \le j \le J} \norm{A_j(g_n) - A_j(g_m)}_{L^p_M} + \varepsilon.
\end{align*}
Now take $n,m \to \infty$ to see the first term on the last line goes to zero, hence
\begin{align*}
  \lim_{n,m\to \infty} \norm{g_n - g_m}_{\mathcal B^\alpha_{p,q}} \lesssim \varepsilon,
\end{align*}
for any $\varepsilon >0$. In addition, we note that the hidden multiplicative constant is from Proposition \ref{prop:dyadic norm} and it does not depend on $J$. This implies  the subsequence $\{g_n\}_{\mathbb N}$ is Cauchy in $\mathcal B^\alpha_{p,q}$, as desired.
\end{proof}

\subsection{Connection to the Dirichlet energy}
Here we show a connection between the heat semi-group based Besov norm and the Dirichlet form $\mathcal{E}$. It will be useful for obtaining a global solution to \eqref{eq:phi n formal} in the presence of a singular energy measure. Let us first introduce a different notion of Besov (semi)norm associated with the heat kernel, which was developed in \cite{UConn1}.
\begin{definition}\label{def:Baud besov norm}
For $p \ge 1$ and $\alpha \ge 0$, let $f$ be a measurable function on $M$,  define %the following semi-norm
\begin{align*}
  \norm{f}_{p,\alpha} := \sup_{t>0} t^{-\alpha} \pr{\iint_{M^2} \abs{f(x) - f(y)}^p p_t(x,y) \mu(\mathrm dy) \mathrm (\mathrm dx)}^\frac{1}{p}.
\end{align*}
\end{definition}
\begin{lemma}\label{lem: compare our besov and baudoin besov}
Let $\alpha \ge 0$ and $p \ge 1$, for any measurable function $f: M \to \R$, we have
\begin{align}\label{ineq: compare our besov and baudoin besov}
  \norm{f}_{\dot{\mathcal B}_{p,\infty}^\alpha} \lesssim \norm{f}_{p,\frac{\alpha}{d_w}},
\end{align}
whenever the right hand side is finite.
\end{lemma}
\begin{proof}
For $k \in \mathbb N$ and $t \in (0,1]$, recall $q_{k,t}: M^2 \to \R $ from  Proposition \ref{prop:kernel_bound_dt}, and let $k > \alpha$. Then for any $p \ge 1$, there is some $c>0$ in dependent of $t \in (0,1]$ so that
\begin{align*}
  \norm{Q_t^{(k)}f}_{L^p_M} &= \pr{\int_M \left|\int_M q_{k,t}(x,y) f(y) \mu(\mathrm dy)\right|^p \mu(\mathrm dx)}^\frac{1}{p}\\
  &\overset{\eqref{eq:locality integral q equals 0}}{=}\pr{\int_M \left|\int_M q_{k,t}(x,y) \pr{f(y) - f(x)} \mu(\mathrm dy) \right|^p\mu(\mathrm dx)}^\frac{1}{p}\\
  &\overset{\eqref{ineq: q sub gaussian bound}}{\lesssim} \pr{\int_M \left[\int_M p_{ct}(x,y) \abs{f(x) - f(y)}\mu(\mathrm dy) \right]^p\mu(\mathrm dx)}^\frac{1}{p}\\
  &\le \pr{\iint_{M^2} p_{ct}(x,y) \abs{f(x) - f(y)}^p \mu(\mathrm dy) \mu(\mathrm dx)}^\frac{1}{p},
\end{align*}
where we used Jensen's inequality in the last line. Therefore, multiplying both side by $t^{-\frac{\alpha}{d_w}}$ and take supremum to see the desired result.
\begin{comment}
    , we get
  \begin{align*}
    \norm{f}_{\dot{\mathcal B}_{p,\infty}^\alpha}&=\sup_{t \in (0,1]}t^{- \frac{\alpha}{d_w}}\norm{Q^{(k)}_t f}_{L^p_M}\\
    &\lesssim \sup_{t>0} (ct)^{-\frac{\alpha}{d_w}}\pr{\iint_{M^2} p_{ct}(x,y) \abs{f(x) - f(y)}^p \mu(\mathrm dy) \mu(\mathrm dx)}^\frac{1}{p}\\
    &= \norm{f}_{p,\frac{\alpha}{d_w}},
  \end{align*}
  which is the desired result.
\end{comment}
\end{proof}

\begin{remark}
    The reverse inequality in the above Lemma is not always true. For when they are true, we refer the reader to \cite[Theorem 1.5]{grigor2015heat} and \cite[Theorem 7.2]{baudoin2024korevaarschoensobolevspacescriticalexponents}.
\end{remark}

The following proposition is a slight modification of \cite[Proposition 4.7]{UConn1}, whose proof can be found in the Appendix. It has the advantage that the absolute value sign inside the  norm is removed.
\begin{prop}\label{prop: mod of baudoin berstein}
Suppose $n \in \mathbb N$, and $p \in \mathbb \N$ is odd and $p > 2 - n$. Then for any $\alpha \ge 0$, we have
\begin{align}\label{ineq: mod of baudoin berstein}
  \norm{f^{2p+n-2}}_{1,\alpha} \le 2 \frac{2p+n-2}{p} \norm{f^{2(p+n - 2)}}_{L^1_M} ^\frac{1}{2}\norm{f^p}_{2,2\alpha}^\frac{1}{2},
\end{align}
whenever the right hand side is finite. In particular,
\begin{align}\label{ineq: mod of baudoin berstein and d form}
  \norm{f^{2p+n - 2}}_{1,\frac{1}{2}} \le \sqrt{8} \frac{2p+n - 2}{p} \norm{f^{2(p+n-2)}}_{L^1_M}^\frac{1}{2}  \mathcal E \left(f^p,f^p \right)^\frac{1}{2},
\end{align}
whenever the right hand side is finite.
\end{prop}

Recall in the Euclidean setting, many useful inequalities that connects the Besov norms and Sobelev norms are results of Berstein-type inequality. Under the current setting, it is unclear how one might obtain a Berstein-type inequality, since the energy measure who plays the role of derivative is singular with respect to the reference measure $\mu$. Nonetheless, one can still make connections between the Besov and Sobelev norm via the following proposition.

\begin{prop}\label{lemma: Chain rule workaround}
Suppose $n \in \mathbb N$, and $p \in \mathbb \N$ is odd and $p \ge  2 - n$. Then
we have
$$\norm{f^{n+2p-2}}_{\dot{\mathcal B}^{d_w/2}_{1,\infty}}\lesssim \norm{f^{2p+2n-4}}_{L^1_M}^{\frac{1}{2}}\mathcal{E}(f^{p},f^{p})^{\frac{1}{2}},$$
whenever the right hand side is finite.
\end{prop}
\begin{proof}%[Proof of Proposition \ref{lemma: Chain rule workaround}]
Let $p,\, n \in \N$ and assume $p$ is an odd number and $p \ge n-2$, then by Lemma \ref{lem: compare our besov and baudoin besov} and  Proposition \ref{prop: mod of baudoin berstein}, we see
\begin{align*}
  \norm{f^{2p+n - 2}}_{\dot{\mathcal B}_{1,\infty}^\frac{d_w}{2}} &\overset{\eqref{ineq: compare our besov and baudoin besov}}{\lesssim}\norm{f^{2p+n-2}}_{1, \frac{1}{2}}\overset{\eqref{ineq: mod of baudoin berstein and d form}}{\lesssim} \norm{f^{2(p+n - 2)}}_{L^1_M}^\frac{1}{2} \mathcal E\left(f^p, f^p \right)^\frac{1}{2},
\end{align*}
which is the desired result.
\end{proof}
\noindent The following one direction relation between Besov spaces and Sobelev spaces is useful.
\begin{lemma}\label{lem:half dw and Dirichlet energy}
$\mathcal B^\frac{d_w}{2}_{2,\infty} \subset \mathcal F \cap L^2(M,m)$ and
\begin{align*}
  \norm{f}_{\dot{\mathcal B}^\frac{d_w}{2}_{2,\infty}} \lesssim  \sqrt{\mathcal E(f,f) }.
\end{align*}
\end{lemma}
\begin{comment}
\begin{remark}
    Note that the right hand side of the inequality in Lemma \ref{lem:half dw and Dirichlet energy} is precisely the Dirichlet energy (or the Sobelev norm) of $f$. However, we do not know if the reverse inequality holds, and it is natural to inquire if it does.
\end{remark}
\end{comment}

\begin{proof}%[Proof of Lemma \ref{lem:half dw and Dirichlet energy}]
By Lemma \ref{lem: compare our besov and baudoin besov} and \cite[Proposition 4.6]{UConn1}, we see
\begin{align*}
  \norm{f}_{\dot{\mathcal B}_{2,\infty}^\frac{d_w}{2}} \lesssim  \norm{f}_{2, \frac{d_w}{2}} \lesssim \sqrt{ \mathcal E(f,f)}.
\end{align*}
\end{proof}

\noindent The following lemma is a well-known result for strongly local Dirichlet forms, which can be found in \cite[Lemma 3.2.5 and Theorem 3.2.2]{FukushimaDF}.
\begin{prop}
For $p\geq 2$, if $f\in \mathcal{F}\cap\, \mathcal C_M$, then $f^p\in \mathcal{F}$, and
$$\mathcal{E}(f^p,f^p)=(p-1)^2\int_M f^{2(p-1)}d\Gamma(f,f).$$
\end{prop}

\noindent We also have the following Hardy-Littlewood-Sobolev inequality.

\begin{lemma}\label{lem: HLS}
Suppose $M$ is compact and the heat kernel satisfies $p_t(x,y)\lesssim t^{-\beta/2}$ for all $t\in (0,1],x,y\in M$ with $\beta\geq 2$.
\begin{itemize}
  \item If $\beta=2$, then for every $1\le q<+\infty$, we have for all $f\in \mathcal{F}$ $$\norm{f}_{L^q(\mu)}\lesssim \mathcal{E}(f,f)^{\frac{1}{2}}+\norm{f}_{L^2_M}.$$
  \item For $\beta>2$, the above estimate holds for all $q\leq q'=\frac{2\beta}{\beta-2}$.
\end{itemize}
\end{lemma}
\begin{proof}
The second item follows from \cite{VAROPOULOS1985HLS}, the Nash inequality, and the compactness hypothesis. See also \cite[Theorem 1.1]{UConn1} for a more general result. The first item follows from the second item by applying the obvious inequality $t^{-1}\leq t^{-\frac{2+\varepsilon}{2}}$ for all $t\in(0,1]$ and then making $\varepsilon$ small enough so that $q\leq \frac{4}{\varepsilon}<\frac{2(2+\varepsilon)}{\varepsilon}$ for any arbitrarily large $q$.
\end{proof}

\section{Multiplication of distributions and Schauder estimates}\label{sec:pp_estimates}

\subsection{Multiplicative inequality}

In this section, we construct the product decomposition and prove Theorem \ref{thm:besov_product}.
\begin{comment}
\begin{thm}[Paraproduct Estimate]\label{thm:besov_product}
There exists a decomposition of product $f \cdot g$ that agrees with the point-wise multiplication when $f,g \in \mathcal C_M$, so that 
$\alpha>0$,  $\beta \in \left(-\Theta, \Theta\right)\backslash\{0\}$ and  $\alpha+ \beta>0$
\begin{align*}
  \norm{f\cdot g}_{\mathcal B^{\alpha \wedge \beta}_{\infty}} \lesssim \norm{f}_{\mathcal B^\alpha_{\infty}} \cdot \norm{g}_{\mathcal B^\beta_{\infty}},
\end{align*}
 uniformly in $f \in \mathcal  B^\alpha_{\infty},\, g \in \mathcal B^\beta_{\infty}$
\end{thm}
\end{comment}
\noindent Our overall strategy follows \cite[Section 3]{bailleul2016heat}, but the singularity of the energy measure in our setting forces %substantial
certain modifications. We discuss these first, both to motivate our choice of product decomposition and to indicate where new ingredients are required.

When the heat kernel satisfies \eqref{ineq:sgu} with $d_w = 2$ and $\Theta = 1$ in \eqref{ineq:holder_kernel} --- the framework of \cite{bailleul2016heat} --- it is classical that $\mathcal{D}(L)$ forms an algebra and the energy measure formally defined as 
$$
\Gamma(f,g) := -\tfrac{1}{2}\bigl(L(fg) - fL(g) - gL(f)\bigr),
\qquad f,g \in \mathcal{D}(L),
$$
is absolutely continuous with respect to $\mu$, with Radon--Nikodym density, i.e. the \textit{carr\'e du champ operator}. In that regime, the product estimates of \cite{bailleul2016heat} rest on a point-wise bound of the form
\begin{align}\label{ineq:BB core restriction}
    t\,\Gamma(P_tf, P_tf)(x) 
    \;\lesssim\; \int_M k_t(x,y)\,|f(y)|^2\,\mu(\mathrm{d}y),
\end{align}
where $k_t$ satisfies Gaussian upper estimates.

Under Assumption \ref{A:main}, $d_w$ may exceed $2$ and $\Gamma(f,f)$ need not be absolutely continuous with respect to $\mu$, so \eqref{ineq:BB core restriction} is no longer available as a starting point. In fact, by \cite{kajino2020singularity}, $d_w>2$ implies $\mathrm d\Gamma(f,f)$ is singular with respect to the reference measure $\mu$ on $M$. We instead exploit a feature internal to the decomposition: in the resonant term \eqref{eq:resonant term}, every occurrence of $\Gamma$ is further composed with $P$- and $Q$-operators. This additional smoothing --- which plays no role in the argument of \cite{bailleul2016heat} --- already suffices to control the 
resonant term directly at the level of the measure $\Gamma(P_tf,P_tf)$ via a Caccioppoli-type estimate (Theorem \ref{thm:gamma_infty_estimates}), bypassing the need for a point-wise carr\'e du champ operator.

\noindent The following results will be used to justify the agreement of our product decomposition with point-wise multiplication. 

\begin{lemma}\label{lem: Dform to Qop}
	For any $g \in \mathcal F \cap \mathcal C_M$, $t>0$ and $n = 0,1,2,\dots$, we have the following point-wise identity,
	\begin{align}
		t\mathcal E(q_{n,t}(x,\cdot), g) = Q^{(n+1)}_t g(x),\qquad x \in M.
	\end{align}
\end{lemma}
\begin{proof}
	Recall the formula
	\begin{align}
		2\mathcal E(f,g) = \lim_{\varepsilon \downarrow 0}\frac{1}{\varepsilon} \iint_{M^2} \pr{f(y) - f(z)}\pr{g(y) - g(z)}p_\varepsilon(y,z) \mu(\mathrm dy)\mu(\mathrm dx). 
	\end{align}
	Now consider by Fubini's and semi-group properties that
	\begin{align*}
		&\varepsilon^{-1}\iint_{M^2} \pr{q_{n,t}(x,y) - q_{n,t}(x,z)}\pr{g(y) - g(z)}p_\varepsilon(y,z) \mu(\mathrm dy)\mu(\mathrm dz)\\
		=&\varepsilon^{-1}\iint_{M^2}q_{n,t}(x,y)g(y)-q_{n,t}(x,y) g(z) - q_{n,t}(x,z)g(y)+q_{n,t}(x,z)g(z) p_\varepsilon(y,z)\mu(\mathrm dy) \mu(\mathrm dz)\\
		=&\varepsilon^{-1}\iint_{M^2} q_{n,t}(x,y) g(y) p_\varepsilon(y,z) \mu(\mathrm dz) \mu(\mathrm dy)-\iint_{M^2} q_{n,t}(x,y) g(z) p_\varepsilon(y,z) \mu(\mathrm dy)\mu(\mathrm dz)  \\
		&\qquad +\varepsilon^{-1}\iint_{M^2} q_{n,t}(x,z) g(z) p_\varepsilon(y,z)\mu(\mathrm dy) \mu(\mathrm dz) -\iint_{M^2} q_{n,t}(x,z) g(y) p_\varepsilon(y,z) \mu(\mathrm dz) \mu(\mathrm d y)\\
		=&2\varepsilon^{-1}\pr{(tL)^n e^{-tL} g(x)- e^{-\varepsilon L} (tL)^n e^{-tL} g(x)}.
	\end{align*}
	Sending $\varepsilon \downarrow 0$ to see the above expression converges in sup norm to $L Q^{(n)}_t g(x)$. Therefore, for $x \in M$,
	\begin{align*}
		t\mathcal E(q_{n,t}(x,\cdot), g) &= \lim_{\varepsilon \downarrow 0}\varepsilon^{-1}\pr{(tL)^n e^{-tL} g(x)- e^{-\varepsilon L} (tL)^n e^{-tL} g(x)}\\
		&=tLQ^{(n)}_tg(x)\\
		&= Q^{(n+1)}_t g(x),
	\end{align*}
    which is the desired result.
\end{proof}
\noindent The following corollary is a direct consequence of Lemma \ref{lem: Dform to Qop} and \eqref{eq:energy measure def}.
\begin{corollary}\label{cor:op applied to energy measure point wise}
	For $f,g \in \mathcal F \cap \mathcal C_M$ and $k \in \Z_+$, the following point-wise in $(t,x) \in (0,\infty) \times M$,
	\begin{align*}
		Q^{(k)}_t\Gamma(f,g)(x):= \int_M q_{k,t}(x,\cdot)\mathrm d\Gamma(f,g) = \mathcal E(fq_{k,t}(x,\cdot),g)+\mathcal E\pr{gq_{k,t}(x,\cdot),f} - \mathcal E(fg, q_{k,t}(x,\cdot)).
	\end{align*}
	Furthermore, if in addition $f,g \in \mathcal D_2(L)$, then 
	\begin{align*}
		Q^{(k)}_t\Gamma\pr{f,g}(x)&= \int_M fq_{k,t}(x,\cdot) \pr{Lg}\mathrm d\mu +\int_M gq_{k,t}(x,\cdot) \pr{Lf}\mathrm d\mu - t^{-1}Q_t^{(k+1)}\pr{f\cdot g}(x)\\
        &= Q^{(k)}_t \pr{f \cdot \pr{Lg}}(x)+Q^{(k)}_t \pr{g \cdot \pr{Lf}}(x) - t^{-1}Q_t^{(k+1)}\pr{f\cdot g}(x).
	\end{align*}
\end{corollary}

\subsubsection{Product Decomposition}\label{sec:sub_para_decomposition}
We follow strictly the decomposition in \cite{bailleul2016heat}. %However, in each step of the decomposition, we give justifications for point-wise equality.
Given two functions $f,g\in \mathcal C_M$, denote
\begin{align*}
\Delta_{-1}^{(b)}(f,g):= P_1^{(b)}\left(P_1^{(b)}f \cdot P^{(b)}_1 g \right).
\end{align*}
By the Calderon's reproducing formula (Theorem \ref{thm:crf}), the following equality holds in $\mathcal C_M$ for $b \in \N$:
\begin{align*}
fg &= \lim_{t \to 0} P^{(b)}_t \left(P^{(b)}_tf\cdot P^{(b)}_tg \right) = - \int_0^1 \partial_t \left( P^{(b)}_t \left(P^{(b)}_tf\cdot P^{(b)}_tg \right)\right)\frac{\mathrm dt}{t}+\Delta_{-1} (f,g)\\
&= \frac{1}{\gamma_b}\int_0^1 P_t^{(b)}\left( Q^{(b)}_tf \cdot P^{(b)}_t g\right) + P^{(b)}_t \left(P^{(b)}_t f \cdot Q^{(b)}_tg \right)+ Q^{(b)}_t\left(P_t^{(b)} f\cdot P^{(b)}_t g \right) \frac{\mathrm dt}{t} \\
&\quad +\Delta_{-1}(f,g)\\
&:=\frac{1}{\gamma_b }\int_0^1 (1) + (2)+ (3)\frac{\mathrm dt}{t} +  \Delta_{-1}^{(b)} \pr{f,g}.
\end{align*}
By Corollary \ref{cor:op applied to energy measure point wise}, the following equality holds point-wise:
\begin{align*}
(3) =  Q^{(b)}_t \left(P_t^{(b)}f \cdot P_t^{(b)}g \right) &= Q^{(b - 1)}_t\left((tL) P_t^{(b)} f\cdot P_t^{(b)} g\right) +  Q^{(b - 1)}_t\left(P_t^{(b)} f\cdot (tL) P_t^{(b)} g\right)\\
&\quad\quad\quad -2  tQ^{(b- 1)}_t\Gamma\left(P_t^{(b)} f,P_t^{(b)} g \right) \\
&:=B_g^{(b)}(f) + B_f^{(b)}(g) +R^{(b)}(f,g).
\end{align*}
Thus, we can decompose the product by
\begin{align*}
fg = \int_0^1\left\{(1) +B_g^{(b)}(f) \right\}+\left\{(2)+ B_f^{(b)}(g)\right\} + R^{(b)}(f,g) \frac{\mathrm dt}{t} +  \Delta_{-1}^{(b)}(f,g).
\end{align*}
By Corollary \ref{cor:op applied to energy measure point wise} again, we have the point-wise identity
\begin{align*}
(1) = P^{(b)}_t\left(Q_t^{(b)}f \cdot P^{(b)}_t g\right) &= P_t^{(b)}\left( tL Q^{(b - 1)}_t f \cdot P^{(b)}_t g\right)\\
&=\left\{2P^{(b)}_t t \Gamma\left(Q_t^{(b-1)}f , P_t^{(b)}g \right) -P^{(b)}_t\left(Q_t^{(b-1)}f \cdot tL \left( P_t^{(b)}g\right)\right)\right\}\\
&\quad \quad\quad +tL P^{(b)}_t \left( Q_t^{(b- 1)}f \cdot P^{(b)}_t g \right)\\
&:= S^{(b)}(f,g) + A_g^{(b)}(f)
\end{align*}
where $S^{(b)}(f,g)$ is the difference in brackets. Similarly,
\begin{align*}
(2) = A_f^{(b)}(g) + S^{(b)}(g,f).
\end{align*}

\begin{definition}\label{def:para_product}
Given an integer $b \ge 2$ and $f,g \in \mathcal C_M$, we define their \textbf{para-product}, parametrized by $b \in \N$, by the formula
\begin{align*}
  \Pi_g^{(b)}(f) &= \frac{1}{\gamma_b} \int_0^t \left\{ A_g^{(b)}(f)+B_g^{(b)}(f)\right\} \frac{\mathrm dt}{t}\\
  &=\frac{1}{\gamma_b} \int_0^1 tL P_t^{(b)}\left(Q_t^{(b-1)}f \cdot P_t^{(b)}g \right)+ Q^{(b - 1)}_t \left( tL P_t^{(b)}f \cdot P_t^{(b)} g\right)\frac{\mathrm dt}{t}.
\end{align*}
\end{definition}
We will prove the well-definition of the para-product by bounding its operator norm later. With this notation, Calderon's reproducing formula for $f,g \in \mathcal C_M$ becomes the point-wise identity
\begin{equation}\label{Calderon Paraproduct Formula}
fg = \Pi_g^{(b)}(f) + \Pi_f^{(b)}(g) +\Pi^{(b)}(f,g) + \Delta_{-1}^{(b)}(f,g),
\end{equation}
with the ``low-frequency tail"
\begin{align*}
\Delta_{-1}^{(b)}(f,g):= P_1^{(b)}\left(P^{(1)}_1f \cdot P_1^{(b)} g \right)
\end{align*}
and the ``resonant term"
\begin{equation}\label{eq:resonant term}
    \begin{aligned}
    \Pi^{(b)}(f,g)&:= \frac{1}{\gamma_b}\int_0^1 \left\{S(f,g)+S(g,f)+R(f,g) \right\} \frac{\mathrm dt}{t}\\
    &=\frac{1}{\gamma_b}\int_0^1\left\{2P^{(b)}_t t \Gamma\left(Q_t^{(b-1)}f , P_t^{(b)}g \right) -P^{(b)}_t\left(Q_t^{(b-1)}f \cdot tL \left( P_t^{(b)}g\right)\right) \right\} \frac{\mathrm dt}{t}\\
    &\quad + \frac{1}{\gamma_b}\int_0^1\left\{2P^{(b)}_t t \Gamma\left(Q_t^{(b-1)}g , P_t^{(b)}f \right) -P^{(b)}_t\left(Q_t^{(b-1)}g \cdot(tL) \left( P_t^{(b)}f\right)\right) \right\} \frac{\mathrm dt}{t}\\
    &\quad - \frac{2}{\gamma_b}\int_0^1 Q^{(b - 1)}_t t \Gamma\left( P^{(b)}_t f,P^{(b)}_t g\right)\frac{\mathrm dt}{t}.
    \end{aligned}
\end{equation}

\subsubsection{Energy measure and Besov norms}\label{sec: singular energy and implications}
In this subsection, we state a relation between energy measures and Besov norms, which will be essential for controlling the resonant term. 

\noindent We will need the following version of the Cacciopoli inequality, whose proof can be found in the Appendix. 

\begin{lemma}\label{lem:cacciopoli}
There exists $C_s>0$ so that for all $r>0$ and ball $B$ of radius $r$ and $f \in \mathcal D_2(L)$, 
\begin{align*}
    \Gamma(f,f)\pr{B} \le r^{d_w} \norm{Lf}_{L^2_{ B}}^2+ \frac{C_s}{r^{d_w}} \mathrm{osc}^2_{B}(f),
\end{align*}
where  
\begin{align*}
    \mathrm{osc}^2_{B}\pr{f}:= \norm{\frac{1}{\mu(B)} \int_B f(\cdot) - f(y)\mu(\mathrm dy) }_{L^2_B}^2.
\end{align*}
\end{lemma}

\begin{thm}\label{thm:gamma_infty_estimates}
    Assume $d_h +2\Theta- d_w>0$. Suppose $\alpha < \Theta$ and $U,V \in \{Q^{(k)} : k \in \Z_+\}$. 
    Then it holds uniformly in $f \in \Binf{\alpha}$ that 
    \begin{align*}
        t\,\norm{V_t \Gamma(U_t f, U_t f)}_{L^\infty_M} 
        \lesssim t^{\frac{2\alpha}{d_w}} \norm{f}_{\Binf{\alpha}}^2.
    \end{align*}
\end{thm}
\begin{comment}
{\issue (what are you trying to say with this remark? I don't understand. If you are comparing to the method that was discarded, the reader will not know...)\begin{remark}
        A further benefit of Theorem \ref{thm:gamma_infty_estimates} is 
    that it provides a single estimate valid for all H\"older regularity 
    exponents $\Theta \in (0,1]$ of the heat kernel. The argument of 
    \cite{bailleul2016heat} is tailored to the case $\Theta = 1$, and 
    auxiliary estimates are invoked to handle $\Theta < 1$. By 
    contrast, Theorem \ref{thm:gamma_infty_estimates} treats the full 
    range $\Theta \in (0,1]$, which is essential in the 
    sub-Gaussian setting where $\Theta < 1$ is the generic situation.
\end{remark}}
\end{comment}
\begin{proof}
    For simplicity, we will assume $U_t =V_t = P_t$ for $t >0$. The proof for the other cases are verbatim. For $(t,x) \in (0,\infty) \times M$ and $j \in \N$, denote $B^{t,x}_j:= B(x,2^{-j}t^\frac{1}{d_w})$. 

   \noindent Recall $\Phi$ in Assumption \ref{A:main}, by Lemma \ref{lem:cacciopoli}, we have
    \begin{align*}
        tP_t\Gamma\pr{P_tf,P_tf}(x) &= t \int_M p_t(x,y)\Gamma\pr{P_tf,P_tf}(\mathrm dy)\\
        &= t\sum_{j = -\infty}^\infty \int_{B^{t,x}_j\backslash B^{t,x}_{j+1}}p_t(x,y)\Gamma\pr{P_tf,P_tf}(\mathrm dy)\\
        &\overset{\eqref{ineq:sgu} }{\lesssim}  t\sum_{j = -\infty}^\infty t^{-\frac{d_h}{d_w}}\Phi\pr{2^{-j}}\Gamma\pr{P_tf,P_tf}\pr{B^{t,x}_j}\\
        &\lesssim t\sum_{j = -\infty}^\infty t^{-\frac{d_h}{d_w}} \pr{2^{-j} t^\frac{1}{d_w}}^{d_w}\Phi\pr{2^{-j}} \norm{LP_tf}_{L^2_{B_j^{t,x}}}^2\\
        &\qquad \qquad+ t\sum_{j =-\infty}^\infty t^{-\frac{d_h}{d_w}} \pr{2^{-j}t^\frac{1}{d_w}}^{-d_w} \Phi\pr{2^{-j}}\mathrm{osc}_{B^{t,x}_j}^2\pr{P_t f}\\
        &=: \mathrm{\Rnum{1}} + \mathrm{\Rnum{2}}.
    \end{align*}
    For \Rnum{1}, it holds uniformly in $x \in M$ and $t \in (0,1]$ that 
    \begin{align*}
        \mathrm{\Rnum{1}}&=t^{-\frac{d_h}{d_w}} \sum_{j = -\infty}^\infty 2^{-jd_w}\Phi\pr{2^{-j}}\norm{Q^{(1)}_tf}_{L^2_{B^{t,x}_j}}^2\\
        &\lesssim t^{-\frac{d_h}{d_w}}\sum_{j = -\infty}^\infty 2^{-jd_w} \Phi\pr{2^{-j}}\mu\pr{B^{t,x}_j} \norm{Q^{(1)}_tf}_{L^\infty_M}^2\\
        &\overset{\eqref{ineq:a_reg}    }{ \lesssim} t^{-\frac{d_h}{d_w}}\sum_{j = -\infty}^\infty \Phi\pr{2^{-j}}2^{-jd_w} 2^{-jd_h}t^{\frac{d_h}{d_w}} t^\frac{\alpha}{d_w} \norm{f}_{\dot{\mathcal B}^\alpha_\infty}^2\\
        &\lesssim t^{\frac{2\alpha}{d_w}} \norm{f}_{\dot{\mathcal B}^\alpha_\infty}^2.
    \end{align*}
    For $g \in \mathcal C_M$ and an open set $B \subset M$, define
    \begin{align*}
        \mathrm{osc}^\infty_B(g):= \sup_{x,y \in B} \abs{f(x) - f(y)}.
    \end{align*}
    Then by H\"older,
    \begin{align*}
        \mathrm{osc}_{B}^2(P_tf) \le \mu(B) \mathrm{osc}_B^\infty (P_tf)^2. 
    \end{align*}
    By Calderon's reproducing formula, we thus have for $j \ge 0$
    \begin{align*}
        \mathrm{osc}_{B_j^{t,x}}^2\pr{P_tf}&\overset{\eqref{ineq:a_reg}  }{\lesssim}2^{-jd_h}t^{\frac{d_h}{d_w}}\mathrm{osc}_{B_j^{t,x}}^\infty\pr{P_tf}^2\\
        &\le 2^{-jd_h  }t^\frac{d_h}{d_w}\pr{\mathrm{osc}_{B^{t,x}_j}^\infty\pr{P_2f}+\int_t^2 \mathrm{osc}_{B^{t,x}_j}^\infty\pr{Q^{(1)}_sf} \frac{\mathrm ds}{s}}^2\\
        &\overset{\eqref{ineq:holder_kernel}}{\lesssim}2^{-jd_h}t^\frac{d_h}{d_w}\pr{\pr{2^{-j\Theta}t^\frac{\Theta}{d_w}}\norm{P_1f}_{L^\infty_M} +\int_t^2 \pr{\frac{t}{s}}^\frac{\Theta}{d_w}2^{-j\Theta}\norm{Q^{(1)}_{s/2}f}_{L^\infty_M}\frac{\mathrm ds}{s}}^2\\
        &\lesssim 2^{-j(d_h+2\Theta)}t^{\frac{2\Theta+d_h}{d_w}}\pr{\norm{P_1f}+ \int_t^2 s^{\frac{-\Theta+\alpha}{d_w}} \frac{\mathrm ds}{s}\norm{f}_{\dot{\mathcal B}_\infty^\alpha} }^2\\
        &\lesssim 2^{-j(d_h+2\Theta)} t^{\frac{2\Theta+d_h}{d_w}} t^\frac{2\alpha - 2\Theta}{d_w} \norm{f}_{\Binf{\alpha}}^2\\
        &=2^{-j(d_h+2\Theta)} t^{\frac{d_h+2\alpha}{d_w}}\norm{f}_{\Binf{\alpha}}^2,
    \end{align*}
    where we used the fact that $\alpha< \Theta$ and $t \in (0,1]$ in the fifth line. Similarly, for $j \le 0$ we decompose the integral as 
    \begin{align*}
        \int_t^2\mathrm{osc}_{B_{j}^{t,x}}^\infty\pr{Q_s^{(1)}f} \frac{\mathrm ds}{s} &=\pr{\int_t^{2^{-j}t^{1/d_w}} + \int_{2^{-j}t^{1/d_w}}^2} \mathrm{osc}_{B_j^{t,x}}\pr{Q^{(1)}_s f} \frac{\mathrm{d}s }{s}\\
        &\lesssim \int_t^{2^{-j}t^{1/d_w}} 2^{-j\Theta} \pr{\frac{t}{s}}^\frac{\Theta}{d_w} s^\frac{\alpha}{d_w} \norm{f}_{\Binf{\alpha}} \frac{\mathrm ds}{s} + \int_{2^{-j}t^{1/d_w}}^2 2^{-j\Theta}\pr{\frac{t}{s}   }^\frac{\Theta}{d_w} s^\frac{\alpha}{d_w} \norm{f}_{\Binf{\alpha}} \frac{\mathrm  ds}{s}\\
        &= \int_t^{2} 2^{-j\Theta}\pr{\frac{t}{s}}^\frac{\Theta}{d_w}s^\frac{\alpha  }{d_w} \frac{\mathrm ds}{s} \norm{f}_{\Binf{\alpha}}\\
        &\lesssim t^\frac{\alpha}{d_w} \norm{f}_{\Binf{\alpha}}.
    \end{align*}
    Hence, for $j \le -1$, we have
    \begin{align*}
        \mathrm{osc}_{B_j^{t,x}}^2\pr{P_tf} \lesssim 2^{-j\pr{d_h+2\Theta}}t^\frac{d_h+2\alpha}{d_w}\norm{f}_{\Binf{\alpha}}^2.
    \end{align*}
    Insert this into \Rnum{2} to see 
    \begin{align*}
        \mathrm{\Rnum{2}}&\lesssim t\sum_{j = -\infty}^\infty t^{-\frac{d_h}{d_w}-1} \Phi\pr{2^{-j}}2^{j(d_h+2\Theta - d_w)}t^{\frac{d_h+2\alpha}{d_w}}\norm{f}_{\Binf{\alpha}}^2\\
        &=t^\frac{2\alpha}{d_w} \norm{f}_{\Binf{\alpha}}^2\sum_{j = -\infty}^\infty \Phi\pr{2^{-j}}2^{-j(d_h+2\Theta-d_w)}\\
        &\lesssim t^\frac{2\alpha}{d_w} \norm{f}_{\Binf{\alpha}}^2,
    \end{align*}
    where we used the assumption $d_h+2\Theta-d_w>0$ in the last line. Collect both estimates for the desired result.
\end{proof}

Note that the proof of \ref{thm:gamma_infty_estimates} requires us to control the oscillation of $Q^{(k)}f$ uniformly on small balls, which depends on $\Theta$. This is the source of the $\Theta$ hypothesis in Theorem \ref{thm:besov_product}.

\begin{comment}
\noindent The follow Proposition can be viewed as a converse to Theorem \ref{thm:gamma_infty_estimates}, whose proof can be found in the Appendix~\ref{sec:appendix}.
\begin{prop}\label{prop:sharpness}
    Let $\alpha \in (0,1]$, $k \in \Z_+$, and suppose there exists $C_\alpha>0$ and $f \in \Binf{\alpha}$ so that for all $t \in (0,1]$,
    \begin{align}\label{ineq:sharpness condition}
        \norm{Q^{(1)}_t\Gamma \pr{P_tf,P_tf}}_{L^\infty_M} +\norm{P_t\Gamma\pr{P_tf,Q_t^{(1)}f}}_{L^\infty_M} \le C_\alpha t^{\frac{2\alpha}{d_w}}\norm{f}_{\Binf{\alpha}}^2.
    \end{align}
    Then $P_tf\in \mathcal C_M^\alpha$. In particular, $\alpha\le \Theta$. 
\end{prop}
\begin{remark}
    Note that Proposition \ref{prop:sharpness} shows that under the current method of decompositions, the $\Theta$ restriction is sharp. 
\end{remark}
\end{comment}

\subsubsection{Product estimates}\label{sec:sub_pp_estimates}
The goal here is to prove Theorem \ref{thm:besov_product}. We will estimate each piece of the decomposition \eqref{Calderon Paraproduct Formula}, which when collected will imply the inequality in Theorem \ref{thm:besov_product}.

\begin{prop}\label{prop:pe_lf}
Let $\alpha,\beta \in \R$, for every positive $\gamma$ and  $b \in \N$, it holds uniformly for $f\in \Binf{\alpha},g\in\Binf{\beta}$ that
\begin{align*}
  \norm{\Delta_{-1}^{(b)}(f,g)}_{\mathcal B^\gamma_{\infty}} \lesssim \norm{f}_{\mathcal B^\alpha_{\infty}} \norm{g}_{\mathcal B^\beta_{\infty}}.
\end{align*}
\end{prop}

\begin{proof}%[Proof of Proposition \ref{prop:pe_lf}]
Let $a \ge \gamma$ and consider for $s \in (0,1]$,
\begin{align*}
  Q^{(a)}_s\Delta_{-1}^{(b)}(f,g) = Q^{(a)}_s P_1^{(b)}\left(P_1^{(b)}f \cdot P^{(b)}_1 g \right).
\end{align*}
Observe by Lemma \ref{lem:q_regulariz} that $P_1^{(b)}f \in L^\infty_M$, and  $Q^{(a)}_s P_1^{(b)} = s^a e^{-sL} L^a P_1^{(b)}$. Hence we obtain by Proposition \ref{prop:Pa1_besov_bound} that
\begin{align*}
  \norm{Q_s^{(b)} \Delta_{-1}(f,g)}_{L^\infty_M} & = s^a\norm{e^{-sL}L^aP_1^{(b)} \left(P^{(b)}_1f \cdot P^{(b)}_1 g \right)}_{L^\infty_M} \\
  &\lesssim s^a \norm{L^a P_1^{(b)}}_{L^\infty_M \to L^\infty_M} \norm{P_1^{(b)} f\cdot P_1^{(b)} g}_{L^\infty_M}\\
  &\lesssim s^a \norm{P_1^{(b)} f}_{L^\infty_M} \norm{P_1^{(b)}g}_{L^\infty_M}\\
  &\lesssim s^\gamma \norm{f}_{\mathcal B^\alpha_{\infty}}\norm{g}_{\mathcal B^\beta_{\infty}}.
\end{align*}
Where we used the fact $s \in (0,1]$ and $a \ge  \gamma$ in the last inequality.
\end{proof}

\begin{prop}\label{prop:pe_para}
Let $b \in \N$ with $b \ge 2$, suppose  $\alpha \in (-d_w, d_w)$, %and $f \in \mathcal B^\alpha_{\infty}$. 
\begin{itemize}
  \item Uniformly in $g \in L^\infty_M$ and  $f \in \mathcal B^\alpha_{\infty}$
    \begin{align*}
      \norm{\Pi^{(b)}_g(f)}_{\mathcal B^\alpha_{\infty}} \lesssim \norm{g}_{L^\infty_M} \norm{f}_{\mathcal B^\alpha_{\infty}}.
    \end{align*}
  \item Uniformly in $g \in \mathcal  B^\beta_{\infty}$ and $f \in \mathcal B^\alpha_{\infty}$ with $\beta <0$ and $\alpha + \beta \in (-d_w ,d_w)$,
    \begin{align*}
      \norm{\Pi^{(b)}_g(f)}_{\mathcal B^{\alpha+\beta}_{\infty}} \lesssim \norm{g}_{\mathcal B^\beta_{\infty}} \norm{f}_{\mathcal B^\alpha_{\infty}}.
    \end{align*}
\end{itemize}
\end{prop}

\begin{proof}
Recall that
\begin{align*}
  \Pi_g^{(b)}(f)=\frac{1}{\gamma_b} \int_0^1 tLP_t^{(b)}\left(Q_t^{(b-1)}f \cdot P_t^{(b)}g \right)+ Q^{(b - 1)}_t \left( tL P_t^{(b)}f \cdot P_t^{(b)} g\right)\frac{\mathrm dt}{t}.
\end{align*}
Consider $Q^{(b - 1)}_s \Pi_g^{(b)}(f)$, and split the above integral into two intervals $(0,s)$ and $[s,1]$. For $s \le t$, we use
\begin{align*}
  Q^{(b-1)}_s (tL) P_t^{(b)} = \left(\frac{s}{t} \right)^{b - 1}(tL)^bP_t^{(b)}e^{-sL};\quad Q^{(b - 1)}_s Q^{(b - 1)}_t = \left( \frac{s}{t}\right)^{b - 1}Q^{2(b - 1)}_t e^{-sL}.
\end{align*}
For $t \le s$, we use
\begin{align*}
  Q^{(b - 1)}_s (tL) P_t^{(b)} = \left(\frac{t}{s} \right) Q^{(b)}_s P_t^{(b)}; \quad Q_s^{(b-1)}Q_t^{(b-1)} = \left(\frac{t}{s} \right) Q^{(b)}_s Q^{(b - 2)}_t.
\end{align*}
By the boundedness in $L^\infty_M \to L^\infty_M$ norm of $Q,P$ operators, we have
\begin{equation}\label{ineq: para product decomposition master}
  \begin{aligned}
    \norm{Q_s^{(b - 1)}\Pi^{(b)}_g(f)}_{L^\infty_M}\lesssim \int_0^s \left(\frac{t}{s} \right) \left(\norm{Q_t^{(b - 1)}f}_{L^\infty_M}\norm{P_t^{(b)}g}_{L^\infty_M}+\norm{tL P_t^{(b)} f}_{L^\infty_M}\norm{P_t^{(b)}g}_{L^\infty_M}\right)\frac{\mathrm dt}{t}\\
    +\int_s^1 \left(\frac{s}{t} \right)^{b - 1}\left(\norm{Q_t^{(b - 1)}f}_{L^\infty_M}\norm{P_t^{(b)}g}_{L^\infty_M}+\norm{tL P_t^{(b)} f}_{L^\infty_M}\norm{P_t^{(b)}g}_{L^\infty_M}\right)\frac{\mathrm dt}{t}.
  \end{aligned}
\end{equation}

Since $0< \alpha < d_w$ and $f \in \mathcal B^\alpha_{\infty}$, Proposition \ref{prop:independent_b} gives us
\begin{align}\label{ineq: para product decomposition f inequal }
  \norm{Q_t^{(b - 1)} f}_{L^\infty_M}+ \norm{(tL)^1 P_t^{(b)}f}_{L^\infty_M}\lesssim t^\frac{\alpha}{d_w} \norm{f}_{\mathcal B^\alpha_{\infty}}.
\end{align}
For $g\in L^\infty_{M}$,
\begin{align}\label{ineq: para product decomposition g inequal infty}
  \norm{P_t^{(b)} g}_{L^\infty_M}\lesssim \norm{g}_{L^\infty_M},\quad \text{ uniformly in }t \in (0,1].
\end{align}
If $g \in {\mathcal B}^\beta_{\infty}$ for some $\beta < 0$, then
\begin{equation}\label{ineq: para product decomposition g beta}
  \begin{aligned}
    \norm{P_t^{(b)}g}_{L^\infty_M}&\le    \int_t^1 \norm{Q^{(b)}_u g}_{L^\infty_M}\frac{\mathrm du}{u}+ \norm{P_1^{(b)}g}_{L^\infty_M}\\
    &\lesssim \left( \int_t^1 u^\frac{\beta}{d_w} \frac{\mathrm du}{u}+1\right) \norm{g}_{\mathcal B^\beta_{\infty}}\\
    &\lesssim t^\frac{\beta}{d_w}\norm{g}_{\mathcal B^\beta_{\infty}}.
  \end{aligned}
\end{equation}

Therefore, if $g \in L^\infty_M$, since $\alpha > -  d_w$, we may insert \eqref{ineq: para product decomposition f inequal } and \eqref{ineq: para product decomposition g inequal infty} into \eqref{ineq: para product decomposition master} to see
\begin{align*}
  \norm{Q^{(b- 1)}_t \Pi_g^{(b)}(f)}_{L^\infty_M} &\lesssim \left(\int_0^s \left(\frac{t}{s} \right) t^\frac{\alpha}{d_w} \frac{\mathrm dt}{t}+ \int_s^1 \left( \frac{s}{t}\right)^{b - 1} t^\frac{\alpha}{d_w} \frac{\mathrm dt}{t} \right)\norm{f}_{\mathcal B^\alpha_{\infty}} \norm{g}_{L^\infty_M}\\
  &= \pr{s^{-1}\int_0^st^{1 + \frac{\alpha}{d_w}} \frac{\mathrm dt}{t} + s^{b-1}\int_s^1 t^{\frac{\alpha}{d_w} - (b-1)}\frac{\mathrm{d}t   }{t}}\norm{f}_{\mathcal B^\alpha_{\infty}} \norm{g}_{L^\infty_M}\\
  &\lesssim \left( s^\frac{\alpha}{d_w} + s^{b - 1}+ s^{\frac{\alpha}{d_w}} \right)\norm{f}_{\mathcal B^\alpha_{\infty}} \norm{g}_{L^\infty_M}\\
  &\lesssim s^\frac{\alpha}{d_w} \norm{f}_{\mathcal B^\alpha_{\infty}}\norm{g}_{L^\infty_M},
\end{align*}
where we used the fact that $1 + \frac{\alpha}{d_w} >0$ in the third line, and the fact that $b - 1 \ge \frac{\alpha}{d_w}$ for $b \ge 2$, $1 \ge \frac{\alpha}{d_w}$ in the last line.

If $g \in {\mathcal B}^\beta_{\infty}$ with $\alpha+ \beta \in (- d_w,  d_w)$, we see by inserting \eqref{ineq: para product decomposition f inequal } and \eqref{ineq: para product decomposition g beta} into \eqref{ineq: para product decomposition master}, we see that
\begin{align*}
  \norm{Q^{(b - 1)}_s \Pi^{(b)}_g(f)}_{L^\infty_M}&\lesssim \left(\int_0^s \left(\frac{t}{s} \right) ^1 t^\frac{\alpha+\beta}{d_w} \frac{\mathrm dt}{t}+ \int_s^1 \left( \frac{s}{t}\right)^{b - 1} t^\frac{\alpha+\beta}{d_w} \frac{\mathrm dt}{t}\right)\norm{f}_{\mathcal B^\alpha_{\infty}}\norm{g}_{\mathcal B^\beta_{\infty}}\\
  &\lesssim \left( s^\frac{\alpha+\beta}{d_w} +s^{b - 1}+ s^\frac{\alpha+\beta}{d_w} \right)\norm{f}_{\mathcal B^\alpha_{\infty}}\norm{g}_{\mathcal B^\beta_{\infty}}\\
  &\lesssim s^\frac{\alpha+\beta}{d_w} \norm{f}_{\mathcal B^\alpha_{\infty}}\norm{g}_{\mathcal B^\beta_{\infty}}.
\end{align*}
This holds since $b - 1 \ge 1 > \frac{\alpha+\beta}{d_w} > - 1$. This holds uniformly for $s \in (0,1]$, which concludes the proof.
\end{proof}

\begin{prop}\label{prop:pe_re}
Suppose $b \in \N$, then for any $\alpha, \beta \in (-\infty,\Theta)$ with $\alpha+\beta>0$,   the following inequality holds uniformly  in $f \in \mathcal B^\alpha_{\infty}$ and $g \in \mathcal B^\beta_{\infty}$,
\begin{align*}
  \norm{\Pi^{(b)}(f,g)}_{\mathcal B^{\alpha+\beta}_{\infty}} \lesssim \norm{f}_{\mathcal B^\alpha_{\infty}} \norm{g}_{\mathcal B^\beta_{\infty}}.
\end{align*}
\end{prop}

\begin{proof}
Recall
\begin{align*}
  \Pi^{(b)}(f,g)
  &=\frac{1}{\gamma_b}\int_0^1\left\{2P^{(b)}_t t \Gamma\left(Q_t^{(b-1)}f , P_t^{(b)}g \right) -P^{(b)}_t\left(Q_t^{(b-1)}f \cdot tL \left( P_t^{(b)}g\right)\right) \right\} \frac{\mathrm dt}{t}\\
  &\quad + \frac{1}{\gamma_b}\int_0^1\left\{2P^{(b)}_t t \Gamma\left(Q_t^{(b-1)}g , P_t^{(b)}f \right) -P^{(b)}_t\left(Q_t^{(b-1)}g \cdot tL \left( P_t^{(b)}f\right)\right) \right\} \frac{\mathrm dt}{t}\\
  &\quad - \frac{2}{\gamma_b}\int_0^1 Q^{(b - 1)}_t t \Gamma\left( P^{(b)}_t f,P^{(b)}_t g\right)\frac{\mathrm dt}{t}.
\end{align*}
For $s \in (0,1]$ we consider $Q^{(b - 1)}_s \Pi^{(b)}(f,g)$. We again split the integrals on $[0,1]$ into $[0,s]$ and $[s,1]$ and call them $\mathrm I$ and $\mathrm{\Rnum{2}}$ respectively.

\noindent First consider $\mathrm I$, the case $t  \le s$. We see that 
\begin{align*}
  \norm{Q^{(b - 1)}_tf}_{L^\infty_M}+ \norm{tL P_t^{(b)}f}_{L^\infty_M}\lesssim t^\frac{\alpha}{d_w}\norm{f}_{\mathcal B^\alpha_{\infty}}.
\end{align*}
For terms that involve $\Gamma$, since $\max\{\alpha,\beta\}< \Theta$, we see by Theorem \ref{thm:gamma_infty_estimates}, that for $U_t,V_t,W_t \in \{Q^{(b-1)}_t, P_t^{(b)}, tLP_t^{(b)}\}$  that 
\begin{align*}
    \norm{tV_t \Gamma\left(U_t f,W_t g\right)}_{L^\infty_M} &\overset{\eqref{ineq:cs for energy measure}}{\le}\norm{\pr{t\int_M \abs{V}_t(\cdot, y) \Gamma\pr{U_tf,U_tf }(\mathrm d y)}^\frac{1}{2}}_{L^\infty_M} \norm{\pr{t\int_M \abs{V}_t(\cdot, y) \Gamma\pr{W_tg,W_tg }(\mathrm d y)}^\frac{1}{2}}_{L^\infty_M}\\
    &\lesssim t^\frac{\alpha+\beta}{d_w} \norm{f}_{\Binf{\alpha}}\norm{g}_{\Binf{\alpha}},
\end{align*}
where $\abs{V}_t(x,y)$ denotes the absolute value of the integral kernel of the operator $V_t$ with respect to $\mu$.
Indeed, the proof of Theorem \ref{thm:gamma_infty_estimates} requires only the upper bound in \eqref{ineq:sgu}.
Hence, for  $0 \le t \le s\le 1$
\begin{align*}
  \norm{Q^{(b - 1)}_s (\mathrm I)}_{L^\infty_M}&\lesssim \int_0^t t^\frac{\alpha+\beta }{d_w} \frac{\mathrm dt}{t} \norm{f}_{\mathcal B^\alpha_{\infty}}\norm{g}_{\mathcal B^\beta_{\infty}}\\
  &\lesssim s^\frac{\alpha+\beta }{d_w}\norm{f}_{\mathcal B^\alpha_{\infty}} \norm{g}_{\mathcal B^\beta_{\infty}}.
\end{align*}
Where we used the fact that $\alpha+\beta >0$.

\noindent For the second part, we observe that for $t>s$,
\begin{align*}
  Q^{(b - 1)}_s P_t^{(b)} = \left(\frac{s}{t} \right)^{b - 1} e^{-sL} (tL)^{b - 1}P_t^{(b)}; \quad Q^{(b-1)}_s Q^{(b - 1)}_t = \left( \frac{s}{t}\right)^{b - 1} Q^{2(b - 1)}_te^{-sL}.
\end{align*}
Thus we get for the second part
\begin{align*}
  \norm{Q^{(b - 1)}_s (\mathrm{\Rnum{2}})} &\lesssim \left(\int_s^1 \left( \frac{s}{t} \right)^{b - 1}t^\frac{\alpha +\beta }{d_w} \frac{\mathrm dt}{t}\right)\norm{f}_{\mathcal B^\alpha_{\infty}} \norm{g}_{\mathcal B^\beta_{\infty}}\\
  &\lesssim \left( s^{b - 1} + s^\frac{\alpha+\beta }{d_w}\right)\norm{f}_{\mathcal B^\alpha_{\infty}} \norm{g}_{\mathcal B^\beta_{\infty}}\\
  &\lesssim s^\frac{\alpha+\beta }{d_w}\norm{f}_{\mathcal B^\alpha_{\infty}} \norm{g}_{\mathcal B^\beta_{\infty}},
\end{align*}
where we used the fact that $b - 1 >1 \ge \frac{2\Theta}{d_w}> \frac{\alpha+\beta}{d_w}$ and $s \in (0,1]$.
\end{proof}

\noindent The proof of the following elementary lemma can be found in the Appendix.
\begin{lemma}\label{lem:besov_norm_compare}
Let $u \in {\mathcal S}$ and $\alpha,\beta \in \R$, the following inequalities hold,
\begin{align*}
  \norm{u}_{\mathcal B^\alpha_{\infty}} &\lesssim \norm{u}_{\mathcal B^\beta_{\infty}} ,\qquad \text{  for } \alpha \le \beta\\
  \norm{u}_{\mathcal B^\alpha_{\infty}} &\lesssim \norm{u}_{L^\infty_M},\qquad \text{ for } \alpha \le 0.\\
  \norm{u}_{L^\infty_M} &\lesssim \norm{u}_{\mathcal B^\alpha_{\infty}} ,\qquad\text{ for } \alpha >0.
\end{align*}
\end{lemma}

\begin{comment}
\begin{cor}\label{cor:besov_product}
    For $\alpha, \beta \in \left(-d_w, \frac{d_w}{2}\wedge 2\Theta\right)$ and $\alpha+ \beta \in (0, d_w)$, one can find a decomposition of $f\cdot g$ so that
    \begin{align*}
        \norm{f\cdot g}_{\alpha \wedge \beta} \lesssim \norm{f}_\alpha \cdot \norm{g}_\beta,\quad \text{ if  }f \in B^\alpha,\, g \in B^\beta.
    \end{align*}

\end{cor}
\end{comment}

\begin{comment}
\begin{theorem}\label{thm: para_product_in_body}
    Suppose $\alpha>\beta$ are real numbers so that $\alpha >0$, $\beta \in (-\Theta,\Theta) \backslash \{0\}$ and $\alpha+\beta >0$. Then there exists $1 \in (0,\frac{2\Theta}{d_w})$ so that the decomposition in \eqref{Calderon Paraproduct Formula} of $fg$ for $f \in \mathcal B^\alpha_{\infty}$ and $g \in \mathcal B^\beta_{\infty}$ with the following estimates holds uniformly in $f$ and $g$,
    \begin{align*}
        \norm{fg}_{\mathcal B^\beta_{\infty}} \lesssim \norm{f}_{\mathcal B^\alpha_{\infty}}\norm{g}_{\mathcal B^\beta_{\infty}}.
    \end{align*}
\end{theorem}
\end{comment}
\begin{proof}[Proof of Theorem \ref{thm:besov_product}]
Assume $\alpha>\beta$ are real numbers so that $\alpha >0$, $\beta \in\left(-\Theta, \Theta\right)$  and $\alpha+ \beta >0$. Then there exists $\sigma \in (0,\Theta)$ so that $\sigma+\beta >0$, Set $\gamma:= \min\{\alpha,\sigma\} \in (0,\Theta)$.

\noindent Consider for sufficiently large $b \in \mathbb N$ that
\begin{align*}
  \norm{f \cdot g}_{\mathcal B^\beta_{\infty}} \le \norm{\Delta_{-1}^{(b)}(f,g)}_{\mathcal B^\beta_{\infty}} + \norm{\Pi^{(b)}_f(g)}_{\mathcal B^\beta_{\infty}} + \norm{\Pi_g^{(b)}(f)}_{\mathcal B^\beta_{\infty}} + \norm{\Pi^{(b)}(f,g)}_{\mathcal B^\beta_{\infty}}.
\end{align*}
 By Proposition \ref{prop:pe_lf} we see $$\norm{\Delta_{-1}^{(b)}(f,g)}_{\mathcal B^\beta_{\infty}} \lesssim \norm{f}_{\mathcal B^\alpha_{\infty}} \norm{g}_{\mathcal B^\beta_{\infty}}.$$
Since $\beta < \Theta$, and $\alpha >0$, we see by Proposition \ref{prop:pe_para} and Lemma \ref{lem:besov_norm_compare} that
\begin{align*}
  \norm{\Pi_f^{(b)}(g)}_{\mathcal B^\beta_{\infty}} \lesssim \norm{f}_{L^\infty_M}\norm{g}_{\mathcal B^\beta_{\infty}} \lesssim\norm{f}_{\mathcal B^{\alpha}_{\infty}}\cdot \norm{g}_{\mathcal B^\beta_{\infty}}.
\end{align*}
Similarly, since $\gamma < \Theta$, we see if $\beta >0$, it holds by Proposition \ref{prop:pe_para} and Lemma \ref{lem:besov_norm_compare} that
\begin{align*}
  \norm{\Pi^{(b)}_g(f)}_{\mathcal B^\beta_{\infty}} \lesssim \norm{g}_{L^\infty_M} \norm{f}_{\mathcal B^\gamma_{\infty}} \lesssim \norm{g}_{\mathcal B^\beta_{\infty}} \norm{f}_{\mathcal B^\alpha_{\infty}} .
\end{align*}
On the other hand, if $\beta <0$, we see
\begin{align*}
  \norm{\Pi^{(b)}_g(f)}_{\mathcal B^\beta_{\infty}} \lesssim\norm{\Pi^{(b)}_g(f)}_{\mathcal B^{\gamma+\beta}_{\infty}} \lesssim\norm{g}_{\mathcal B^\beta_{\infty}} \norm{f}_{\mathcal B^\gamma_{\infty}} \lesssim \norm{g}_{\mathcal B^\beta_{\infty}} \norm{f}_{\mathcal B^\alpha_{\infty}}.
\end{align*}
Finally, since $\gamma+\beta >0$, we see from Lemma \ref{lem:besov_norm_compare} and Proposition \ref{prop:pe_re} that
\begin{align*}
  \norm{\Pi^{(b)}(f,g)}_{\mathcal B^{\beta}_{\infty}} \lesssim\norm{\Pi^{(b)}(f,g)}_{\mathcal B^{\beta+\gamma}_{\infty}} \lesssim \norm{f}_{\mathcal B^\gamma_{\infty}} \norm{g}_{\mathcal B^\beta_{\infty}} \lesssim \norm{f}_{\mathcal B^\alpha_{\infty}} \norm{g}_{\mathcal B^\beta_{\infty}}.
\end{align*}
Collecting all terms to see the desired inequality.
\end{proof}

\subsection{Schauder estimate}\label{sec:s_estimates}
In this section, we prove a  Schauder-type estimate. For any normed space $(\mathcal B,\norm{\cdot}_{\mathcal B})$, we denote by $\mathcal C_T^\alpha \mathcal B$ the space of $\alpha$-H\"older continuous functions from $[0,T]$ to $\mathcal B$, equipped with its natural norm. We write $\mathcal C_T \mathcal B := \mathcal C_T^0 \mathcal B$.
 We  define the \textit{resolution operator} ${\mathcal R}^{(k)}$ for $k \in \Z_+$ via  the formula
\begin{align}\label{eq:resolution operator}
{\mathcal R}^{(k)}(v)_t := \int_0^t P^{(k)}_{t-s} v(s)\mathrm ds;
\end{align}

\begin{prop}[Schauder Estimate]\label{prop: schauder}
Let $\beta \in \R$, $p \in [1,\infty]$ and $T>0$ be arbitrary but fixed. if $v \in {\mathcal C}_T \mathcal B^\beta_{p,\infty}$, ${\mathcal R}^{(k)}(v)$ belongs to ${\mathcal C}_T \mathcal B^{\beta +d_w}_{p,\infty}$. Furthermore, for all $\eta \ge 1$, we have
\begin{align*}
  \norm{{\mathcal R}^{(k)}(v)_t}_{\mathcal B^{\beta + d_w/\eta}_{p,\infty}} \lesssim (t+t^\frac{\eta-1}{\eta}) \sup_{s \in [0,t]} \norm{v(s)}_{\mathcal B^\beta_{p,\infty}},\quad \text{ for all } t \in [0,T].
\end{align*}
Moreover, if $\alpha \in \R$ with $-d_w < \beta -\alpha < 0$, then
\begin{align*}
  \norm{{\mathcal R}^{(k)}(v)}_{{\mathcal C}^\frac{\beta- \alpha +d_w}{d_w}_T\mathcal B^\alpha_{p,\infty}} \lesssim \norm{v}_{{\mathcal C}_T\mathcal B^\beta_{p,\infty}}.
\end{align*}
\end{prop}
\begin{proof}
We fix $b \in \Z_+$, and denote $\mathcal R:= \mathcal R^{(b)}$ for notational convenience. Let $c \ge \frac{|\beta|}{d_w}+1$ be an integer and $\tau \in (0,1]$. Then
\begin{align*}
  Q^{(c)}_\tau {\mathcal R}(v)_t = \int_0^t Q^{(c)}_\tau P^{(b)}_{t-s} v(s)ds.
\end{align*}
For $0<t-s$ and $\tau >0$, we have
\begin{equation}\label{ineq: shauder step 1}
  \begin{aligned}
    \norm{Q^{(c)}_\tau P^{(b)}_{t-s} v(s)}_{L^p_M} = \norm{P^{(b)}_{t-s}Q^{(c)}_\tau  v(s)}_{L^p_M}\le \norm{P_{t-s}^{(b)}}_{L^p_M \to L^p_M}\norm{Q^{(c)}_\tau  v(s)}_{L^p_M}\lesssim \tau^\frac{\beta}{d_w} \norm{v(s)}_{\mathcal B^\beta_{p,\infty}}.
  \end{aligned}
\end{equation}
But  for $0<t-s \le \tau$, we see that $\frac{1}{2} \le\frac{\tau}{t-s+\tau} < 1$ and hence
\begin{align*}
  \tau^\frac{\beta}{d_w} = \left( t-s +\tau\right)^\frac{\beta}{d_w}\left( \frac{\tau}{t-s+\tau}\right)^\frac{\beta}{d_w}\lesssim\left( t-s +\tau\right)^\frac{\beta}{d_w}\left( \frac{\tau}{t-s+\tau}\right)^c.
\end{align*}
Thus
\begin{equation}\label{ineq: schauder step 2}
  \begin{aligned}
    \norm{Q^{(c)}_\tau P^{(b)}_{t-s} v(s)}_{L^p_M} &\overset{\eqref{ineq: shauder step 1}}{\lesssim} \tau^\frac{\beta}{d_w} \norm{v(s)}_{\mathcal B^\beta_{p,\infty}} \\
    &\lesssim\left(\frac{\tau}{\tau + t-s}\right)^c \left(\tau + t-s \right)^\frac{\beta}{d_w} \norm{v(s)}_{\mathcal B^\beta_{p,\infty}}.
  \end{aligned}
\end{equation}

For the case $t - s > \tau$,
\begin{align*}
  \norm{Q^{(c)}_\tau P^{(b)}_{t-s} v(s)}_{L^p_M}&=\pr{\frac{\tau}{t-s}}^c \norm{P_{\tau}(t-s)^c L^c P^{(b)}_{t-s}v(s)}_{L^p_M}  \\
  & \le \pr{\frac{\tau}{t-s}}^c\norm{(t-s)^c L^c P_{t-s}^{(c)} v(s)}_{L^p_M} \\
  &\le \pr{\frac{\tau}{t-s}}^c\sum_{k = 0}^{c - 1} \frac{1}{k!} \norm{Q^{(k+c)}_{t-s}v(s)}_{L^p_M}\\
  &\le\pr{\frac{\tau}{t-s}}^c \sum_{k = 0}^{c-1}\frac{1}{k!} (t-s)^\frac{\beta}{d_w} \sup_{\tau \in (0,t-s]}\tau^{-\frac{\beta}{d_w}} \norm{Q_{\tau}^{(k+c)} v(s)}_{L^p_M}\\
  &\lesssim \pr{\frac{\tau}{t-s}}^c(t-s)^\frac{\beta}{d_w} \norm{v(s)}_{\mathcal B^\beta_{p,\infty}}.
\end{align*}
Note for $t -s \ge \tau$ implies $\frac{\tau}{t-s +\tau} \ge \frac{1}{2}\frac{\tau}{t-s}$, and
\begin{align*}
  \begin{cases}
    \pr{t-s}^\frac{\beta}{d_w} \le \pr{t-s+\tau}^\frac{\beta}{d_w},&\beta \ge 0\\
    (t-s)^\frac{\beta}{d_w} = 2^{- \frac{\beta}{d_w}} \pr{2(t-s)}^{\frac{\beta}{d_w}} \lesssim (t-s+\tau)^\frac{\beta}{d_w},& \beta \le 0
  \end{cases}
\end{align*}
Hence for $t-s\ge \tau$, it holds that
\begin{equation}\label{ineq: schauder step 3}
  \begin{aligned}
    \norm{Q^{(c)}_\tau P_{t-s}^{(b)} v(s)}_{L^p_M} \lesssim \frac{\tau}{t-s+\tau}(t-s+\tau)^\frac{\beta}{d_w} \norm{v(s)}_{\mathcal B^\beta_{p,\infty}}.
  \end{aligned}
\end{equation}
Combining \eqref{ineq: schauder step 2} and \eqref{ineq: schauder step 3}, we see that for $t-s>0$ and $\tau \in (0,1]$,
\begin{align*}
  \norm{Q^{(c)}_\tau P_{t-s}^{(b)} v(s)}_{L^p_M}  \lesssim \left(\frac{\tau}{t - s+ \tau} \right)^c (t-s+\tau)^\frac{\beta}{d_w} \norm{v(s)}_{\mathcal B^\beta_{p,\infty}}.
\end{align*}
Thus for $\eta \ge 1$, and $\eta' = \frac{\eta}{\eta -1}$ (assume $1/0 = \infty$), we may choose $c$ large so that
\begin{align*}
  \int_0^t \norm{Q^{(c)}_\tau P^{(b)}_{t-s}v(s)}_{L^p_M} ds &\lesssim \int_0^t \left(\frac{\tau}{t - s+ \tau} \right)^c (t-s+\tau)^\frac{\beta}{d_w} ds \sup_{s \in [0,t]}\norm{v(s)}_{\mathcal B^\beta_{p,\infty}}\\
  &\le t^\frac{1}{\eta'} \left( \int_0^t \tau^{c\eta} (t-s+\tau)^{\frac{\beta\eta}{d_w} - \eta c}\right)^\frac{1}{\eta}   \sup_{s \in [0,t]}\norm{v(s)}_{\mathcal B^\beta_{p,\infty}}  \\
  &\lesssim t^\frac{1}{\eta'} \left( \tau^{c\eta} (t+\tau)^{\frac{\eta \beta +d_w}{d_w} - c\eta} + \tau^\frac{\beta \eta +d_w}{d_w}\right)^\frac{1}{\eta} \sup_{s \in [0,t]}\norm{v(s)}_{\mathcal B^\beta_{p,\infty}}    \\
  &\le t^\frac{1}{\eta'} \tau^\frac{\beta + d_w/\eta}{d_w} \sup_{s \in [0,t]}\norm{v(s)}_{\mathcal B^\beta_{p,\infty}},
\end{align*}
where we used the fact that for large $c$, $\frac{\beta\eta+d_w}{d_w} - c\eta <0$ implies $(t+\tau)^{\frac{\beta\eta+d_w}{d_w}-c\eta} \le \tau^{\frac{\beta\eta+d_w}{d_w}-c}$ in the last inequality. Hence
\begin{align*}
  \sup_{\tau \in (0,1]} \tau^{-\frac{\beta +d_w/\eta}{d_w}}\norm{Q^{(c)}_\tau {\mathcal R}(v)_t}_{L^p_M} \lesssim  t^\frac{1}{\eta'}\sup_{s \in [0,t]} \norm{v(s)}_{\mathcal B^\beta_{p,\infty}}.
\end{align*}
Also,
\begin{align*}
  \norm{P_1 {\mathcal R}(v)_t}_{L^p_M} \le \int_0^t \norm{P^{(b)}_{t-s} P_1v(s)}_{L^p_M} ds \lesssim T\sup_{s \in [0,t]} \norm{v(s)}_{\mathcal B^\beta_{p,\infty}}.
\end{align*}
This concludes the proof of the first part of the proposition.

For the second part, note that for for $0 \le s < t \le T$,
\begin{align}\label{eq: shauder 2 step 1}
  {\mathcal R}(v)_t - {\mathcal R}(v)_s = \int_0^s \left(P_{t - r}^{(b)} - P_{s - r}^{(b)}\right) v(r)dr + \int_s^t P^{(b)}_{t-r} v(r) dr
\end{align}
Observe also
\begin{align*}
  \int_0^s \left(P_{t - r}^{(b)} - P_{s - r}^{(b)}\right) v(r)dr &= \int_0^s \int_{s}^t  \partial_z P^{(b)}_{z - r}dz v(r) dr\\
  &= \frac{1}{\gamma_b}\int_0^s \int_s^t Q^{(b)}_{z-r} \frac{dz}{z - r} v(r) dr.
\end{align*}
Hence, by Lemma \ref{lem:q_regulariz}, we have for $\alpha - \beta >0$
and sufficiently large $k \in \mathbb N$, it holds for $h \in (0,1]$  that
\begin{equation*}
  \begin{aligned}
    h^{-\frac{\alpha}{d_w}}\norm{Q^{(k)}_h\int_0^s \left(P^{(b)}_{t-r} - P^{(b)}_{s-r} \right) v(r)dr}_{L^p_M}&\lesssim \int_0^s \int_s^t h^{-\frac{\alpha}{d_w}}\norm{Q_h^{(k)}Q_{z-r}^{(b)} v(r)}_{L^p_M} \frac{dz}{z - r} dr\\
    &\lesssim \sup_{s \in [0,t]} \norm{v(s)}_{\mathcal B^\beta_{p,\infty}} \int_0^s \int_s^t (z-r)^{\frac{\beta-\alpha}{d_w} - 1} dz\,dr.
  \end{aligned}
\end{equation*}
Since  $\alpha - \beta >0$ 
and $\frac{\beta - \alpha}{d_w}>-1$, we see
\begin{align*}
  \int_0^s \int_s^t (z-r)^{\frac{\beta - \alpha}{d_w} - 1} dz\,dr &= \frac{\beta - \alpha}{d_w}\int_0^s (t-r)^\frac{\beta-\alpha}{d_w}-(s - r)^\frac{\beta-\alpha}{d_w} dr\\
  &\lesssim \left|t^\frac{\beta-\alpha+d_w}{d_w}   - s^\frac{\beta-\alpha+d_w}{d_w}\right| + (t-s)^\frac{\beta-\alpha+d_w}{d_w}\\
  &\lesssim (t-s)^\frac{\beta-\alpha+d_w}{d_w},
\end{align*}
where the last inequality is due to the fact that $\frac{\beta - \alpha +d_w}{d_w} \in (0,1]$.
Hence
\begin{align}\label{eq: shauder 2 step 2}
  h^{-\frac{\alpha}{d_w}}\norm{Q^{(k)}_h\int_0^s \left(P^{(b)}_{t-r} - P^{(b)}_{s-r} \right) v(r)dr}_{L^p_M} \lesssim \pr{t-s}^\frac{\beta - \alpha+d_w}{d_w} \sup_{s \in (0,t]} \norm{v(s)}_{\mathcal B^\beta_{p,\infty}}
\end{align}

Finally, we see again by Lemma \ref{lem:q_regulariz} and the fact that $-d_w < \beta - \alpha < 0$ that
\begin{equation}\label{eq: shauder 2 step 3}
  \begin{aligned}
    \int_s^t h^{-\frac{\alpha}{d_w}}\norm{Q^{(k)}_h P_{t-r}^{(b)} v(r)}_{L^p_M} dr &\lesssim \int_s^t \sum_{\ell = 0}^b h^{-\frac{\alpha}{d_w}}\norm{Q^{(k)}_h Q^{(\ell)}_{t-r} v(r)}_{L^p_M} \mathrm dr \\
    &\lesssim \sup_{r \in [0,t]} \norm{v(r)}_{\mathcal B^\beta_{p,\infty}}\int_s^t (t-r)^{\frac{\beta - \alpha}{d_w}} \mathrm dr\\
    &\lesssim \sup_{r \in [0,t]} \norm{v(r)}_{\mathcal B^\beta_{p,\infty}} (t-s)^\frac{\beta-\alpha+d_w}{d_w}.
  \end{aligned}
\end{equation}
Plug \eqref{eq: shauder 2 step 2} and \eqref{eq: shauder 2 step 3} into \eqref{eq: shauder 2 step 1} to see for sufficiently large $k \in \N$,
\begin{align}\label{eq: shauder 2 step 4}
  \sup_{0\le s < t\le T}\frac{\sup_{\tau \in (0,1]}\norm{Q^{(k)}_\tau \pr{ \mathcal R(v)_t -\mathcal R(v)_s}}_{L^p_M}}{\pr{t-s}^{\frac{\beta - \alpha +d_w}{d_w}}} \lesssim \norm{v}_{\mathcal C_T \mathcal B^\beta_{p,\infty}}.
\end{align}
Now observe that
\begin{align*}
  \norm{P_1 \int_0^s \pr{P_{t-s}^{(b)} - P_{s-t}^{(b)}} v(r) \mathrm ds}_{L^p_M}&\le  \int_0^s \int_s^t \norm{Q^{(b)}_{z-r} P_1 v(r)}_{L^p_M} \frac{\mathrm dz}{z - r} \mathrm dr\\
  &\le \int_0^s \int_s^t \norm{Q_{z-r}^{(b)}}_{L^p_M \to L^p_M} \norm{P_1v(r)}_{L^p_M} \frac{\mathrm dz}{z - r}\mathrm dr\\
  &\lesssim \int_0^s \int_s^t (z-r)^{-1}\mathrm dz \mathrm dr \sup_{r \in [0,t]}\norm{v(r)}_{\mathcal B^\beta_{p,\infty}}.
\end{align*}
Also note that for any $\varepsilon >0$, the following inequality holds for $0\le s <t \le T$:
\begin{align*}
  \abs{\int_0^s \int_s^t (z-r)^{-1}\mathrm dz \mathrm dr} &=\abs{\int_0^s \ln\pr{t-r} - \ln(s - r) \mathrm dr} \\
  &=\abs{(t-s)\ln(t-s)} + \abs{t\ln (t) - s \ln(s)} \\
  &\lesssim \abs{t-s}^{1-\frac{\varepsilon}{d_w}}.
\end{align*}
Hence we can pick $0<\varepsilon= \alpha - \beta$ so that
\begin{align*}
  \norm{P_1 \int_0^s \pr{P_{t-s}^{(b)} - P_{s-t}^{(b)}} v(r) \mathrm ds}_{L^p_M} \lesssim \sup_{r \in [0,t]} \norm{v(r)}_{\mathcal B^\beta_{p,\infty}} \abs{t-s}^{\frac{d_w - \varepsilon}{d_w}} \le \sup_{r \in [0,t]} \norm{v(r)}_{\mathcal B^\beta_{p,\infty}} \abs{t-s}^{\frac{d_w - \alpha +\beta }{d_w}}.
\end{align*}
Similarly,
\begin{align*}
  \norm{P_1\int_s^t P^{(b)}_{t-r} v(r)\mathrm dr}_{L^p_M} &\le \int_s^t \norm{P_{t-r}^{(b)}}_{L^p_M \to L^p_M}\norm{P_1 v(r)}_{L^p_M} \mathrm dr \lesssim \abs{t-s}.
\end{align*}
Therefore, we see by \eqref{eq: shauder 2 step 4} that
\begin{align*}
  \norm{\mathcal R(v)}_{\mathcal C_T^\frac{\beta - \alpha +d_w}{d_w}\mathcal B^\alpha_{p,\infty}}&= \sup_{t \in [0,T]} \norm{\mathcal R(v)_t}_{\mathcal B_{p,\infty}^\alpha} + \sup_{0\le s < t\le T}\frac{\norm{\mathcal R(v)_t - \mathcal R(v)_s}_{\mathcal B^\alpha_{p,\infty}}}{(t-s)^\frac{\beta - \alpha+d_w}{d_w}}\\
  &\lesssim \norm{v}_{\mathcal C_T\mathcal B^\beta_{p,\infty}},
\end{align*}
which is the desired result.
\end{proof}

\section{The Edwards-Wilkinson Equation and its Wick powers}\label{sec: EW and Wick}

We have now developed all the analytic machinery needed to rigorously solve \eqref{eq:phi n formal} in the renormalized sense of \cite{DPD,MW17,TW18}. In this section,  we first introduce $\xi$, the space-time white noise on general metric measure spaces and $Y$, the solution to the Edwards-Wilkinson equation \eqref{eq:she mass}. As noted in the introduction, if $d_h < d_w$, the solution is a.s. a continuous function on $M$, but can only be understood as distributions if $d_h \ge d_w$. Finally, we will  construct the Wick powers of $Y$ and compute their regularity indices, while making the renormalization terms explicit.

\subsection{Spacetime White Noise and It\^o-Wiener Chaos Expansion}

Let $\xi$ be a space-time white noise on $\R \times M$ (with respect to the reference measure $\mu$) on some probability space  $(\Omega,\mathcal{F}_\Omega,\P)$, which we fix from now on. That is, 
$$
\left\{\xi \left(f \right); f \in L^2(\R \times M, \mathrm dt \times \mu):=\mathcal{H}\right\}
$$
is a collection of centered Gaussian random variables so that for each $f,g \in L^2(\R \times M, dt\times \mu)$,
\begin{align*}
\E \left[\xi(f) \xi(g) \right] = \int_\R \int_M f(t,x) g(t,x) \mu(\mathrm dx)\mathrm dt
\end{align*}
We set
\begin{align*}
\widetilde {\mathcal F}_t:=\sigma \pr{\xi(\phi): \phi|_{(t,\infty) \times M}\equiv 0,\, \phi \in L^2\pr{\R\times M, \mathrm dt\times \mu}}, \qquad t\in \R,
\end{align*}
and denote $(\mathcal F_t)_{t > -\infty}$ the usual augmentation of the filtration $(\widetilde{\mathcal F}_t)_{t > - \infty}.$

We now assume that $\mathcal{F}_\Omega$ is the $\sigma-$algebra generated by $\set{\xi(f)}_{f\in\mathcal{H}}$. %\textcolor{red}{(do we need this? can $\mathcal F_\Omega$ be larger? We need its augmentation at least)} {\color{blue}If you want to work with filtrations or something, you can do that, it will not affect what we do.}. 
Under this assumption, it is well known (see, for example \cite{nualart2006malliavin,holden1996stochastic}) that we have the orthogonal decomposition $$L^2(\Omega,\mathcal{F}_\Omega,\P)=\bigoplus_{k\geq 0}^\infty \mathcal{H}^{k},$$
where $\mathcal{H}^k$ is the $k-$th homogeneous Wiener chaos. More precisely, $\mathcal{H}^0=\R$, $\mathcal{H}^1=\set{\xi(\phi):\phi\in \mathcal{H}}$ and$$\mathcal{H}^{k}=\set{I_k(\phi):\phi\in \mathcal{H}^{\otimes_{\text{sym}}k}}.$$
Here $\mathcal{H}^{\otimes_{\text{sym}}k}$ denotes the set of square integrable, symmetric kernels $\phi:(\R\times M)^k\to \R$ and $I_k$ is the $k-$fold iterated stochastic integral (see \cite[Section 1.1.1, 1.1.2]{nualart2006malliavin} and \cite{holden1996stochastic}). We recall the following useful estimate from \cite{nelson1973free} (see also \cite[Theorem 1.4.1]{nualart2006malliavin}).
\begin{prop}[Nelson's Estimate]\label{prop: Nelson estimate}
For any $k\geq 2$, $p\geq 2$, and $X\in \mathcal{H}^k$, we have $$\norm{X}_{L^p_\Omega}^2\lesssim \norm{X}_{L^2_\Omega}^2.$$
\end{prop}

\subsection{The Edwards-Wilkinson Equation}

We denote $Y$ as the solution of the Edwards-Wilkinson equation on $\R_+ \times M$ with appropriate initial condition $Y_0$, which we will specify later,
\begin{equation}\label{eq: EW eqn}
\begin{cases}
  \partial_t Y =-(L +1 )Y + \xi\\
  Y(0,\cdot) = Y_0.
\end{cases}
\end{equation}
For $d_h \ge d_w$ (which we assumed), $Y(t):=Y(t,\cdot)$ must be  understood as a distribution on $M$ and is not point-wise defined for $x \in M$. Motivated by distribution theory and Duhamel's principle, we define the solution $Y$ of \eqref{eq: EW eqn} as follows: for every test function $\varphi \in \mathcal S$ and $t>0$,
\begin{align}\label{eq: def of solution to EW}
Y(t,\varphi):= \left \langle e^{-t}P_tY_0, \varphi \right \rangle + \int_0^t \int _M \left \langle \varphi, e^{-(t-s)}p_{t-s}(y,\cdot) \right \rangle\xi(\mathrm dy, \mathrm ds).
\end{align}

An important case is when  $Y_0:= \int_{-\infty}^0 e^{-s}p_{-s}(\cdot,y)\xi(\mathrm dy, \mathrm ds)$, then solution to \eqref{eq: EW eqn} is a stationary (in time) Gaussian field with covariance function
\begin{align*}
G(x,y):= 2\int_0^\infty e^{-t} p_t(x,y) \mathrm ds,\qquad x,y \in M.
\end{align*}
That is, for any $\varphi,\psi \in \mathcal S$ and $t \ge 0$, $Y_t(\varphi)$ is a mean zero Gaussian random variable and
\begin{align*}
\E\left[ Y \pr{t,\varphi} Y\pr{t,\psi} \right] = \iint_{M^2} \varphi(x)G(x,y) \psi(y) \mu(\mathrm dy)\mu(\mathrm dx).
\end{align*}

\subsection{Wick Powers}
Recall that given any $\sigma \in \R$ and $n \in \mathbb N$, the $n$-th Hermite polynomial with variance $\sigma^2$ on $\R$ can be inductively defined by setting $H_0(x,\sigma^2) = 1$, and for $n \ge 1$,
\begin{align*}
H_n(x,\sigma^2) = xH_{n-1}(x,\sigma^2) - \sigma^2 \frac{\mathrm d}{\mathrm dx}H_{n-1}(x,\sigma^2).
\end{align*}
For the rest of this paper, we will denote
\begin{align}\label{eq:alpha_0}
\alpha_0:= \frac{d_h - d_w}{2} \ge 0.
\end{align}
and
\begin{align}\label{eq:n0}
n_0 = \max\left\{n \in \N: n < \frac{d_h}{d_h - d_w}\right\}.
\end{align}
Denote $\nu$ as the law of the stationary solution to \eqref{eq: EW eqn}, i.e. 
\begin{align}\label{def:GFF measure}
    \nu := \mathrm {Law}\pr{\int_{-\infty}^0 e^{s(L+1)} \xi\pr{\mathrm ds}}
\end{align}
We denote
\begin{align}\label{eq:green function}
G_\varepsilon(x,y):= 2 \int_0^\infty e^{-t} p_{t +2\varepsilon}(x,y) \mathrm dt,\qquad x,y \in M,
\end{align}
and $C_\varepsilon(x):= G_\varepsilon(x,x)$ for $x \in M$, which will be the counter term in the Wick renormalization. By elementary computation using \eqref{ineq:sgu}, we see it holds uniformly in $x \in M$ and  $\varepsilon\in(0,1]$ that
\begin{align*}
    C_\varepsilon(x) \asymp \begin{cases}
        1+\log\pr{\varepsilon^{-1}},&\text{ for }d_h = d_w\\
        \varepsilon^{-\frac{d_h - d_w}{d_w}},& \text{ for }d_h >d_w.
    \end{cases}
\end{align*}
Observe that it is spatially dependent rather than a constant. For the rest of this section, we shall fix an arbitrary $\alpha >\alpha_0$.
\begin{definition}\label{def:wick renormaliztion of distributions}
Let $X \in \mathcal C(\R_+; \Binf{-\alpha})$ be a time-dependent distribution (on $M$) and $\varepsilon>0$, we define its $\varepsilon$-approximate Wick power of $n$-th order to be
\begin{align*}
  X^{:n:}_\varepsilon(t,x):=H_n\pr{X_\varepsilon(t,x), C_\varepsilon(x)},\qquad (t,x) \in \R_+\times M,
\end{align*}
where $X_\varepsilon(t,x):= P_\varepsilon(X_t)(x)$ for $(t,x) \in \R_+\times M$ and $H_n$ as in the $n$-th Hermit polynomial. 

\noindent Let $\mathfrak{A}\subset \mathcal C(\R_+;\Binf{-\alpha})$ be the subset of all time-dependent distributions $X$ so that the collection of distributions
\begin{align*}
  \underline{\mathbf X}_{n_0}:=\pr{X, X^{:2:},\dots,X^{:n_0:}} := \lim_{\varepsilon\downarrow 0}\pr{X_\varepsilon^{:k:}}_{k = 1}^{n_0}\in \bigoplus_{k = 1}^{n_0} \mathcal C\pr{\R_+; \Binf{-\alpha k}}.
\end{align*}
\end{definition}

\noindent The proof of the following lemma can be found in the Appendix.% This is shown in the following lemma, whose  proof can be found in the Appendix.
\begin{lemma}\label{lem: holder regularity of mollifed y}
Let $n \in \N$, $\varepsilon >0$ and $Y$ be solution to \eqref{eq: EW eqn} with $Y_0\in \Binf{-\alpha}$ for all sufficiently small $\delta_1,\delta_2 >0$,
\begin{align*}
  Y_\varepsilon^{:n:} \in \mathcal C^{\delta_1}\left(\R_+; \mathcal C^{\delta_2}_M\right)
\end{align*}
\end{lemma}

To construct renormalized Wick powers of $Y$, we follow the approach outlined in \cite{chandra2025non}, starting with the following definition.
\begin{definition}[{\cite[Definition 3.11]{chandra2025non}}] \label{def:c set}
Denote by $\mathfrak{C}$ the set of all $Y_0 \in\Binf{-\alpha}$ so that the stochastic processes $\{Y^{:k:}_t\}_{t \in [0,1]}$ for $k = 1,\dots, n_0$ converges in $\mathcal C([0,1];\Binf{-\alpha k})$ in $L^q_\Omega$ for every $q \in [1,\infty)$ and $\P$-almost surely as $\varepsilon \downarrow 0$.
\end{definition}

The following theorem is the main result of this subsection which shows the Wick power defined above is meaningful and has a limit in an appropriate Besov space. In the following theorem, we use the convention that $\frac{a}{0} = \infty$ for $a >0$.
\begin{theorem}\label{thm:Wick Power Convergence}
Let   $Y_0\sim \nu$, then
$
\underline{\mathbf Y}_{n_0}^\varepsilon:= \pr{Y^{:k:}_\varepsilon}_{k = 1}^{n_0}
$ converges in $ \bigoplus_{k = 1}^{n_0} \mathcal C\pr{[0,T];\Binf{-\alpha k }}$ for each $T>0$. More specifically, there exists $\underline{\mathbf Y}_{n_0}=(Y, Y^{:2:},\dots, Y^{:n_0:}) \in \bigoplus_{k = 1}^{n_0}\mathcal C\pr{\R_+;\Binf{-\alpha k }}$  so that for each $T>0$ and $p \ge 1$,
\begin{align*}
  \lim_{\varepsilon\downarrow 0} \mathbb E_{Y_0 \sim \nu} \left[\norm{Y_\varepsilon^{:n:}-Y^{:n:}}_{\mathcal C_T^\gamma \mathcal B^{-\alpha n}_\infty}^p\right] = 0,\qquad 1 \le n \le n_0,
\end{align*}
for all sufficiently small $\gamma>0$. In addition,  $\nu(\mathfrak{C}) = 1$.
\end{theorem}
%\begin{remark}
Note that $Y^{:1:} = Y$ in Theorem \ref{thm:Wick Power Convergence}. In light of Theorem \ref{thm:Wick Power Convergence}, we \textit{define the Wick powers} of $Y$ as the limit of $Y^{:n:}_\varepsilon$ as $\varepsilon \downarrow 0$. We will need a few lemmas to prove it.
%\end{remark}
\begin{remark}
Let $Y_0 \in \Binf{-\alpha}$ so that $Y \in \mathcal C(\R_+; \Binf{-\alpha})$ is the solution to \eqref{eq: EW eqn}, then  $Y^{:n:}_\varepsilon$ is equal to the projection of $Y^n_\varepsilon$ onto $\mathcal{H}^n$ (c.f. \cite[Exercise 1.1.1]{nualart2006malliavin}).
\end{remark}
\noindent The following Lemma is a consequence of semi-group property.
\begin{lemma}\label{lem: q_ku decomposition}
For $u>0$ and $k \in \N$, we have
\begin{align*}
  q_{k,u}(x,y) = \int_M q_{k,u/2}(x,z) p_{u/2}(z,y) \mu(\mathrm dz).
\end{align*}
\end{lemma}

\begin{lemma}\label{lem:greens function estimates}
Denote $r_0:= \mathrm{diam}(M)$
\begin{itemize}
  \item Suppose $a> d_w - d_h$. The following inequality holds uniformly in  $y_1,\, y_2 \in M$ with $y_1 \neq y_2$,
    \begin{align*}
      \int_0^\infty s^{-a/d_w} e^{-s}p_s(y_1,y_2) \mathrm ds \lesssim d(y_1,y_2)^{d_w - d_h -a}.
    \end{align*}
  \item Suppose $a = d_w - d_h$. For any $\varepsilon >0$, the following inequality holds uniformly in $t \in (0,T]$ and $y_1,y_2 \in M$ with $d(y_1, y_2) >0$,
    \begin{align*}
      \int_0^\infty t^{-a/d_w} e^{-t}p_t(y_1,y_2) \mathrm ds \asymp 1+\log\pr{d(y_1,y_2)^{-1}}.
    \end{align*}
  \item For any $a \in \R$, the following inequality holds uniformly in $y_1,y_2 \in M$ with $d(y_1,y_2) >0$,
    \begin{align*}
      d(y_1,y_2)^{d_w - d_h - a} \lesssim  \int_0^1 t^{-a/d_w} p_t(y_1,y_2) \mathrm dt \lesssim \int_0^\infty t^{-a/d_w} e^{-t}p_t(y_1,y_2) \mathrm ds
    \end{align*}
\end{itemize}
\end{lemma}
\begin{proof}
Let $y_1,y_2 \in M$ and denote $r := d(y_1,y_2)$. Suppose first that $a > d_w - d_h$. Using the change of variable with $u = rs^{-\frac{1}{d_w}}$, we see that
\begin{align*}
  \int_0^\infty s^{-a/d_w}e^{-s} p_s(y_1,y_2) \mathrm ds &\overset{\eqref{ineq:sgu}}{\asymp} \int_0^{r_0}  s^{-(d_h+a)/d_w} \exp\left( -c \left(r / s^{1/d_w} \right)^\frac{d_w}{d_w-1}\right)\mathrm ds + 1\\
  &\lesssim  r^{d_w - d_h - a} \int_{r\cdot r_0^{-\frac{1}{d_w}}}^\infty u^{d_h +a - d_w } \exp \left(- c u^\frac{d_w}{d_w - 1}\right)\frac{\mathrm du}{u}+1\\
  &\lesssim r^{d_w - d_h - a} \int_{0}^\infty u^{d_h +a - d_w } \exp \left(- c u^\frac{d_w}{d_w - 1}\right)\frac{\mathrm du}{u} +1\\
  &\lesssim r^{d_w - d_h -a},
\end{align*}
where we've used the fact that $d_w - d_h - a<0$ and $0<r\le r_0<\infty$.

Suppose now $a =  d_w - d_h$, then by another change of variable, we see for any $T>0$ and $\varepsilon>0$, the following inequality holds uniformly in $t \in (0,T]$ and $y_1,y_2 \in M$ with $d(y_1,y_2)>0$:
\begin{align*}
  \int_0^\infty s^{-a/d_w}e^{-s} p_s(y_1,y_2) \mathrm ds &\lesssim \pr{\int_0^{r^{d_w}} +\int_{r^{d_w}}^{r_0^{d_w}}} s^{-1} \exp \pr{-c \pr{r/s^\frac{1}{d_w}}^{\frac{d_w}{d_w-1}}}\mathrm ds +1\\
  &\lesssim \int_1^\infty \exp \pr{-cu^\frac{d_w}{d_w - 1}} \frac{du}{u} + \int_{r/r_0}^{1} \frac{\mathrm du}{u} + 1\\
  &\lesssim 1+ \log\pr{\frac{1}{r}}.
\end{align*}
For the last inequality, we have
\begin{align*}
  \int_0^\infty t^{-a/d_w}e^{-t} p_t(y_1,y_2) \mathrm dt &\gtrsim \int_0^1 t^{-a/d_w} p_t(y_1,y_2) \mathrm dt \\
  &\overset{\eqref{ineq:sgu}}{\gtrsim} \int_0^1 t^{-(a+d_h)/d_w}\exp\left(- C \left( \frac{r}{t^\frac{1}{d_w}}\right)^\frac{d_w}{d_w - 1} \right) dt\\
  &\gtrsim r^{d_w - d_h - a} \int_r^\infty t^{d_h +a-d_w - 1} \exp\left( -C t^\frac{d_w}{d_w - 1}\right)\mathrm dt\\
  &\gtrsim r^{d_w - d_h -a} \int_{r_0}^\infty t^{d_h +a-d_w - 1} \exp \left( -C t^\frac{d_w}{d_w - 1}\right)\mathrm dt\\
  &\gtrsim r^{d_w - d_h - a}.
\end{align*}
Furthermore, if $a = d_w - d_h$, then
\begin{align*}
  \int_0^\infty t^{-a/d_w}e^{-t} p_t(y_1,y_2) \mathrm dt &\gtrsim \int_0^{r_0} t^{-a/d_w}p_t(y_1,y_2) \mathrm dt\\
  &\gtrsim \int_{r/r_0}^\infty \exp \pr{-ct^\frac{d_w}{d_w - 1}}\mathrm dt\\
  &\gtrsim \int_{r/r_0}^1 \frac{\mathrm dt}{t} + \int_1^\infty \exp\pr{-ct^\frac{d_w}{d_w -1}}\mathrm dt\\
  &\gtrsim 1+ \log\pr{\frac{1}{r}}.
\end{align*}
\end{proof}

\begin{lemma}\label{lem: continuity of time integral of heat kernel}
For $\varepsilon >0$, $i = 1,2$ and $(t_i,y_i) \in \R_+ \times M$, let $$\mathcal K_\varepsilon(t_1,t_2,y_1,y_2) := \int_{-\infty}^{t_1 \wedge t_2}e^{-(t_1+t_2-2s)}p_{t_1+t_2+2\varepsilon - 2s}(y_1,y_2) \mathrm ds.$$
Then it holds  uniformly for $y_1\neq y_2 \in M$, $0 \le t_1,t_2 $ and $\varepsilon \in [0,1)$ that
\begin{align}\label{ineq: cal K bounds}
  \mathcal K_\varepsilon(t_1,t_2,y_1,y_2) \lesssim
  \begin{cases}
    d(y_1,y_2)^{d_w - d_h} &\text{ if }d_w < d_h\\
    1+\log\pr{\frac{r_0}{r}} &\text{ if }d_w = d_h,
  \end{cases}
\end{align}
where $r_0:= \mathrm{diam}(M)$.
For  $\delta \in (0,d_w]$, it holds uniformly for $y_1,y_2 \in M$, $\varepsilon,\varepsilon' \in (0,1]$ and $0 \le t_1,t_2 $ that
\begin{equation}
  \begin{aligned}\label{ineq: cal k time holder}
    \left|\mathcal K_\varepsilon(t_1,t_1,y_1,y_2) - \mathcal K_\varepsilon(t_1,t_2,y_1,y_2) \right| &\lesssim d(y_1,y_2)^{d_w-d_h - \delta }|t_1 - t_2|^\frac{\delta}{d_w},\\
    \abs{\mathcal K_\varepsilon(t_1,t_1,y_1,y_2) - \mathcal K_{\varepsilon'}(t_1,t_1,y_1,y_2)} &\lesssim d(y_1,y_2)^{d_w - d_h - \delta}\abs{\varepsilon' - \varepsilon }^\frac{\delta}{d_w}.
  \end{aligned}
\end{equation}
\end{lemma}

\begin{proof}
Let us denote $r := d(y_1,y_2)$, we see by Lemma \ref{lem:greens function estimates} that uniformly in $\varepsilon \in[0,1]$, $0<t_1,t_2$ and $y_1,y_2 \in M$ with $d(y_1,y_2) >0$,
\begin{align*}
  \mathcal K_\varepsilon(t_1,t_2,,y_1,y_2) = \frac{e^{2\varepsilon}}{2}\int_{|t_1 - t_2|+2\varepsilon}^{\infty}e^{-s} p_{s}(y_1,y_2) ds
  \lesssim \int_{0}^{\infty}e^{-s} p_{s}(y_1,y_2) ds
  \lesssim
  \begin{cases}
    r^{d_w - d_h},&\text{ if }d_h>d_w\\
    1+\log\pr{\frac{r_0}{r}},&\text{ if }d_h = d_w
  \end{cases}
  ,
\end{align*}
which gives \eqref{ineq: cal K bounds}.

For \eqref{ineq: cal k time holder},  let $\delta\in (0,d_w)$ and $\eta_\delta := (d_h - d_w+\delta) \frac{d_w - 1}{d_w}$. Observe that by a change of variable the following inequality holds uniformly in $\varepsilon \in(0,1]$, $0<t_1,t_2 $ and $y_1\neq y_2 \in M$:
\begin{align*}
  \left|\mathcal K_\varepsilon(t_1,t_1,y_1,y_2) - \mathcal K_\varepsilon(t_1,t_2,y_1,y_2) \right| &= \frac{e^{2\varepsilon}}{2} \int_{2\varepsilon}^{\abs{t_1-t_2}+2\varepsilon} e^{-s}p_s(y_1,y_2) \mathrm ds \\
  &\lesssim \int_{2\varepsilon}^{\abs{t_1-t_2}+2\varepsilon} s^{-\frac{d_h}{d_w}}\exp \pr{-cr^\frac{d_w}{d_w-1}s^{- \frac{1}{d_w - 1}}}\mathrm ds\\
  &\lesssim r^{-\frac{d_w}{d_w - 1}\eta_\delta} \int_{2\varepsilon}^{\abs{t_1-t_2}+2\varepsilon} s^{-\frac{d_h}{d_w}}s^{\frac{1}{d_w-1}\eta_\delta} \mathrm ds\\
  &=r^{d_w - d_h - \delta}\int_{2\varepsilon}^{\abs{t_1-t_2   } + 2\varepsilon} s^\frac{\delta}{d_w} \frac{\mathrm ds}{s}\\
  &\lesssim r^{d_w - d_h - \delta}\int_{0}^{\abs{t_1-t_2   } } s^\frac{\delta}{d_w} \frac{\mathrm ds}{s}\\
  &\lesssim r^{d_w - d_h - \delta} \abs{t_1-t_2}^\frac{\delta}{d_w}.
\end{align*}
Where we used the elementary fact that for every $p>0$, $\exp \left(-a \right) \lesssim a^{-p}$ uniformly in $a >0$ in the second and third line. Finally, we observe that for $1\ge \varepsilon' > \varepsilon >0$, and $t_2 := t_1+2(\varepsilon'-\varepsilon)$,
\begin{align*}
  \mathcal K_{\varepsilon'}(t_1,t_1,y_1,y_2) = e^{2(\varepsilon'-\varepsilon)}\int_{-\infty}^{t_1\wedge t_2} e^{-(t_1+t_2-2s )}p_{t_1+t_2 +2\varepsilon}(y_1,y_2) \mathrm ds =e^{2(\varepsilon'-\varepsilon)} \mathcal K_{\varepsilon}(t_1,t_2, y_1,y_2).
\end{align*}
Therefore, by the first inequality in \eqref{ineq: cal k time holder}, it holds uniformly in $t_1\ge 0$, $y_1,y_2 \in M$ and $\varepsilon, \varepsilon' \in (0,1]$ that
\begin{align*}
  \abs{\mathcal K_\varepsilon(t_1,t_1,y_1,y_2) - \mathcal K_{\varepsilon'}(t_1,t_1,y_1,y_2)}&\le \abs{\mathcal K_\varepsilon(t_1,t_1,y_1,y_2) - \mathcal K_{\varepsilon}(t_1,t_2,y_1,y_2)}\\
  &\qquad+\abs{\mathcal K_\varepsilon(t_1,t_2, y_1,y_2)} \abs{e^{2(\varepsilon' - \varepsilon)} -1}\\
  &\lesssim d(y_1,y_2)^{d_w - d_h - \delta}\abs{\varepsilon' - \varepsilon}^\frac{\delta}{d_w} + \abs{\varepsilon' - \varepsilon}\abs{\mathcal K_\varepsilon(t_1,t_2, y_1,y_2)}\\
  &\overset{\eqref{ineq: cal K bounds}     }{\lesssim} d(y_1,y_2)^{d_w - d_h - \delta}\abs{\varepsilon' - \varepsilon}^\frac{\delta}{d_w},
\end{align*}
where we used the fact that $d_w>\delta>0$ and $1+\log{r_0/r} \lesssim r^{- \delta}$ uniformly in $0<r\le r_0$.
\end{proof}

We will also need the following elementary result, whose proof can be found in the Appendix.
\begin{lemma}\label{lem:imcomplete beta function bounds}
Let $a \in [0,1)$ and $b \in \R$, it holds uniformly in $0<u\le 1$ that
\begin{align*}
  \mathrm I(u):=\int_0^1 s^{-a}(u+s)^{-b} \mathrm ds \asymp
  \begin{cases}
    1,& \text{for } a+b<1\\
    1+\log(1/u), &\text{for } a+b = 1\\
    u^{1-a-b},&\text{for }a+b>1
  \end{cases}
\end{align*}
\end{lemma}

\begin{lemma}\label{lem: holder besov space of wick product}
There exists a sufficiently small $\delta>0$ so that for each $p \in \N$,  the following inequality holds uniformly in $\varepsilon \in (0,1]$ and $t_1,t_2 \ge 0$:
\begin{align*}
  \mathbb E_{Y_0 \sim \nu} \left[\norm{Y_\varepsilon^{:n:}(t_1) - Y_\varepsilon^{:n:}(t_2)}_{\mathcal B^{-\alpha n}_{2p,2p}}^{2p} \right]^\frac{1}{2p}\lesssim \abs{t_1-t_2}^\frac{\delta}{2d_w}.
\end{align*}
Consequently, for any $T>0$,
\begin{align*}
  \sup_{ \varepsilon \in (0,1]}\E_{Y_0 \sim \nu} \left[\norm{Y_\varepsilon^{:n:}}_{\mathcal C_T \mathcal B^{-\alpha n}_{2p,2p}}^{2p}\right] <\infty.
\end{align*}
\end{lemma}
\begin{proof}
For $Y_0 := \int_{-\infty}^0e^{s}p_{0-s}(\cdot,y) \xi\pr{\mathrm dy,\mathrm ds}$, $Y_\varepsilon(s,\cdot)$ is a random field given by
\begin{align*}
  Y_\varepsilon(s,x)= \int_{-\infty}^t \int_M e^{-(t-s)} p_{t-s+\varepsilon}\pr{x, y} \xi\pr{\mathrm dy,\mathrm ds },\qquad (t,x) \in \R_+\times M.
\end{align*}
By Fubini and Proposition \ref{prop: Nelson estimate}, we see for each $p' \in \mathbb N$ and and $k = 0,1,\dots$, the following inequality holds uniformly for $t_1,t_2 , u> 0$:
\begin{equation}\label{ineq:time reg wick power step 1}
  \begin{aligned}
    &\E\left[\norm{Q^{(k)}_u \left(Y_\varepsilon^{:n:}(t_1) -Y_\varepsilon^{:n:}(t_2)  \right)}_{L^{2p'}_M}^{2p'}\right]\\ =& \int_M \norm{\int_M q_{k,u}(x,y)\left(Y_\varepsilon^{:n:}(t_1,y) -Y_\varepsilon^{:n:}(t_2,y)  \right)\mu(\mathrm dy)}_{L^{2p'}_\Omega}^{2p'} \mu(\mathrm dx)\\
    \lesssim& \int_M \E\left[\left(\int_M q_{k,u}(x,y)\left(Y_\varepsilon^{:n:}(t_1,y) -Y_\varepsilon^{:n:}(t_2,y)  \right)\mu(\mathrm dy)\right)^2\right]^{p'} \mu(\mathrm dx).
  \end{aligned}
\end{equation}

By Wick's theorem, we see for $t_1,t_2 >0$ and $y_1,y_2 \in M$,
\begin{align*}
  \E\left[Y_\varepsilon^{:n:}(t_1,y_1)Y^{:n:}_\varepsilon(t_2,y_2) \right] &= n!\E\left[Y_\varepsilon(t_1,y_1)Y_\varepsilon(t_2,y_2) \right]^n\\
  &=n! \left( \int_0^{t_1 \wedge t_2} \int_M  p_{t_1 - s + \varepsilon}(y_1,z) p_{t_2 - s+\varepsilon}(y_2,z) \mu(\mathrm dz)\mathrm ds\right)^n\\
  &=n! \left(\int_{-\infty}^{t_1 \wedge t_2}e^{-\pr{t_1+t_2 - 2s}}p_{t_1+t_2 - 2s +2\varepsilon}(y_1,y_2) \mathrm ds\right)^n\\
  &=n! \mathcal K_\varepsilon(t_1,t_2,y_1,y_2)^n,
\end{align*}
where $\mathcal K_\varepsilon$ is as in Lemma \ref{lem: continuity of time integral of heat kernel}. By Fubini,
\begin{equation}\label{eq: holder besov step 1}
  \begin{aligned}
    &\quad\,\,\E\left[\left(\int_M q_{k,u}(x,y)\left(Y_\varepsilon^{:n:}(t_1,y) -Y_\varepsilon^{:n:}(t_2,y)  \right)\mu(\mathrm dy)\right)^2\right]\\
    &=\int_{M}\int_M q_{k,u}(x,y_1)q_{k,u}(x,y_2) \bigg(\E\left[Y_\varepsilon^{:n:}(t_1,y_1) Y_\varepsilon^{:n:}(t_1,y_2) \right] - \E\left[Y_\varepsilon^{:n:}(t_2,y_1) Y_\varepsilon^{:n:}(t_1,y_2) \right]\\
    &\quad\quad\quad\quad\quad\quad\quad+\E\left[Y_\varepsilon^{:n:}(t_2,y_1) Y_\varepsilon^{:n:}(t_2,y_2) \right] - \E\left[Y_\varepsilon^{:n:}(t_1,y_1) Y_\varepsilon^{:n:}(t_2,y_2) \right]\bigg) \mu(\mathrm dy_1)\mu(\mathrm dy_2)\\
    &=n!\int_{M}\int_M q_{k,u}(x,y_1) q_{k,u}(x,y_2) \bigg(\mathcal K_\varepsilon(t_1,t_1,y_1,y_2)^n - \mathcal K_\varepsilon(t_2,t_1,y_1,y_2)^n\\
    &\quad\quad\quad\quad\quad\quad\quad\quad\quad\quad\quad\quad\quad\quad\quad+ \mathcal K_\varepsilon(t_2,t_2,y_1,y_2)^n  - \mathcal K_{\varepsilon}(t_1,t_2,y_1,y_2)^n\bigg)\mu(\mathrm dy_1)\mu(\mathrm dy_2).
  \end{aligned}
\end{equation}

Since $d_h \ge d_w$ and $(M,d)$ is a compact metric space, Lemma \ref{lem: continuity of time integral of heat kernel} gives for any $\delta \in (0,d_w]$, we have uniformly in $0 \le t_1,\, t_2\le T$, $y_1,\, y_2 \in M$ and $\varepsilon \in (0,1]$
\begin{align*}
  \abs{\mathcal K_\varepsilon(t_1,t_1,y_1,y_2)^n - \mathcal K_\varepsilon(t_2,t_1,y_1,y_2)^n}&\lesssim  \abs{\mathcal K_\varepsilon(t_1,t_1,y_1,y_2) - \mathcal K_\varepsilon(t_2,t_1,y_1,y_2)}\\
  &\qquad\qquad\times \pr{\sum_{k = 1}^n\abs{\mathcal K_\varepsilon(t_1,t_1,y_1,y_2)^{k}\mathcal K_\varepsilon(t_1,t_1,y_1,y_2)^{n-k-1}}}\\
  &\overset{\eqref{ineq: cal k time holder}}{\lesssim} \abs{t_1 - t_2}^\frac{\delta}{d_w} d(y_1,y_2)^{d_w - d_h - \delta}\pr{\sup_{t_1,t_2 >0,\atop y_1\neq y_2 \in M} \mathcal K_\varepsilon(t_1,t_2,y_1,y_2)}^{n-1}\\
  &\overset{\eqref{ineq: cal K bounds}}{\lesssim } \abs{t_1-t_2}^\frac{\delta}{d_w}d(y_1,y_2)^{n(d_w - d_h) - \delta},
\end{align*}
where we used the fact that $(1+\log(r_0/r))^n \lesssim r^{-\delta}$ uniformly in $0<r\le r_0< \infty$.

By Proposition \ref{prop:kernel_bound_dt}, there is some $c>0$ so that uniformly in $(t,x,y) \in (0,\infty)\times M\times M$,
\begin{align*}
  \abs{q_{k,u}(x,y)} \lesssim p_{ct}(x,y).
\end{align*}
Now let $\delta>0$ be sufficiently small so that $(n-1) (d_h - d_w) + \delta < d_w$. Inserting above inequalities into \eqref{eq: holder besov step 1} and use Lemma \ref{lem:greens function estimates} to replace $d(y_1,y_2)^{-n(d_w-d_h) - \delta}$ by the time integral of the heat kernel to see
\begin{align*}
  &\E\left[\left(\int_M q_{k,u}(x,y)\left(Y_\varepsilon^{:n:}(t_1,y) -Y_\varepsilon^{:n:}(t_2,y)  \right)\mu(\mathrm dy)\right)^2\right]\\\lesssim&\abs{t_1 - t_2}^\frac{\delta}{d_w}\int_M \int_M p_{cu}(x,y_1)p_{cu}(x,y_2)d(y_1,y_2)^{n(d_w - d_h) - \delta}\mu(\mathrm dy_1) \mu(\mathrm dy_2)\\
  \lesssim &\abs{t_1 - t_2}^\frac{\delta}{d_w} \int_M \int_M p_{cu}(x,y_1)p_{cu}(x,y_2)\int_0^1 s^{-\frac{(n-1)(d_h - d_w) + \delta}{d_w}} p_s(y_1,y_2) \mathrm ds \mu(\mathrm dy_1) \mu(\mathrm d y_2)\\
  =&\abs{t_1 - t_2}^\frac{\delta}{d_w} \int_0^1 s^{- \frac{(n-1)(d_h-d_w)+\delta}{d_w}}p_{s+2cu}(x,x) \mathrm ds\\
  \lesssim &\abs{t_1 - t_2}^\frac{\delta}{d_w} \int_0^1 s^{- \frac{(n-1)(d_h-d_w) +\delta}{d_w}}(s+u)^{-\frac{d_h}{d_w}}\mathrm ds\\
  \lesssim& \abs{t_1-t_2}^\frac{\delta}{d_w} u^{\frac{ -n(d_h - d_w) - \delta}{d_w}},
\end{align*}
where we used Lemma \ref{lem:imcomplete beta function bounds} and the fact that for $\delta >0$ and $d_h \ge d_w \ge 1$,
\begin{align*}
  1 < \frac{d_h +\delta }{d_w}< \frac{(n-1)(d_h - d_w) +\delta +d_h}{d_w}.
\end{align*}
Inserting this into \eqref{ineq:time reg wick power step 1} and use the compactness of $M$ to see for each $p' \in \mathbb N$, we have
\begin{align*}
  \E\left[\norm{Q^{(k)}_u \left(Y_\varepsilon^{:n:}(t_1) -Y_\varepsilon^{:n:}(t_2)  \right)}_{L^{2p'}_M}^{2p'}\right]^\frac{1}{2p'} \lesssim |t_1 - t_2|^\frac{\delta}{2d_w} u^{- \frac{n(d_h - d_w) +\delta}{2d_w}}.
\end{align*}

Now let $\alpha>\alpha_0$, select $\delta>0$ so that $\alpha > \alpha_0 +\delta $, we see by Minkowski's inequality, it holds uniformly for $t_1,t_2 \ge0$ and $\varepsilon \in (0,1]$ that
\begin{align*}
  \E\left[\norm{Y_\varepsilon^{:n:}(t_1) - Y_\varepsilon^{:n:}(t_2)}_{\mathcal B^{-\alpha n}_{2p',2p'}}^{2p'}\right]^\frac{1}{2p'}&\le \mathbb E \left[\norm{P_1 Y_\varepsilon^{:n:}(t_1) - P_1 Y_\varepsilon^{:n:}(t_2)}_{L^{2p'}_M} ^{2p'}\right]^\frac{1}{2p'} \\
  &\qquad + \pr{\int_0^1 \pr{u^{\frac{\alpha n}{d_w}}\norm{\norm{Q^{(k)}_u \pr{Y_\varepsilon^{:n:}(t_1) -Y_\varepsilon^{:n:}(t_2) }(\cdot)}_{L^{2p}_\Omega} }_{L^{2p'}_M}}^{2p'} \frac{\mathrm ds}{s}}^\frac{1}{2p}\\
  &\lesssim\abs{t_1-t_2}^\frac{\delta}{2d_w} \pr{1+ \pr{\int_0^1 \pr{u^\frac{\alpha n}{d_w}u^{- \frac{\alpha_0 n +\delta/2}{d_w}}}^{2p'} \frac{\mathrm du}{u}}^\frac{1}{2p'}}\\
  &\lesssim \abs{t_1-t_2 }^\frac{\delta}{2d_w}.
\end{align*}
Therefore, let $p'$ be sufficiently large and use  Kolmogorov's continuity test to see for each $T>0$, there exists $C = C(\gamma,\alpha, p',T)>0$ so that
\begin{align*}
  \E\left[ \norm{Y_\varepsilon^{:n:}}_{\mathcal C^\gamma_T \mathcal B^{-\alpha n}_{2p',2p'}} ^{2p'}\right] \lesssim   C,
\end{align*}
for all $\gamma \in \left(0, \frac{\delta}{2d_w} - \frac{1}{2p'}\right)$, as desired.
\end{proof}

%We are now ready to prove Theorem \ref{thm:Wick Power Convergence}.
\begin{proof}[Proof of Theorem \ref{thm:Wick Power Convergence}]
We will again assume without loss of generality that
\begin{align*}
  Y_0 = \int_{-\infty}^0 \int_Me^{s}p_s(\cdot, y) \xi\pr{\mathrm dy,\mathrm ds}
\end{align*}
Let $n \in \N$, $\varepsilon,\varepsilon' \in (0,1]$ and assume without loss of generality that $\varepsilon' > \varepsilon$, denote for $(t,y) \in \R_+ \times M$ that $\widetilde Y^{(n)}_{\varepsilon,\varepsilon'}(t,y): = Y^{:n:}_\varepsilon(t,y) - Y^{:n:}_{\varepsilon'}(t,y)$. For $y_1,y_2 \in M$, denote
\begin{align*}
  \Delta^{(n)}_{\varepsilon,\varepsilon'}(t,y_1,y_2):= Y_{\varepsilon}^{:n:}(t,y_1)Y_{\varepsilon'}^{:n:}(t,y_2).
\end{align*}
Note that by Wick's Theorem, for any $\varepsilon,\varepsilon' \in (0,1]$, $t\ge 0$ and $y_1,y_2 \in M$,
\begin{align*}
  \E \left[ \Delta^{\pr{n}}_{\varepsilon, \varepsilon'}(t,y_1,y_2)\right]  = n!\pr{\int_{-\infty}^t e^{-(t-s)}p_{2(t-s) +\varepsilon'+\varepsilon}\mathrm ds}^n = n! \pr{\mathcal K_{\frac{\varepsilon+\varepsilon'}{2}}(t,t,y_1,y_2)}^n,
\end{align*}
and
\begin{align*}
  \widetilde Y^{(n)}_{\varepsilon,\varepsilon'}(t_1,y) - \widetilde Y^{(n)}_{\varepsilon,\varepsilon'}(t_2,y) = Y^{:n:}_{\varepsilon}(t_1,y) - Y^{:n:}_{\varepsilon'}(t_1,y) - Y^{:n:}_{\varepsilon}(t_2,y) + Y^{:n:}_{\varepsilon'}(t_2,y).
\end{align*}
Let $\delta \in (0,d_w)$, it holds uniformly in $t \ge 0$, $\varepsilon, \varepsilon' \in (0,1]$ and $y_1,y_2 \in M$ that
\begin{align*}
  &\abs{\E\bigg [\Delta^{(n)}_{\varepsilon,\varepsilon}(t,y_1,y_2)-2\Delta^{(n)}_{\varepsilon,\varepsilon'}(t,y_1,y_2)+\Delta^{(n)}_{\varepsilon',\varepsilon'}(t,y_1,y_2) \bigg]}\\
  =&n!\abs{ \mathcal K_\varepsilon(t,t,y_1,y_2)^n - 2\mathcal K_{\frac{\varepsilon+\varepsilon'}{2}}(t,t,y_1,y_2)^n +\mathcal K_{\varepsilon'}(t,t,y_1,y_2)^n}\\
  \lesssim&\abs{\mathcal K_{\varepsilon}(t,t,y_1,y_2) - \mathcal K_{\frac{\varepsilon+\varepsilon'}{2}}(t,t,y_1,y_2)}\abs{\sum_{k = 0}^{n-1} \mathcal K_{\varepsilon}(t,t,y_1,y_2)^k \mathcal K_{\frac{\varepsilon+\varepsilon'}{2}}(t,t,y_1,y_2)^{n-1-k}}\\
  &\qquad+ \abs{\mathcal K_{\varepsilon'}(t,t,y_1,y_2) - \mathcal K_{\frac{\varepsilon+\varepsilon'}{2}}(t,t,y_1,y_2)}\abs{\sum_{k = 0}^{n-1} \mathcal K_{\varepsilon'}(t,t,y_1,y_2)^k \mathcal K_{\frac{\varepsilon+\varepsilon'}{2}}(t,t,y_1,y_2)^{n-1-k}}\\
  &\lesssim d(y_1,y_2)^{n(d_w - d_h) - \delta}\abs{\varepsilon - \varepsilon'}^\frac{\delta}{d_w},
\end{align*}
where we used Lemma \ref{lem: continuity of time integral of heat kernel} in the last line.

Hence, for any $k = 0,1,\dots$, we may follow the proof of Lemma \ref{lem: holder besov space of wick product} to see it holds uniformly for $u, \varepsilon, \varepsilon' \in (0,1]$ and $(t,x) \in \R_+\times M$ that
\begin{align*}
  \E\left[\left(Q^{(k)}_u  \pr{\widetilde Y^{:n:}_{\varepsilon,\varepsilon'}(t)}(x) \right)^2 \right]
  =& \int_{M^2}q_{k,u}(x,y_1) q_{k,u}(x,y_2) \E\bigg [\Delta^{(n)}_{\varepsilon,\varepsilon}(t,y_1,y_2)\\
  &\qquad-2\Delta^{(n)}_{\varepsilon,\varepsilon'}(t,y_1,y_2)+\Delta^{(n)}_{\varepsilon',\varepsilon'}(t,y_1,y_2) \bigg] \mu(\mathrm dy_1) \mu(\mathrm dy_2)\\
  &\lesssim \iint_{M^2} p_{cu}(x,y_1)p_{cu}(x,y_2) d(y_1,y_2)^{n(d_w - d_h) - \delta} \abs{\varepsilon - \varepsilon'}^\frac{\delta}{d_w}\mu(\mathrm dy_1) \mu(\mathrm dy_2)\\
  &\lesssim \iint_{M^2}p_{cu}(x,y_1)p_{cu}(x,y_2) \int_0^1s^{-\frac{(n-1)(d_h - d_w) +\delta}{d_w}}p_s(y_1,y_2) \mathrm ds  \mu(y_1) \mu(y_2)\\
  &=\abs{\varepsilon-\varepsilon'}^\frac{\delta}{d_w}\int_0^1 s^{-\frac{(n-1)(d_h - d_w) +\delta}{d_w}} p_{2cu +s}\mathrm ds \\
  &\lesssim \abs{\varepsilon - \varepsilon'}^\frac{\delta}{d_w} u^{-\frac{n(d_h - d_w) +\delta   }{d_w}},
\end{align*}
where we used Lemma \ref{lem:greens function estimates} in the second line and Lemma \ref{lem:imcomplete beta function bounds} in the last line. Hence, by Proposition \ref{prop: Nelson estimate}, we see for each $p \ge 1$ and $\delta \in (0, d_w)$ and $k = 0,1,\dots, $ it holds uniformly in $(t,x) \in \R_+\times M$ that
\begin{align*}
  \norm{Q^{(k)}_u\pr{\widetilde Y_{\varepsilon,\varepsilon'}^{(n)}(t)}(x)}_{L^{2p}_\Omega} \lesssim \abs{\varepsilon - \varepsilon'}^\frac{\delta}{2d_w} u^{-\frac{\alpha_0 n +\delta/2}{d_w}}.
\end{align*}
Hence, for any $\alpha' > \alpha_0$, we may select a sufficiently small $\delta >0$ so that $\alpha' > \alpha_0+\delta$ By Fubini's, Minkowski's inequality and Proposition \ref{prop: Nelson estimate}, we see it holds uniformly in $\varepsilon, \varepsilon' \in (0,1]$ and $t \ge 0$ that
\begin{align*}
  \norm{\norm{\widetilde Y_{\varepsilon,\varepsilon'}^{(n)}(t)}_{\mathcal B^{-\alpha' n}_{2p,2p}}}_{L^{2p}_\Omega}&\le \norm{ \norm{P_1\pr{\widetilde Y_{\varepsilon,\varepsilon'}^{(n)}(t)}(\cdot)}_{L^{2p}_\Omega} }_{L^{2p}_M}\\
  &\qquad \qquad+ \pr{\int_0^1 \pr{u^{\frac{\alpha' n }{d_w}}\norm{\norm{\pr{Q^{(k)}_u\pr{\widetilde Y_{\varepsilon,\varepsilon'}^{(n)}(t)}(\cdot)}}_{L^{2p}_\Omega}}_{L^{2p}_M} }^{2p}\frac{\mathrm du}{u}}^\frac{1}{2p}\\
  &\lesssim \abs{\varepsilon - \varepsilon'}^\frac{\delta}{2d_w}\pr{1+\pr{\int_0^1 \pr{u^{\frac{\alpha' n}{d_w} - \frac{\alpha_0 n - \delta/2}{2d_w}}}^{2p} \frac{\mathrm du}{u}}^\frac{1}{2p}}\lesssim \abs{\varepsilon - \varepsilon'}^\frac{\delta}{2d_w}.
\end{align*}
Therefore, uniformly $\varepsilon,\varepsilon' \in (0,1]$ and $t,t' \ge 0$,
\begin{align*}
  \norm{\norm{\widetilde Y^{(n)}_{\varepsilon,\varepsilon'}(t) -\widetilde Y^{(n)}_{\varepsilon,\varepsilon'}(t')}_{\mathcal B^{-\alpha' n}_{2p,2p}}}_{L^{2p}_\Omega} &= \norm{ \norm{\pr{Y^{:n:}_\varepsilon(t) - Y^{:n:}_{\varepsilon'}(t)}- \pr{Y^{:n:}_\varepsilon(t') - Y^{:n:}_{\varepsilon'}(t')}}_{\mathcal B^{-\alpha' n}_{2p,2p}}}_{L^{2p}_\Omega}\\
  &\le 2\sup_{t \ge0}\norm{\norm{\widetilde Y^{(n)}_{\epsilon,\varepsilon'}(t)}_{\mathcal B^{-\alpha' n}_{2p,2p}}}_{L^{2p}_{\Omega}}\lesssim \abs{\varepsilon - \varepsilon'}^\frac{\delta}{2d_w}.
\end{align*}

On the other hand, as we see in Lemma \ref{lem: holder besov space of wick product},
\begin{align*}
  \norm{\norm{\widetilde Y^{(n)}_{\varepsilon,\varepsilon'}(t) -\widetilde Y^{(n)}_{\varepsilon,\varepsilon'}(t')}_{\mathcal B^{-\alpha' n}_{2p,2p}}}_{L^{2p}_\Omega} &=\norm{ \norm{\pr{Y^{:n:}_\varepsilon(t) - Y^{:n:}_{\varepsilon'}(t)}- \pr{Y^{:n:}_\varepsilon(t') - Y^{:n:}_{\varepsilon'}(t')}}_{\mathcal B^{-\alpha' n}_{2p,2p}}}_{L^{2p}_\Omega}\\
  &\le \norm{\norm{Y^{:n:}_\varepsilon(t) - Y^{:n:}_{\varepsilon}(t')}_{\mathcal B^{-\alpha' n}_{2p,2p}}}_{L^{2p}_\Omega}+\norm{\norm{Y^{:n:}_{\varepsilon'}(t) - Y^{:n:}_{\varepsilon'}(t')}_{\mathcal B^{-\alpha' n}_{2p,2p}}}_{L^{2p}_\Omega}\\
  &\lesssim \abs{t-t'}^\frac{\delta}{2d_w}.
\end{align*}
This implies, for each $\alpha' > \alpha_0$, there exists a sufficiently small $\delta >0$ so that uniformly in $t,t' \ge 0$ and $\varepsilon, \varepsilon' \in (0,1]$,
\begin{align*}
  \norm{\norm{\widetilde Y^{(n)}_{\varepsilon,\varepsilon'}(t) -\widetilde Y^{(n)}_{\varepsilon,\varepsilon'}(t')}_{\mathcal B^{-\alpha' n}_{2p,2p}}}_{L^{2p}_\Omega} \lesssim \abs{\varepsilon - \varepsilon'}^\frac{\delta}{4d_w}\abs{t - t'}^\frac{\delta}{4d_w}.
\end{align*}
We may pick $p \ge 1$ sufficiently large, and  apply Kolmogorov's continuity test (first to time variable and then to $\varepsilon\in (0,1]$) to see there exists $\gamma,\gamma'>0$, so that for each $T>0$
\begin{align*}
  \norm{\sup_{0< \varepsilon< \varepsilon' \le 1}\abs{\varepsilon' - \varepsilon}^{-\gamma'}\norm{Y_\varepsilon^{:n:} - Y_{\varepsilon'}^{:n:}}_{\mathcal C^\gamma_T \mathcal B^{-\alpha' n}_{2p,2p}}}_{L^{2p}_\Omega} < \infty.
\end{align*}
Finally, we may choose $\alpha' \in (\alpha_0,\alpha)$, and choose $p\ge 1$ sufficiently large so that $\alpha'+d_h/p< \alpha$. Then by Proposition \ref{prop:besov_embed}, we see
\begin{align*}
  \norm{\sup_{0< \varepsilon< \varepsilon' \le 1}\abs{\varepsilon' - \varepsilon}^{-\gamma'}\norm{Y_\varepsilon^{:n:} - Y_{\varepsilon'}^{:n:}}_{\mathcal C^\gamma_T \Binf{-n \alpha}}}_{L^{2p}_\Omega}&\lesssim \norm{\sup_{0< \varepsilon< \varepsilon' \le 1}\abs{\varepsilon' - \varepsilon}^{-\gamma'}\norm{Y_\varepsilon^{:n:} - Y_{\varepsilon'}^{:n:}}_{\mathcal C^\gamma_T \Binf{-n \alpha' - d_h/2p}}}_{L^{2p}_\Omega}\\
  &\lesssim\norm{\sup_{0< \varepsilon< \varepsilon' \le 1}\abs{\varepsilon' - \varepsilon}^{-\gamma'}\norm{Y_\varepsilon^{:n:} - Y_{\varepsilon'}^{:n:}}_{\mathcal C^\gamma_T \mathcal B^{-\alpha' n}_{2p,2p}}}_{L^{2p}_\Omega}< \infty.
\end{align*}
In particular, this implies for each $p \ge 1$ and $T>0$,
\begin{align*}
  \mathbb E_{Y_0 \sim \nu}\left[ \E \left[\pr{\sup_{0< \varepsilon< \varepsilon' \le 1}\abs{\varepsilon' - \varepsilon}^{-\gamma'}\norm{Y_\varepsilon^{:n:} - Y_{\varepsilon'}^{:n:}}_{\mathcal C^\gamma_T \Binf{-n \alpha}} }^{2p}|Y_0\right]\right],
\end{align*}
which gives the desired result.
\end{proof}

\noindent We list some consequences of Theorem \ref{thm:Wick Power Convergence}. Recall the binomial identity for Hermite polynomials: for $n\in \N$,
\begin{align}\label{eq:binomial hermit}
H_n(x+y, \sigma^2) = \sum_{k = 0}^n\binom{n}{k}x^k H_{n - k}(y,\sigma^2),\qquad x,y,\sigma \in \R.
\end{align}
The lemma below is similar to that of \cite[Lemma 3.6]{chandra2025non}.
\begin{lemma}\label{lem:wick power stability}
Suppose $1 \le n \le n_0$ and $\alpha>\alpha_0$ satisfies $\alpha(n-1) < \Theta$, where $\Theta\in (0,1]$ is as in \eqref{ineq:holder_kernel}. Let $Y$ be solution to \eqref{eq: EW eqn} with initial condition $Y_0\in \mathfrak{C}$ and  $\gamma \in (\alpha (n-1), \Theta)$. Then for any $v \in \mathcal C([0,1]; \Binf{\gamma})$,
\begin{align*}
  \pr{\pr{v_t+Y_t}^{:n:}}_{t \in [0,1]}= \pr{\sum_{k = 0}^n \binom{n}{k}v^k_t Y^{:n - k:}_t}_{t \in [0,1]} \in \mathcal C\pr{[0,1];\Binf{-\alpha n}}.
\end{align*}
Moreover, almost surely, the map
\begin{align*}
  \mathcal C \pr{[0,1];\Binf{\gamma}}\ni v \mapsto \pr{v+Y}^{:n:}
\end{align*}
is locally Lipschitz and polynomial of degree $n$.
\end{lemma}
\begin{proof}
This follows from Theorem \ref{thm:besov_product} and Proposition \ref{prop:holder=besov}, where $\Binf{\gamma}$ is a Banach algebra and \eqref{eq:binomial hermit}.
\end{proof}

As in \cite[Lemma 3.13]{chandra2025non}, we show $\mathfrak{C}$ is stable under addition of sufficiently smooth functions. We denote for $\Psi_0 \in \Binf{-\alpha}$ that
\begin{align*}
\llbracket \Psi_0 \rrbracket_p:= \max_{k = 1,\dots,n_0}\E_{Y_0 =\Psi_0} \left[\norm{Y^{:k:}}_{\mathcal C_1\Binf{-\alpha k}} ^p\right]^\frac{1}{p}
\end{align*}
\begin{lemma}\label{lem:stability of c}
Take $\alpha>\alpha_0$ and $1\le n \le n_0$ be as in Lemma \ref{lem:wick power stability}. Let $\Psi_0 \in \mathfrak{C}$ and $v_0 \in \Binf{\gamma}$ with $\gamma > \alpha(n-1)$. Then $\Psi_0+v_0 \in \mathfrak{C}$ and
\begin{align*}
  \llbracket \Psi_0+v_0\rrbracket_p \lesssim \pr{1+\llbracket \Psi_0\rrbracket_p}\pr{1+\norm{v_0}_{\Binf{\gamma}}}^{n_0}.
\end{align*}
\end{lemma}
\begin{proof}
 By Lemma \ref{lem:q_regulariz}, the map $t \mapsto e^{-(1+L)t}v_0$ is in $\mathcal C\pr{[0,1];\Binf{\beta}}$ for any $\beta \in (\alpha(n-1), \gamma)$. Hence the desired result follows from  Lemma \ref{lem:wick power stability}.
\end{proof}

\begin{lemma}\label{lem:uniform unit interval moments of wick powers}
Under the settings of Lemma \ref{lem:wick power stability}, let $\Psi_0\in \mathfrak C$ and let $Y$ solve \eqref{eq: EW eqn} with $Y(0,\cdot)=\Psi_0$. Then, for every $k\in\{1,\dots,n\}$ and every $q\in[1,\infty)$,
\[
  \sup_{t\ge 0}\E_{Y_0=\Psi_0}\left[\norm{Y^{:k:}}_{\mathcal C([t,t+1];\Binf{-\alpha k})}^q\right]<\infty.
\]
\end{lemma}

\begin{proof}
Fix $k\in\{1,\dots,n\}$ and $q\in[1,\infty)$. Let $\widetilde Y_0\sim \nu$ be independent of $\xi$, and let $\widetilde Y$ be the solution to \eqref{eq: EW eqn} driven by the same noise $\xi$ with initial condition $\widetilde Y(0,\cdot)=\widetilde Y_0$. For each $t\ge 0$ and $s\in[0,1]$, set
\begin{align*}
  v^{(t)}_s:=e^{-(1+L)(t+s)}(\Psi_0-\widetilde Y_0),
\end{align*}
so that $Y(t+s)=\widetilde Y(t+s)+v^{(t)}_s$.

Choose $\gamma\in(\alpha(n-1),\Theta)$. By Lemma \ref{lem:q_regulariz}, we have for $t\ge 1$,
\begin{align*}
  \norm{v^{(t)}}_{\mathcal C_1\Binf{\gamma}}\lesssim \norm{\Psi_0-\widetilde Y_0}_{\Binf{-\alpha}}.
\end{align*}
Hence for every $r\ge 1$,
\begin{align*}
  \sup_{t\ge 1}\E\left[\norm{v^{(t)}}_{\mathcal C_1\Binf{\gamma}}^{rq}\right]<\infty,
\end{align*}
where $\E$ is taken over $(\xi,\widetilde Y_0)$.

By Theorem \ref{thm:Wick Power Convergence}, $\nu(\mathfrak C)=1$ and for each $m\in\{1,\dots,n\}$ and $r\ge 1$,
\begin{align*}
  \E\left[\norm{\widetilde Y^{:m:}}_{\mathcal C_1\Binf{-\alpha m}}^{rq}\right]<\infty.
\end{align*}
Since $\widetilde Y$ is stationary, this implies
\begin{align*}
  \sup_{t\ge 0}\E\left[\norm{\widetilde Y^{:m:}}_{\mathcal C([t,t+1];\Binf{-\alpha m})}^{rq}\right]<\infty,
  \qquad 1\le m\le n.
\end{align*}

Now apply \eqref{eq:binomial hermit} and Theorem \ref{thm:besov_product} on the interval $[t,t+1]$
\begin{align*}
  \norm{Y^{:k:}}_{\mathcal C([t,t+1];\Binf{-\alpha k})}
  \lesssim \sum_{j=0}^k \norm{v^{(t)}}_{\mathcal C_1\Binf{\gamma}}^j\norm{\widetilde Y^{:k-j:}}_{\mathcal C([t,t+1];\Binf{-\alpha(k-j)})}.
\end{align*}
Therefore, by H\"older's inequality and the previous bounds,
\begin{align*}
  \sup_{t\ge 1}\E\left[\norm{Y^{:k:}}_{\mathcal C([t,t+1];\Binf{-\alpha k})}^{q}\right]<\infty.
\end{align*}

For $t\in[0,1]$, we have
\begin{align*}
  \norm{Y^{:k:}}_{\mathcal C([t,t+1];\Binf{-\alpha k})}
  \le \norm{Y^{:k:}}_{\mathcal C_1\Binf{-\alpha k}} + \norm{Y^{:k:}}_{\mathcal C([1,2];\Binf{-\alpha k})}.
\end{align*}
The first term has finite $q$-th moment by Definition \ref{def:c set} since $\Psi_0\in\mathfrak C$, and the second term is bounded by the estimate for $t\ge 1$ (take $t=1$). This proves the lemma.
\end{proof}

\section{Solution theory for the remainder equation}\label{sec:Solution Theory}
We now move on to solve the remainder equation \eqref{eq:formal eq for v intro}. We will show that for appropriate initial conditions, the local (in time) solution to \eqref{eq:formal eq for v intro} exists under \eqref{ineq: dbd condition}. If in addition \eqref{ineq:global solution condition} holds, then the solution to the remainder equation exists for all $t>0$.  
For appropriate $v_0$, the solution of \eqref{eq:formal eq for v intro} is understood as the solution to the integral equation
\begin{align}\label{eq:remainder mild form}
    v_t = e^{-t}P_tv_0 - \sum_{k = 0}^n\binom{n}{k}\int_0^t e^{-(t-s)}P_{t-s}\pr{Y^{:k:}(s)v(s)^{n-k}}\mathrm ds.
\end{align}
For obtaining global solution, we will also need to consider the following mollified remainder equation (see Remark \ref{rmk:need for mollified}):
\begin{align}\label{eq: DPD decomposition mollified}
\begin{cases}
  \partial_t v_\varepsilon &= -(L+1) v_\varepsilon - \sum_{k = 0}^n \binom{n}{k} Y^{:k:}_\varepsilon v^{n-k}_\varepsilon\\
  v_\varepsilon(0,\cdot) &=P_\varepsilon\pr{v_0},
\end{cases}
\end{align}
which is understood in the mild form
\begin{align*}
    v_\varepsilon(t)= e^{-t} P_{\varepsilon+t} v_0 - \sum_{k = 0}^n \binom{n}{k}\int_0^t e^{-(t-s)} P_{t-s} \pr{Y_\varepsilon^{:k:}(s)v_\varepsilon(s)^{n-k}}\mathrm ds.
\end{align*}

 \subsection{Local well-posedness of the remainder equation}\label{sec:local sol}

We consider first a slightly more general setting than \eqref{eq:formal eq for v intro}. Recall $\alpha_0= (d_h - d_w)/2$ as in \eqref{eq:alpha_0} and $n_0$ as in \eqref{eq:n0}. Condition \eqref{ineq: dbd condition} is equivalent to the following:  $1 \le n \le n_0$, and there exists $\alpha>\alpha_0$ so that %and let us fix $\alpha >\alpha_0$ and $1 \le n \le n_0$ for now so that
\begin{align}\label{ineq:regularity condition general alpha}
\alpha(n-1) <\pr{d_w - \alpha n}\wedge \Theta,
\end{align}
where $\Theta\in (0,1]$ is as in \eqref{ineq:holder_kernel}. In this subsection, we fix such $\alpha>\alpha_0$.  For $T\ge 1$, consider the Banach space of time-dependent distributions
\begin{align*}
\mathcal G_T := \bigoplus_{k = 1}^n \mathcal C\pr{[0,T]; \Besovp{-\alpha k}{\infty}}
\end{align*}
equipped with the norm
\begin{align*}
\norm{\mathbf Y}_{\mathcal G_T}:= \max_{k = 1,\dots, n}\sup_{t \in [0,T]} \norm{Y_{t}^{(k)}}_{\Besovp{-\alpha k}{\infty}},\qquad \mathbf Y = \pr{Y^{(1)},\dots, Y^{(n)}} \in \mathcal G_T.
\end{align*}
and set $\mathcal G:= \mathcal G_1$.
Consider the equation
\begin{align}\label{eq:DPD_general}
\partial_t v = -(L+1) v + \sum_{k = 0}^n a_k v^{n-k} Y^{(k)},
\end{align}
where $a_{k} \in \R$ for $k = 0,\dots n$. We will also let
\begin{align}\label{ineq:regularity of solution restriction general alpha}
\gamma \in (\alpha (n-1), d_w - \alpha n)\cap  (0,\Theta).
\end{align}
The following proposition is similar to \cite[Proposition 3.2]{chandra2025non}.
\begin{prop}\label{prop:local_solution_with_weights}
Let $\gamma>0$ be as in \eqref{ineq:regularity of solution restriction general alpha}, $\eta \in (- \frac{d_w}{n}, \gamma]\backslash\{0\}$, and $K >2$. Denote
\begin{align*}
  B_K:= \left\{\pr{v_0, \mathbf Y} \in \Besovp{\eta}{\infty}\times \mathcal G: \norm{v_0}_{\Besovp{\eta}{\infty}}+\norm{\mathbf{Y}}_{\mathcal G} \le K  \right\}.
\end{align*}
Then there exists $c> 0$, $\kappa >0$, depending only on $\alpha, \eta, n,\gamma$, so that, for $T^*= cK^{-1/\kappa} \in (0,1)$, \eqref{eq:DPD_general} with initial value $v_0$ admits a unique solution in the Banach space $\mathcal S$ of functions $v \in \mathcal C \pr{[0,T^*], \Besovp{\eta}{\infty}}$, for which
\begin{align}\label{eq:weighted time norm}
  \norm{v}_{\mathcal S_{T^*}}:= \sup_{t \in (0,T^*]}\left\{\norm{v_t}_{\Besovp{\eta}{\infty} } + t^{\frac{\gamma - \eta}{d_w}}\norm{v_t}_{\Besovp{\gamma}{\infty}} +t^{-\frac{\eta\wedge0}{d_w}}\norm{v_t}_{L^\infty_M}  \right\}<\infty.
\end{align}
Moreover, the solution map
\begin{align*}
  B_K \ni \pr{v_0, \mathbf Y} \mapsto v \in \mathcal S
\end{align*}
is Lipschitz and $\norm{v}_\mathcal S \le C\pr{\norm{v_0}_{\Besovp{\eta}{\infty}} +1}$, where $C>0$ depends only on $\alpha, \eta, n,\gamma$.
\end{prop}

\begin{proof}
In this proof, we will denote $\mathcal S_T$ for each $T\in(0,1]$ the Banach space of functions $v \in \mathcal C\pr{[0,T]; \Binf{\eta}}$ with the norm
\begin{align*}
  \norm{v}_{\mathcal S_T}:= \sup_{t \in (0,T]}\left\{\norm{v_t}_{\Besovp{\eta}{\infty} } + t^{\frac{\gamma - \eta}{d_w}}\norm{v_t}_{\Besovp{\gamma}{\infty}} +t^{-\frac{\eta\wedge0}{d_w}}\norm{v_t}_{L^\infty_M}  \right\}.
\end{align*}
Consider the solution map of \eqref{eq:DPD_general} acting on $v \in \mathcal S_T$:
\begin{align*}
  \mathcal S_T \ni v \mapsto \mathbf \Phi(v):= \pr{ e^{-t}P_tv_0 +\int_0^t e^{-(t-s)}P_{t-s} Y^{(n)}_s \mathrm ds + \sum_{k = 0}^{n-1}\int_0^t e^{-(t-s)}P_{t-s}\pr{v^{n-k}_sY^{(k)}_s} \mathrm ds}_{t \ge 0}.
\end{align*}
We will first show $\mathbf\Phi: \mathcal S \to \mathcal S$. Consider first
\begin{align}\label{ineq:solution_map_general_bound}
  \norm{\mathbf \Phi(v)_t}_{\Binf{\gamma}}\lesssim \norm{P_tv_0}_{\Binf{\gamma}} + \int_0^t \norm{P_{t-s} \Psi^{(n)}_s}_{\Binf{\gamma}}\mathrm ds + \sum_{k = 0}^{n-1}\int_0^t \norm{P_{t-s}\pr{v_s^{n-k}Y^{(k)}_s}}_{\Binf{\gamma}} \mathrm d s.
\end{align}
Consider first by Lemma \ref{lem:q_regulariz} that uniformly in $t \in (0,1]$,
\begin{align*}
  \int_0^t \norm{P_{t-s} Y^{(n)}_s}_{\Binf{\gamma}}\mathrm ds \lesssim \int_0^t\pr{t-s}^{-\frac{\gamma+n \alpha}{d_w}}\norm{Y^\prk{n}_s}_{\Binf{-\alpha n}}\mathrm ds \overset{\eqref{ineq:regularity of solution restriction general alpha}}{\lesssim} K.
\end{align*}
and
\begin{align*}
  \norm{P_tv_0}_{\Binf{\gamma}} \lesssim t^{\frac{\eta - \gamma}{d_w}}\norm{v_0}_{\Binf{\eta}}.
\end{align*}
Let $\varepsilon>0$ be sufficiently small by Proposition \ref{prop:holder=besov} and  Lemma \ref{lem: Besov Interpolation}, we see from  \eqref{ineq:regularity of solution restriction general alpha} that for each $k = 0,\dots , n-1$,
\begin{align*}
  \norm{v_s^{n-k}}_{\Binf{\alpha k+\varepsilon}}&\lesssim \norm{v_s}_{\Binf{\alpha k+\varepsilon}} \norm{v_s}_{L^\infty_M}^{n-k-1}\le \norm{v_s}_{\Binf{\gamma}}^\frac{\alpha k+\varepsilon}{\gamma} \norm{v_s}_{L^\infty_M}^{n-k-\frac{\alpha k +\varepsilon}{\gamma}}.
\end{align*}
Hence, by Lemma \ref{lem:q_regulariz} and Theorem \ref{thm:besov_product}
\begin{align*}
  \int_0^t P_{t-s}\norm{v^{n-k}_s Y^\prk{k}_s}_{\Binf{\gamma}} \mathrm ds &\lesssim \int_0^t \pr{t-s}^{- \frac{\gamma+ \alpha k}{d_w}}  \norm{v_s^{n-k}}_{\Binf{\alpha k+\varepsilon}}\norm{Y^\prk{k} }_{\Binf{- \alpha k }}\mathrm ds\\
  &\le K\int_0^t \pr{t-s}^{- \frac{\gamma+ \alpha k}{d_w}}  \norm{v_s}_{\Binf{\gamma}}^\frac{\alpha k+\varepsilon}{\gamma} \norm{v_s}_{L^\infty_M}^{n-k-\frac{\alpha k +\varepsilon}{\gamma}}\mathrm ds\\
  &\le K\norm{v}_{\mathcal S_T}^{n-k} \int_0^t\pr{t-s}^{-\frac{\gamma+\alpha k }{d_w}}s^{-\Xi_k}\mathrm ds,
\end{align*}
where
\begin{align*}
  \Xi_k = {\frac{\alpha k+\varepsilon}{\gamma}\frac{ \gamma- \eta }{d_w} - \frac{\eta\wedge 0}{d_w}\pr{n-k - \frac{\alpha k +\varepsilon}{\gamma}}}.
\end{align*}
When $\eta \ge 0$, we see for all $k = 0,\dots, n-1$ and sufficiently small $\varepsilon>0$,  $\Xi_k = \frac{\gamma - \eta}{\gamma } \frac{\alpha k +\varepsilon}{d_w}<1$  by \eqref{ineq:regularity condition general alpha} and $\Xi_k = 0$ when $\eta = \gamma>0$. For $-\frac{d_w}{n} < \eta < 0$,
\begin{align*}
  \Xi_k = \frac{\alpha k +\varepsilon}{d_w}  -  \frac{\eta}{d_w}(n-k),
\end{align*}
which is linear in $k = 0,\dots, n-1$. Hence $\max_{k = 0,\dots,n-1} \Xi_k = \max\{\Xi_0,\Xi_{n-1}\}$.
Note $\Xi_0 = \frac{\varepsilon}{d_w}  - \frac{n\eta}{d_w}<1$ for sufficiently small $\varepsilon >0$, and
\begin{align*}
  \Xi_{n-1} = \frac{\alpha (n-1) +\varepsilon}{d_w} - \frac{\eta}{d_w} \le \frac{n-1}{2n-1} + \frac{\varepsilon}{d_w} < 1,
\end{align*}
for all $n \ge 1$ and sufficiently small $\varepsilon >0$. Hence uniformly in $t \in (0,1]$,
\begin{align*}
  t^{\frac{\gamma - \eta}{d_w}}\int_0^t \norm{P_{t-s}\pr{v^{n-k}_s Y^\prk{k}_s}}_{\Binf{\gamma}} \mathrm ds \lesssim \norm{v}_{\mathcal S_T}^{n-k}  K t^{1+ \frac{\gamma - \eta}{d_w}-\pr{\frac{\alpha k +\varepsilon+\gamma}{d_w}+\Xi_k}}.
\end{align*}
Consider first when $0 \le \eta \le \gamma$, we may pick a sufficiently small $\varepsilon>0$ so that for all $k = 0,\dots, n-1$
\begin{align*}
  1+ \frac{\gamma - \eta}{d_w}-\frac{\alpha k +\varepsilon+\gamma}{d_w}-\Xi_k&=1+ \frac{\gamma - \eta}{d_w}-\frac{\alpha k +\varepsilon+\gamma}{d_w}-\frac{\alpha k+\varepsilon}{\gamma} \frac{\gamma - \eta}{d_w}\\
  &\overset{\eqref{ineq:regularity of solution restriction general alpha}}{> } 1- \frac{\alpha k +\varepsilon +\gamma}{d_w}\overset{\eqref{ineq:regularity of solution restriction general alpha} }{>}0.
\end{align*}
For $- \frac{d_w}{n}< \eta <0$, the exponent is then
\begin{align*}
  \Xi'_k:= &1+ \frac{\gamma - \eta}{d_w} - \frac{\alpha k +\varepsilon +\gamma}{d_w} - \frac{\alpha k +\varepsilon}{\gamma}\frac{\gamma - \eta}{d_w} + \frac{\eta}{d_w} \pr{n-k- \frac{\alpha k +\varepsilon}{\gamma}}\\
  =& 1 - \frac{\eta}{d_w}-\frac{\alpha k+\varepsilon}{d_w}- \frac{\alpha k +\varepsilon}{d_w} +  \frac{\alpha k+\varepsilon}{\gamma}\frac{\eta}{d_w} + \frac{\eta}{d_w}(n-k) - \frac{\eta}{d_w} \frac{\alpha k +\varepsilon}{\gamma}\\
  =&1 - \frac{\eta}{d_w} - 2\frac{\alpha k +\varepsilon}{d_w} + \frac{\eta}{d_w}(n-k),
\end{align*}
which is linear in $k = 0,\dots, n-1$. Since $- \frac{d_w}{n}< \eta$, we see for sufficiently small $\varepsilon >0$,
\begin{align*}
  \Xi'_0 = 1- \frac{\eta}{d_w}(n-1) - \frac{2\varepsilon}{d_w}>0.
\end{align*}
On the other hand, since $\alpha < \frac{d_w}{2n-1}$ from \eqref{ineq:regularity of solution restriction general alpha}, we see for all sufficiently small $\varepsilon >0$,
\begin{align*}
  \Xi'_{n-1} = 1 - \frac{2\alpha (n-1)}{d_w} - \frac{2\varepsilon}{d_w}>1-2(n-1) \frac{1}{2n-1} - \frac{2\varepsilon}{d_w} > 0.
\end{align*}
Denote
$$\kappa':=  \max\left\{\Xi'_{k}, 1-\frac{\alpha n}{d_w}:0 \le k \le n-1\right\},$$
then by collecting all terms, we see that
\begin{align}
  \sup_{t \in (0,T]}  t^{\frac{\gamma - \eta}{d_w}} \norm{v_t}_{\Binf{\gamma}} \lesssim KT^{\kappa'}\pr{\sum_{k = 1}^{n}\norm{v}_{\mathcal S_T}^k}.
\end{align}
Now observe that if $\eta >0$, then $\norm{\mathbf \Phi(v)_t}_{L^\infty_M} \lesssim \norm{\mathbf \Phi(v)_t}_{\Binf{\eta}}$ uniformly in $t\ge0$, hence we will consider the case for $- \frac{d_w}{n}< \eta\le 0$.
Let us now consider
\begin{align*}
  \norm{\mathbf \Phi\pr{v}_t}_{L^\infty_M} &\lesssim \int_0^t \norm{P_{t-s}Y^\prk{n}_s}_{L^\infty_M} \mathrm ds +\sum_{k = 0}^{n-1} \int_0^t \norm{P_{t-s} \pr{v_s^{n-k} Y^\prk{k}_s }}_{L^\infty_M} \mathrm ds.
\end{align*}
By Lemma \ref{lem:q_regulariz},
\begin{align*}
  \int_0^t \norm{P_{t-s}Y^\prk{n}_s}_{L^\infty_M} \mathrm ds \lesssim \int_0^t \pr{t-s}^{- \frac{\alpha n}{d_w}}\mathrm ds K \lesssim K t^{\kappa'}.
\end{align*}
By a similar argument, we see for $k = 0,\dots, n-1$, it holds uniformly in $t \in (0,1]$ that
\begin{align*}
  t^{- \frac{\eta}{d_w}}\int_0^t\norm{P_{t-s} \pr{v_s^{n-k}Y^\prk{k}_s}}_{L^\infty_M}\mathrm ds \lesssim K \norm{v}_{\mathcal S_T}^{n-k}t^{- \frac{\eta}{d_w}}\int_0^t \pr{t-s}^{- \frac{\alpha k }{d_w}}s^{-\Xi_k} \mathrm ds \lesssim K t^{\kappa'}\norm{v}_{\mathcal S_T}^{n-k}.
\end{align*}
Finally, since $\eta \neq 0$, $t^{-\frac{\eta}{d_w}}\norm{P_tv_0}_{L^\infty_M} \lesssim \norm{v_0}_{\Binf{\eta}} \le K.$ Collecting all the terms to see for so me $C_0\ge1$,
\begin{align}\label{ineq:solution operator norm bound local solution}
  \norm{\mathbf \Phi(v)}_{\mathcal S_T}
  \le C_0 \norm{v_0}_{\Binf{\eta}} + T^{\kappa'} C_0 K \pr{\sum_{k = 1}^n\norm{v}_{\mathcal S_T}^{k}}
  \le C_0 K + T^{\kappa'} C_0 K \pr{\sum_{k = 1}^n\norm{v}_{\mathcal S_T}^{k}}.
\end{align}
We may fix $\ell_0 \ge 2$ and let $R_0:= \ell_0C_0K \ge 2$, $\kappa :=\kappa' /n$, and set
$$
T^*:=\pr{\pr{\frac{\ell_0-1}{n}}^\frac{1}{n} \ell_0 C_0 K}^{- \frac{1}{\kappa}}  \le \pr{\frac{R_0/(C_0K)-1}{\sum_{k = 1}^n R_0^k}}^{\frac{1}{\kappa}}.
$$
Then \eqref{ineq:solution operator norm bound local solution}, along with Proposition \ref{prop: schauder} implies  $\mathbf \Phi$ maps the closed ball of radius $R_0$ centered at $0$ in $\mathcal S_{T^*}$ to itself. By a similar argument, one can show $\mathbf \Phi$ is a contraction on this ball for possibly smaller $T^*$.

Finally, to see the solution map is Lipschitz, we let $(v_0,\mathbf Y), (w_0, \widetilde{ \mathbf Y}) \in B_K$, and let $v,w \in \mathcal S_{T^*}$ be solution to \eqref{eq:DPD_general} with initial condition $v_0,w_0$ and coefficients $\mathbf Y,\widetilde{\mathbf Y}$ respectively. Then by construction, $\norm{v}_{\mathcal S_{T^*}}, \norm{w}_{\mathcal S_{T^*}} \le R_0$. Now consider
\begin{align*}
  \norm{v_t - w_t}_{\Binf{\eta}} &\lesssim \norm{P_t\pr{v_0 - w_0}}_{\Binf{\eta}  } + \int_0^t \norm{P_{t-s} \pr{v_s^n -w^n_s}}_{\Binf{\eta}}\mathrm ds  + \sum_{k = 1}^{n} \int_0^t \norm{P_{t-s}\widetilde{Y}_s^{(k)}\pr{v_s^{n-k} - w_s^{n-k}}}_{\Binf{\eta}}\mathrm ds\\
  &\qquad \qquad +\sum_{k = 1}^n \int_0^t \norm{P_{t-s}w_s \pr{\widetilde Y^{(k)}_s - Y^{(k)}_s}}_{\Binf{\eta}}\mathrm ds,
\end{align*}
and use similar decomposition for $\norm{v_t - w_t}_{L^\infty_M}$. A similar argument shows that the solution map is Lipschitz on the ball of radius $R_0$ in $\mathcal S_{T^*}$.
\end{proof}

\noindent As a consequence of Proposition \ref{prop:local_solution_with_weights}, for any $v_0 \in \Binf{-\alpha}$, there exits a unique local solution to the equations  \eqref{eq:formal eq for v intro} and \eqref{eq: DPD decomposition mollified}. Recall $\mathfrak{C}$ as in Definition \ref{def:c set}.

\begin{lemma}\label{lem:local solution to DPD and convergence of approximation}
Let $\alpha>\alpha_0$ and $2\le n \le n_0$ satisfies \eqref{ineq:regularity condition general alpha}, take $\gamma >0$ be as in \eqref{ineq:regularity of solution restriction general alpha} and $\eta$ be as in Proposition \ref{prop:local_solution_with_weights}. Suppose $Y$ is solution to \eqref{eq: EW eqn} with initial condition $Y_0 \in \mathfrak{C}$. Then for each $v_0 \in \Binf{-\eta}$,
\begin{itemize}
  \item There exists a (random) $T^*>0$ depending only on $\max_{k = 1,\dots,n}\sup_{t \in [0,1]}\norm{Y^{:k:}_t}_{\Binf{-\alpha k }}$ and $\norm{v_0}_{\Binf{\eta}}$, so that equation \eqref{eq:formal eq for v intro} admits a unique solution, $v$, in $\mathcal S_{T^*}$ as in \eqref{eq:weighted time norm}.
  \item For each $\varepsilon \in (0,1]$, there exist (random) $T^*_\varepsilon\ge T^*>0$, so that equation \eqref{eq: DPD decomposition mollified} admits a unique solution, $v_\varepsilon$, in $\mathcal S_{T^*_\varepsilon}$.
\end{itemize}
Furthermore,
\begin{align*}
  \lim_{\varepsilon\downarrow 0} \norm{v- v_\varepsilon}_{\mathcal S_{T^*\wedge T^*_\varepsilon}} = 0.
\end{align*}
\end{lemma}
\begin{remark}\label{rmk:local solution alpha initial condition}
Note that \eqref{ineq:regularity condition general alpha} implies $\alpha < \frac{d_w}{2n-1}\le \frac{d_w}{n}$. Hence $\eta = -\alpha$ satisfies condition of Proposition \ref{prop:local_solution_with_weights} and hence Lemma \ref{lem:local solution to DPD and convergence of approximation}.
\end{remark}
\begin{proof}
The existence and uniqueness of solutions to \eqref{eq:formal eq for v intro} and \eqref{eq: DPD decomposition mollified} is a direct consequence of Proposition \ref{prop:local_solution_with_weights} and Theorem \ref{thm:Wick Power Convergence}. The fact that $T^*_\varepsilon \ge T^*$ for sufficiently small $\varepsilon >0$ is due to the mollified initial condition and Wick powers of $Y$ has smaller Besov norms. Lastly, the convergence is given by the local Lipschitzness of solution map as in Proposition \ref{prop:local_solution_with_weights}.
\end{proof}

\subsection{Global Solution and Coming Down from Infinity}\label{sec:global sol}

One main difference between this section and the corresponding sections in \cite{MW17,TW18} is that on $\T^2$, the solution the the remainder equation \eqref{eq:remainder mild form} $v_t$ is sufficiently regular and belongs to the domain of the Dirichlet form, $\mathcal F$, for $t >0$. This allows them to control the $L^p_M$ norm of $v$ in terms of its Dirichlet energy on the torus. In our setting, $v$ need not be inside the domain $\mathcal F$.  However, we will show that running their arguments for the mollified solution $v_\varepsilon$ is sufficient to produce the desired a priori estimate.

In this subsection, we fix an arbitrary $Y_0\in\mathfrak{C}$ and let $Y$ be the solution to \eqref{eq: EW eqn} with initial condition $Y_0$. 
We assume also for the rest of the paper that \eqref{ineq:global solution condition} holds for some odd $n \ge 3$, which is equivalent to the following condition: for some $\alpha >\alpha_0$,
\begin{align}\label{ineq:global solution condition equivalent}
    n \alpha < \frac{d_w}{2}.
\end{align}
Hence, we shall fix an $\alpha>\alpha_0$ that satisfies both \eqref{ineq:global solution condition equivalent} and \eqref{ineq:regularity condition general alpha}, which is needed in Proposition \ref{prop:CDI_mollified}.

\begin{lemma}\label{lem:regularity of mollified solution}
Let $\varepsilon\in(0,1]$  and $v_0 \in \mathcal B^{-\alpha}_\infty$, suppose $v_\varepsilon$ is the local solution to \eqref{eq: DPD decomposition mollified}  given by Lemma \ref{lem:local solution to DPD and convergence of approximation} on $[0,T^*_\varepsilon]$. Then  for any small enough $\delta_1,\delta_2 ,\delta >0$, we have
\begin{align*}
  \Psi(Y_\varepsilon,v_\varepsilon) \in \mathcal C_{T^*_\varepsilon}\mathcal C_M^{\delta}, \qquad \text{ and }\qquad v_\varepsilon \in \mathcal C_{T^*_\varepsilon}^{\frac{1}{2}+\delta_1} \mathcal B^{\frac{d_w}{2}+\delta_2}_{2,\infty} \qquad \text{ almost surely.}
\end{align*}
\end{lemma}
\begin{proof}
Let $\varepsilon \in (0,1]$, by Lemma \ref{lem: holder regularity of mollifed y}, we can choose a sufficiently small $\delta \in (0,\Theta)$ so that $Y_\varepsilon^{:n:} \in \mathcal C_{T} \mathcal C^\delta_M$ for each $n \ge 1$. Note that by Lemma \ref{lem:q_regulariz}, $v_{0,\varepsilon}:=P_\varepsilon(v_0) \in \mathcal B^\beta_\infty$ for each $\beta>0$. Hence by Lemma \ref{lem:strong continuity of heat semi group}, we may pick an appropriate $\beta > 0$ so that 
\begin{align*}
    t\mapsto P_t\pr{v_{0,\varepsilon}} \in \mathcal C_T^{\frac{1}{2}+ \frac{\delta}{2d_w}}\mathcal B^{\frac{d_w}{2}+\frac{\delta}{2}}_\infty.
\end{align*}

By Proposition \ref{prop:local_solution_with_weights} and the fact $v_{\varepsilon,0} \in \Binf{\delta}$, we see 
and $v_\varepsilon \in \mathcal C_{T^*_\varepsilon}\mathcal B_{\infty}^{\delta} = \mathcal C_{T^*_\varepsilon}\mathcal C_M^\delta$. Since $C_{T^*_\varepsilon}\mathcal C_M^\delta$ forms an algebra, we see $\Psi(Y_\varepsilon,v_\varepsilon) \in \mathcal{C}_{T^*_\varepsilon}\mathcal C_M^\delta = \mathcal C_{T^*_\varepsilon} \mathcal B^\delta_{\infty}$ by Proposition \ref{prop:holder=besov}. Finally, by Jensen's inequality and  Proposition \ref{prop: schauder}, we  see for any $T< T^*_\varepsilon$
\begin{align*}
  \norm{\int_0^\cdot P_{t-s}\Psi_\varepsilon(s)\mathrm ds }_{\mathcal C_{T}^{\frac{1}{2}+\frac{\delta}{2d_w}}\mathcal B_{2,\infty}^{\frac{d_w}{2}+\frac{\delta}{2}}}\lesssim\norm{\int_0^\cdot P_{t-s}\Psi_\varepsilon(s)\mathrm ds }_{\mathcal C_{T}^{\frac{1}{2}+\frac{\delta}{2d_w}}\mathcal B_{\infty}^{\frac{d_w}{2}+\frac{\delta}{2}}} \lesssim \norm{\Psi_\varepsilon(s)}_{\mathcal C_{T}\mathcal B^{\delta}_{\infty}}<\infty.
\end{align*}
\end{proof}

\begin{lemma}[Weak formulation]\label{lem: MW weak sol}
Let $\varepsilon\in(0,1]$ and  $v_0 \in \mathcal B^{-\alpha}_\infty$, suppose $v_\varepsilon$ is the local solution to \eqref{eq: DPD decomposition mollified}  given by Lemma \ref{lem:local solution to DPD and convergence of approximation} on $[0,T^*_\varepsilon]$. For any $0\le t\le T\le T_\varepsilon^*$ and $\varphi\in \mathcal F$,
\[
  \braket{v_\varepsilon(T),\varphi}_{L^2_M}-\braket{v_\varepsilon(t),\varphi}_{L^2_M}
  =\int_t^T\pr{-\mathcal E(v_\varepsilon(s),\varphi)-\braket{v_\varepsilon(s),\varphi}_{L^2_M}+\braket{\Psi_\varepsilon(s),\varphi}_{L^2_M}}\,\mathrm ds,
\]
where $\Psi_\varepsilon(s):=\Psi\pr{Y_\varepsilon(s),v_\varepsilon(s)}$.
\end{lemma}
\begin{proof}
Set $S_t:=e^{-(1+L)t}=e^{-t}P_t$ and $v_{0,\varepsilon}:=P_\varepsilon v_0$. The mild formulation \eqref{eq:remainder mild form} reads
\begin{align*}
  v_\varepsilon(s)=S_s v_{0,\varepsilon}+\int_0^s S_{s-r}\Psi_\varepsilon(r)\,\mathrm dr,\qquad s\in[0,T].
\end{align*}

First take $\varphi\in \mathcal D_2(L)$. By the mild formula \eqref{eq:remainder mild form} at times $T$ and $t$,
\begin{align}
  \label{eq:weak_step_difference}
  \braket{v_\varepsilon(T)-v_\varepsilon(t),\varphi}_{L^2_M}
  &=\braket{(S_T-S_t)v_{0,\varepsilon},\varphi}_{L^2_M}
  +\int_0^t \braket{(S_{T-r}-S_{t-r})\Psi_\varepsilon(r),\varphi}_{L^2_M}\,\mathrm dr\nonumber\\
  &\qquad +\int_t^T \braket{S_{T-r}\Psi_\varepsilon(r),\varphi}_{L^2_M}\,\mathrm dr.
\end{align}
On the other hand, using $\frac{\mathrm d}{\mathrm du}\braket{S_u f,\varphi}_{L^2_M}=\braket{S_u f,-(L+1)\varphi}_{L^2_M}$ for $\varphi\in\mathcal D_2(L)$ and Fubini's theorem,
\begin{align}
  \label{eq:weak_step_generator}
  \int_t^T \braket{v_\varepsilon(s),-(L+1)\varphi}_{L^2_M}\,\mathrm ds
  &=\braket{(S_T-S_t)v_{0,\varepsilon},\varphi}_{L^2_M}
  +\int_0^t \braket{(S_{T-r}-S_{t-r})\Psi_\varepsilon(r),\varphi}_{L^2_M}\,\mathrm dr\nonumber\\
  &\qquad +\int_t^T \braket{(S_{T-r}-\mathrm{Id})\Psi_\varepsilon(r),\varphi}_{L^2_M}\,\mathrm dr.
\end{align}
Subtracting \eqref{eq:weak_step_generator} from \eqref{eq:weak_step_difference}, we get
\begin{align*}
  \braket{v_\varepsilon(T)-v_\varepsilon(t),\varphi}_{L^2_M}
  =\int_t^T\braket{v_\varepsilon(s),-(L+1)\varphi}_{L^2_M}\,\mathrm ds
  +\int_t^T\braket{\Psi_\varepsilon(s),\varphi}_{L^2_M}\,\mathrm ds.
\end{align*}
By Lemma \ref{lem:regularity of mollified solution} and  \eqref{eq:dform and generator}, this is
\begin{align*}
  \braket{v_\varepsilon(T),\varphi}_{L^2_M}-\braket{v_\varepsilon(t),\varphi}_{L^2_M}
  =\int_t^T\pr{-\mathcal E(v_\varepsilon(s),\varphi)-\braket{v_\varepsilon(s),\varphi}_{L^2_M}+\braket{\Psi_\varepsilon(s),\varphi}_{L^2_M}}\,\mathrm ds.
\end{align*}

Now let $\varphi\in\mathcal F$. Define $\varphi_m:=P_{1/m}\varphi\in\mathcal D_2(L)\subset\mathcal F$. By standard semigroup properties of regular symmetric Dirichlet forms, $\varphi_m\to\varphi$ in $\mathcal F$ and in $L^2_M$. Apply the identity above to $\varphi_m$.
Passing $m\to\infty$ is justified term by term. First,
\begin{align*}
  \braket{v_\varepsilon(T)-v_\varepsilon(t),\varphi_m-\varphi}_{L^2_M}\to 0,
\end{align*}
since $\varphi_m\to\varphi$ in $L^2_M$. Next, by Cauchy--Schwarz for the Dirichlet form,
\begin{align*}
  \abs{\int_t^T \mathcal E(v_\varepsilon(s),\varphi_m-\varphi)\,\mathrm ds}
  \le (T-t)\sup_{s\in[t,T]}\mathcal E(v_\varepsilon(s),v_\varepsilon(s))^{1/2}\,\mathcal E(\varphi_m-\varphi,\varphi_m-\varphi)^{1/2}\to 0,
\end{align*}
because $s\mapsto v_\varepsilon(s)$ is continuous in $\mathcal F$ on $[t,T]$. Finally,
\begin{align*}
  \abs{\int_t^T \braket{\Psi_\varepsilon(s),\varphi_m-\varphi}_{L^2_M}\,\mathrm ds}
  \le \norm{\varphi_m-\varphi}_{L^2_M}\int_t^T \norm{\Psi_\varepsilon(s)}_{L^2_M}\,\mathrm ds\to 0,
\end{align*}
and the integral is finite since, by Lemma \ref{lem:regularity of mollified solution}, $\Psi_\varepsilon\in \mathcal C([t,T];\mathcal C_M^{\delta})\subset L^1([t,T];L^2_M)$. Therefore the same identity holds for every $\varphi\in\mathcal F$.
\end{proof}

\begin{lemma}[Testing against $v^{p-1}_\varepsilon$]\label{lem: test Z p-1}
Let $\varepsilon\in(0,1]$ and $v_0 \in \mathcal B^{-\alpha}_\infty$,suppose $v_\varepsilon$ is the local solution to \eqref{eq: DPD decomposition mollified}  given by Lemma \ref{lem:local solution to DPD and convergence of approximation} on $[0,T^*_\varepsilon]$. Then for any integer $p\ge 2$ and  $0\le t\le T\le T_\varepsilon^*$,
\begin{align*}
  \frac{1}{p}\pr{\norm{v_{\varepsilon}(T)}_{L^p_M}^p-\norm{v_\varepsilon(t)}_{L^p_M}^p}
  =\int_t^T \pr{-\mathcal{E}(v_\varepsilon(s),v_\varepsilon^{p-1}(s)) -\norm{v_\varepsilon(s)}_{L^p_M}^p +\braket{\Psi_\varepsilon(s),v^{p-1}_{\varepsilon}(s)}_{L^2_M}}\mathrm ds.
\end{align*}
Consequently, for almost every $t\in(0,T_\varepsilon^*)$,
\begin{align*}
  \frac{1}{p}\partial_t \norm{v_{\varepsilon}(t)}^p_{L^p_M}
  =-\mathcal{E}(v_{\varepsilon}(t),v_{\varepsilon}^{p-1}(t))-\norm{v_\varepsilon(t)}_{L^p_M}^p+\braket{\Psi_{\varepsilon}(t),v^{p-1}_{\varepsilon}(t)}_{L^2_M}.
\end{align*}
\end{lemma}
\begin{remark}\label{rmk:need for mollified}
    Lemma \ref{lem: test Z p-1} is why we need to look at the mollified equation \eqref{eq: DPD decomposition mollified}. Indeed, Proposition \ref{prop:local_solution_with_weights} only gives $v_t \in \mathcal B^{\eta}_\infty$ for $\eta \in (0,\Theta)$ and $t \in [0,T^*]$. But for $\mathcal E(v_t,v_t)$ to make sense, one needs $v_t \in \mathcal B^{d_w/2+\delta}_{2,\infty}$ for some $\delta>0$, which need not be true since $0<\Theta\le 1 <d_w/2$.
\end{remark}

\begin{proof}
Fix $0\le s\le s'\le T_\varepsilon^*$. We start from
\[
  \norm{v_{\varepsilon}(s')}_{L^p_M}^p-\norm{v_\varepsilon(s)}_{L^p_M}^p
  =\braket{v_{\varepsilon}(s')-v_\varepsilon(s),v^{p-1}_{\varepsilon}(s)}_{L^2_M}
  +\braket{v_{\varepsilon}(s'),v_{\varepsilon}(s')^{p-1}-v^{p-1}_{\varepsilon}(s)}_{L^2_M}.
\]
Applying Lemma \ref{lem: MW weak sol} with $t=s$, $T=s'$ and test function $\varphi=v_\varepsilon(s)^{p-1}\in\mathcal F$ yields
\begin{align*}
  &\norm{v_{\varepsilon}(s')}_{L^p_M}^p-\norm{v_\varepsilon(s)}_{L^p_M}^p-\braket{v_{\varepsilon}(s'),v_{\varepsilon}(s')^{p-1}-v^{p-1}_{\varepsilon}(s)}_{L^2_M}\\
  =&\int_s^{s'}\pr{-\mathcal{E}(v_\varepsilon(r),v^{p-1}_\varepsilon(s))
    -\braket{v_\varepsilon(r),v^{p-1}_\varepsilon(s)}_{L^2_M}
  +\braket{\Psi_\varepsilon(r),v^{p-1}_\varepsilon(s)}_{L^2_M}}\,\mathrm dr.
\end{align*}
Now fix $0\le t\le T\le T_\varepsilon^*$ and let $\underline t=(t_0,\dots,t_m)$ be a subdivision of $[t,T]$. Summing the above identity on each interval $[t_i,t_{i+1}]$ gives
\[
  \norm{v_{\varepsilon}(T)}_{L^p_M}^p-\norm{v_\varepsilon(t)}_{L^p_M}^p-\mathfrak G(\underline t)=\mathfrak J(\underline t),
\]
where
\begin{align*}
  \mathfrak G(\underline t):=&\sum_{i=0}^{m-1}\braket{v_\varepsilon(t_{i+1}),v^{p-1}_\varepsilon(t_{i+1})-v^{p-1}_\varepsilon(t_i)}_{L^2_M},\\
  \mathfrak J(\underline t):=&\sum_{i=0}^{m-1}\int_{t_i}^{t_{i+1}}\pr{-\mathcal{E}(v_\varepsilon(s),v^{p-1}_\varepsilon(t_i))
    -\braket{v_\varepsilon(s),v^{p-1}_\varepsilon(t_i)}_{L^2_M}
  +\braket{\Psi_\varepsilon(s),v^{p-1}_\varepsilon(t_i)}_{L^2_M}}\,\mathrm ds.
\end{align*}
Using Lemma \ref{lem:regularity of mollified solution} and the same Young-integration argument as in \cite{MW17,TW18}, as $|\underline t| := \max_{1 \le i \le m}|t_i - t_{i - 1}|\to0$,
\[
  \mathfrak J(\underline t)\to\int_t^T\pr{-\mathcal{E}(v_\varepsilon(s),v^{p-1}_\varepsilon(s))
    -\norm{v_\varepsilon(s)}_{L^p_M}^p
  +\braket{\Psi_\varepsilon(s),v^{p-1}_\varepsilon(s)}_{L^2_M}}\,\mathrm ds.
\]
Moreover, exactly as in \cite[Lemma 3.7]{MW17},
\[
  \mathfrak G(\underline t)\to\frac{p-1}{p}\pr{\norm{v_{\varepsilon}(T)}_{L^p_M}^p-\norm{v_\varepsilon(t)}_{L^p_M}^p}.
\]
Passing to the limit yields the first identity. The map $t\mapsto\norm{v_\varepsilon(t)}_{L^p_M}^p$ is therefore absolutely continuous on compact subintervals of $(0,T_\varepsilon^*)$, and differentiating gives the second identity for a.e. $t\in(0,T_\varepsilon^*)$.
\end{proof}

\noindent To obtain an initial condition independent bound, we need the following elementary result.
\begin{lemma}[{\cite[Lemma 3.8]{TW18}}]\label{lem:elementary comparison}
Let $\lambda>1$ and $f:[0,T]\to\R_+$ be differentiable such that \[f'(t)+c_1f(t)^\lambda\le c_2\]
for some $c_1,c_2>0$ and all $t\in [0,T]$. Then for all $t>0$, \[f(t)\leq \min\pr{\frac{f(0)}{(1+tf(0)^{\lambda-1}(\lambda-1)\frac{c_1}{c_2})^{\frac{1}{\lambda-1}}},\pr{\frac{2c_2}{c_1}}^{\frac{1}{\lambda}}}\leq \min\pr{t^{-\frac{1}{\lambda-1}}\pr{(\lambda-1)\frac{2c_2}{c_1}}^{-\frac{1}{\lambda-1}},\pr{\frac{2c_2}{c_1}}^{\frac{1}{\lambda}}}.\]
\end{lemma}

We now prove an analogue of \cite[Proposition 3.7]{TW18} for $v_\varepsilon$.

\begin{prop}[Coming down from infinity]\label{prop:CDI_mollified}
%Let $p \ge n-2$ be a sufficiently large odd integer. 
Let $\varepsilon\in(0,1]$ and $v_0 \in \mathcal B^{-\alpha}_\infty$, let $v_\varepsilon$ be the mollified local solution of \eqref{eq: DPD decomposition mollified}  on $[0,T_\varepsilon^*]$. 
\begin{comment}
For each $k\in\set{1,\dots,n}$, $\varepsilon \in [0,1]$ and $T>0$, 
we define %fix $\beta_k>\frac{k(d_h-d_w)}{2}$ and define
\[
  \Xi_{k,T}^\varepsilon:=\sup_{0\le s\le T}\norm{Y_\varepsilon^{:k:}(s)}_{\mathcal B^{-\alpha k }_\infty},\qquad \varepsilon \in (0,1].
\]
\end{comment}
%Then for every $T>0$ there exist exponents $q_k>0$ and a constant $C_T>0$, independent of $\varepsilon$, such that for almost every $t\in(0,T\wedge T_\varepsilon^*]$,
%\begin{equation}\label{eq:CDI_differential}
%  \partial_t \norm{v_\varepsilon(t)}_{L^{2p}_M}^{2p}+c\norm{v_\varepsilon(t)}_{L^{2p}(M)}^{2p+n-1}\le C_T\pr{1+\sum_{k=1}^n (\Xi_{k,T}^\varepsilon)^{q_k}}.
%\end{equation}
Then for all $T >0$ and sufficiently large odd integer $p \ge 1$, there exists a constant $C = C(p,T)>0$ independent of $\varepsilon\in(0,1]$ and $v_0$
%initial conditions 
such that for all $t\in(0,T\wedge T_\varepsilon^*]$,
\begin{equation}\label{eq:CDI_bound}
  \norm{v_\varepsilon(t)}_{L^{2p}_M}\le C\pr{1+t^{-\frac{1}{n-1}}}\pr{1+\sum_{k=1}^n \pr{\Xi^\varepsilon}_{k,T}^{q_{k,p}}}^\frac{1}{n-1},
\end{equation}
for some $q_{k,p}>0$ depends on $k =1,\dots, n$ and $p$, where 
\begin{align*}
    \Xi_{k,T}^\varepsilon:=\sup_{0\le s\le T}\norm{Y_\varepsilon^{:k:}(s)}_{\mathcal B^{-\alpha k }_\infty},\qquad \varepsilon \in (0,1].
\end{align*}
\begin{comment}
Equivalently, for every $0<\tau\le T$,
\[
  \sup_{\tau\le t\le T}\|v_\varepsilon(t)\|_{L^{2p}_M}\le C_T\bigl(1+\tau^{-1/(n-1)}\bigr).
\]
\end{comment}
%In addition, both \eqref{eq:CDI_differential} and \eqref{eq:CDI_bound} hold for $v$ itself for $t\in (0,T^*)$, where $\Xi_{k,T}:=\sup_{0\leq s\leq T}\norm{Y^{:k:}(s)}_{\Binf{-\alpha k}}$.
As a consequence, \eqref{eq:CDI_bound} holds true for $v$. 
\end{prop}
\begin{proof}
Write $\Psi(v,Y)=-v^n+\Psi'(v,Y)$, $\Psi'_{\varepsilon}(s):=\Psi'(v_{\varepsilon}(s),Y_{\varepsilon}(s))$, and apply Lemma \ref{lem: test Z p-1} with exponent $2p$:
\begin{align*}
  \frac{1}{2p}\partial_t\norm{v_{\varepsilon}(t)}_{L^{2p}_M}^{2p}
  &+C_p\mathcal{E}(v_{\varepsilon}^{p}(t),v_{\varepsilon}^{p}(t))
  +\norm{v_{\varepsilon}(t)}_{L^{2p}_M}^{2p}
  +\norm{v_{\varepsilon}(t)}_{L^{n+2p-1}_M}^{n+2p-1}\\
  &= \braket{\Psi'_{\varepsilon}(t),v^{2p-1}_{\varepsilon}(t)}_{L^2_M}.
\end{align*}
%Following \cite{MW17,TW18}, we 
The goal is to control $\abs{\braket{\Psi'_{\varepsilon}(t),v_{\varepsilon}^{2p-1}(t)}_{L^2_M}}$ using $K_t$ and $L_t$, where 
%Let us denote 
\[K_t:={\mathcal E}(v_{\varepsilon}^p(t),v_{\varepsilon}^p(t))\text{ and } L_t:=\norm{v_{\varepsilon}(t)}_{L^{n+2p-1}_M}^{n+2p-1}.\]
%The goal is to control $\abs{\braket{\Psi'_{\varepsilon}(t),v_{\varepsilon}^{2p-1}(t)}_{L^2_M}}$ using $K_t$ and $L_t$.
The triangle inequality gives \[\abs{\braket{\Psi'_{\varepsilon}(t),v^{2p-1}_\varepsilon(t)}_{L^2_M}}\leq \sum_{k=1}^n \binom{n}{k}\abs{\braket{Y^{:k:}_\varepsilon(t),v_\varepsilon^{n-k+2p-1}(t)}_{L^2_M}},\]
and we will estimate each term in RHS above individually. %{\color{blue}We first argue for $k\in \set{1,\dots,n-1}$.} 
We now pick a sufficiently small $\delta>0$ so that $\alpha' = \alpha-\delta>\alpha_0$. By Lemma \ref{lem:Besov Duality},  it holds uniformly in $\varepsilon \in (0,1]$ and $t\in [0,T^*_\varepsilon]$ that
\begin{equation}\label{ineq:besov_duality_global_remainder}
\begin{aligned}
    \abs{\braket{Y^{:k:}_\varepsilon(t),v^{n-k+2p-1}_\varepsilon(t)}_{L^2_M}}&\lesssim \norm{Y^{:k:}_\varepsilon(t)}_{\mathcal B^{-\alpha' k }_\infty}\norm{v^{n-k+2p-1}_\varepsilon(t)}_{\mathcal B_{1,1}^{\alpha' k }}\\
    &\lesssim \norm{Y^{:k:}_s}_{\Binf{- \alpha' k}} \norm{v_\varepsilon ^{n-k+2p-1}(t)}_{\Besovp{\alpha k }{1,\infty}},
\end{aligned}
\end{equation}
where we used the fact that $\norm{f}_{\mathcal B^{\beta}_{1,1}}\lesssim \norm{f}_{\mathcal B^{\beta'}_{1,\infty}}$ for every $\beta< \beta'$ in the last inequality. % (this follows from the definition and the fact that integrability in one polynomial weight implies boundedness after multiplying polynomial weights of higher degree), we will pick for each $k$ such an $\alpha_k<\min(\beta_k,\frac{d_w}{2})$ and estimate $\norm{v^{n-k+2p-1}_\varepsilon(t)}_{\mathcal B_{1,\infty}^{\alpha_k}}$ instead.\\
 Recall $\norm{\cdot}_{\dot{\mathcal B}_{p,q}^\alpha}$ in \eqref{eq:besov norm def pq}, we see by  Lemma \ref{lem: Besov Interpolation} and \eqref{ineq:global solution condition equivalent}
\begin{align}\label{eq: apply besov duality}
    \norm{v^{n-k+2p-1}_\varepsilon(t)}_{\mathcal B_{1,\infty}^{\alpha k }}&\lesssim \norm{v^{n-k+2p-1}_\varepsilon(t)}_{L^1_M}^{1-\frac{2\alpha k}{d_w}}\norm{v^{n-k+2p-1}_\varepsilon(t)}_{\mathcal B^{d_w/2}_{1,\infty}}^{\frac{2\alpha k}{d_w}}\nonumber\\
    &\le\norm{v^{n-k+2p-1}_\varepsilon(t)}_{L^1_M}^{1-\frac{2\alpha k}{d_w}}\left(\norm{v^{n-k+2p-1}_\varepsilon(t)}_{L^1_M}+\norm{v^{n-k+2p-1}_\varepsilon(t)}_{\dot{\mathcal B}^{d_w/2}_{1,\infty}}\right)^{\frac{2\alpha k}{d_w}},
\end{align}
where we used the fact that $\norm{P_1f}_{L^p_M} \le \norm{f}_{L^p_M}$ for all $f \in L^p_M$ in the last inequality.

Since $\mu(M)<+\infty$, we have by Jensen's inequality,
$$
\norm{v^{n-k+2p-1}_\varepsilon(t)}_{L^1_M} = \norm{v_\varepsilon(t)}_{L^{n-k+2p-1}_M}^{n-k+2p-1} \lesssim \norm{v_\varepsilon(t)}_{L^{n+2p-1}_M}^{n-k+2p-1}= L_t^{\frac{n-k+2p-1}{n+2p-1}}.
$$
Since $p \ge 2-n$ is  odd, we can apply Lemma \ref{lemma: Chain rule workaround} and obtain
$$
\norm{v^{n-k+2p-1}_\varepsilon(t)}_{\dot{\mathcal B}^{d_w/2}_{1,\infty}}=\norm{v^{2p+(n-k+1)-2}_\varepsilon(t)}_{\dot{\mathcal B}^{d_w/2}_{1,\infty}}\lesssim \norm{v^{2p+2(n-k+1)-4}(t)}_{L^1_M}^{\frac{1}{2}}K_t^{\frac{1}{2}}.
$$
Let $2<q_0:= \frac{2d_h/d_w}{d_h/d_w -1}$ and note that $2p+2(n-k+1)-4= p(2+\frac{4(n-k-1)}{2p})$. Let $p \ge 1$ sufficiently large so that $(2+\frac{4(n-k-1)}{2p})\le q_0$ and apply Lemma \ref{lem: HLS} and Jessen's inequality to see
%By Lemma \ref{lem: HLS},  if $2p>\frac{4n-8}{q-2}$, where $q=\frac{d_h/d_w}{\frac{d_h}{d_w}-1}$(note that for $d_h=d_w$, any $p\geq 1$ works), then for all $1\leq k\leq n-1$ we have $2p+2(n-k+1)-4=p(2+\frac{4(n-k-1)}{2p})$ with $2+\frac{4(n-k-1)}{2p}\leq q$, thus %=A_t^{1+\frac{2(n-k-1)}{2p}}.$$
\begin{align*}
  \norm{v^{2p+2(n-k+1)-4}_\varepsilon(t)}_{L^1_M}&=\norm{v_{\varepsilon}^p(t)}_{L^{2+\frac{4(n-k-1)}{2p}}_M}^{2+\frac{4(n-k-1)}{2p}}\lesssim \left[\norm{v_\varepsilon^p(t)}_{L^2_M}+K_t^\frac{1}{2}\right]^{2+\frac{4(n-k-1)}{2p}}\\
  &\lesssim\left[\norm{v_\varepsilon^p(t)}_{L^2_M}^2+K_t\right]^{1+\frac{2(n-k-1)}{2p}}=\left[\norm{v_\varepsilon(t)}_{L^{2p}_M}^{2p}+K_t\right]^{1+\frac{2(n-k-1)}{2p}}\\
  &\lesssim \left[\norm{v_\varepsilon(t)}_{L^{n+2p-1}_M}^{2p}+K_t\right]^{1+\frac{2(n-k-1)}{2p}}\lesssim \left[L_t^{\frac{2p}{n+2p-1}}+K_t\right]^{1+\frac{2(n-k-1)}{2p}}.
\end{align*}
%Where the last inequality is another application of Jensen. 
Inserting those estimates into \eqref{eq: apply besov duality} to obtain
\begin{align*}
    \norm{v_\varepsilon^{n-k+2p - 1}}_{\mathcal B^{\alpha k}_{1,\infty}} &\lesssim L_t^{\frac{n-k+2p-1}{n+2p-1}\pr{1-\frac{2\alpha k }{d_w}}}\pr{L_t^{\frac{n-k+2p-1}{n+2p-1}}+ \left[L_t^{\frac{2p}{n+2p-1}}+K_t\right]^{\frac{1}{2}+\frac{(n-k-1)}{2p}}K_t^\frac{1}{2}}^\frac{2\alpha k}{d_w}\\
    &\le L_t^{\frac{n-k+2p-1}{n+2p-1}\pr{1-\frac{2\alpha k }{d_w}}}\pr{L_t^{\frac{n-k+2p-1}{n+2p-1}}+ \left[L_t^{\frac{2p}{n+2p-1}}+K_t\right]^{1+\frac{(n-k-1)}{2p}}}^\frac{\alpha k}{d_w}.
\end{align*}
Inserting this back into \eqref{ineq:besov_duality_global_remainder} gives
\begin{align*}
    &\abs{\braket{Y^{:k:}_\varepsilon(t),v^{n-k+2p-1}_\varepsilon(t)}_{L^2_M}}\\
    \lesssim&\norm{Y_\varepsilon^{:k:}}_{\mathcal B^{-\alpha k +\delta}_\infty} L_t^{\frac{n-k+2p-1}{n+2p-1}\pr{1-\frac{2\alpha k }{d_w}}}\pr{L_t^{\frac{n-k+2p-1}{n+2p-1}}+ \left[L_t^{\frac{2p}{n+2p-1}}+K_t\right]^{1+\frac{(n-k-1)}{2p}}}^\frac{2\alpha k}{d_w}.
\end{align*}
\begin{comment}
Combining the above, we have for some $n\frac{d_h-d_w}{2}<\beta_k<\alpha_k<0$,
\begin{align*}
  &\abs{\braket{Y^{:k:}_\varepsilon(t),v^{n-k+2p-1}_\varepsilon(t)}_{L^2_M}}\\
  \lesssim& \norm{Y^{:k:}_\varepsilon(t)}_{\mathcal B^{-\beta}} L_t^{\frac{n-k+2p-1}{n+2p-1}(1-\frac{2\alpha_k}{d_w})}\left(L_t^{\frac{n-k+2p-1}{n+2p-1}}+\left[L_t^{\frac{4p}{n+2p-1}}+K_t\right]^{1+\frac{2(n-k-1)}{2p}}\right)^{\frac{2\alpha_k}{d_w}}.
\end{align*}
\end{comment}
We now want to apply Young's product inequality alongside Lemma \ref{lem:elementary comparison}. First, we choose $p$ large enough such that $(1+\frac{n-k-1}{2p})\frac{2\alpha k}{d_w}<1$ for each $k = 1,\dots, n$, then choose $r_1(k),r_2(k),$ and $r_3(k)$ so that
\begin{align*}
  1=&\frac{1}{r_1(k)}+\frac{1}{r_2(k)}+\frac{1}{r_3(k)}\text{ and }\\
  1>& \frac{n-k+2p-1}{n+2p-1}\left(1-\frac{2\alpha k}{d_w}\right)r_2(k), \frac{n-k+2p-1}{n+2p-1}\left(\frac{2\alpha k}{d_w}\right)r_3(k),\frac{2p+(n-k-1)}{2p}\frac{2\alpha k}{d_w}r_3(k)\\
  =:&\gamma(k),\gamma'(k),\gamma''(k).
\end{align*}
Indeed, such choices are possible since 
\begin{align*}
    \lim_{p \to \infty} \frac{n-k+2p-1}{n+2p-1}\pr{1-\frac{2\alpha k}{d_w}} + \frac{2p + (n-k-1)}{2p} \frac{2\alpha k }{d_w} = 1.
\end{align*}

Applying Young's product inequality to see for each $k \in \set{1,\dots, n}$ 
\[
\abs{\braket{Y^{:k:}_\varepsilon(t),v^{n-k+2p-1}_\varepsilon(t)}_{L^2_M}}\lesssim \norm{Y^{:k:}_\varepsilon(t)}_{\mathcal B^{-\alpha k +\delta }_\infty}^{r_1(k)}+L_t^{\gamma(k)}+L_t^{\gamma'(k)}+K_t^{\gamma''(k)}.
\]
Combined, they give us 
\[
\frac{1}{2p}\partial_t\norm{v_\varepsilon(t)}_{L^{2p}_M}^{2p}\lesssim-C_pK_t-L_t+\sum_{k=1}^{n}\pr{\norm{Y^{:k:}_\varepsilon(t)}_{\mathcal B^{-\alpha k+\delta}_\infty}^{r_1(k)}+L_t^{\gamma(k)}+L_t^{\gamma'(k)}+K_t^{\gamma''(k)}}
\]
 Recall the elementary inequality that for any $c,c'>0$ and $\gamma \in (0,1)$, it holds uniformly in $x \ge0$ that $-cx+c'x^\gamma\lesssim 1$. Applying this to  $K_t$ leaves us with 
\[\frac{1}{2p}\partial_t\norm{v_\varepsilon(t)}_{L^{2p}_M}^{2p}\lesssim1-L_t+\sum_{k=1}^{n}\norm{Y^{:k:}_\varepsilon(t)}_{\mathcal B^{-\alpha k}_\infty}^{r_1(k)}+L_t^{\gamma(k)}+L_t^{\gamma'(k)},
\]
 which holds uniformly in $t\in[0,T^*_\varepsilon]$ and $\varepsilon \in (0,1]$. 
The elementary inequality $x^\gamma\leq cx+\delta$ with $c<1$ allowed to be arbitrarily close to 0 (at the cost of larger $\delta$) applied to $L_t^{\gamma'(k)}$ and $L_t^{\gamma''(k)}$ gives \[\frac{1}{2p}\partial_t\norm{v_\varepsilon(t)}_{L^{2p}_M}^{2p}+cL_t\lesssim 1+\sum_{k=1}^n \pr{\Xi^\varepsilon}_{k,T}^{r_1(k)}.\]
By Jensen, we have \(L_t=\norm{v_{\varepsilon}(t)}_{L^{n+2p-1}_M}^{n+2p-1}\gtrsim \norm{v_{\varepsilon}(t)}_{L^{2p}_M}^{n+2p-1} = \norm{v_\varepsilon(t)}_{L^{2p}_M}^{2p(1+\frac{n-1}{2p})}\). Hence
\[\frac{1}{2p}\partial_t\norm{v_\varepsilon(t)}_{L^{2p}_M}^{2p}+c'\pr{\norm{v_{\varepsilon}(t)}_{L^{2p}_M}^{2p}}^{1+\frac{n-1}{2p}}\lesssim 1+\sum_{k=1}^n \pr{\Xi^\varepsilon}_{k,T}^{r_1(k)}.\]
%This is exactly \eqref{eq:CDI_differential} with $q_k=r_1(k)$. 
%{\issue Need time-weighted bounds on Wick product norms to do some stuff here. After that, this theorem should be good.}\\
 We note that the implicit constants for the above inequalities are independent of $\varepsilon \in (0,1]$, hence by Lemma \ref{lem:elementary comparison}, we see \eqref{eq:CDI_bound} holds. 
%Applying Lemma \ref{lem:elementary comparison} gives us \eqref{eq:CDI_bound}. 

Lastly, we see by Lemma \ref{lem:local solution to DPD and convergence of approximation} that $v_\varepsilon(t)\to v(t)$ in $L^\infty_M$ as $\varepsilon \downarrow 0$ for each $t\in (0,T^*]$. Since $\Xi_{k,T}^\varepsilon$ is uniformly bounded in $\varepsilon$ for each $1 \le k \le n$ and $T>0$, we see the desired results also hold for $v$. 

%we pass the bounds from $v_\varepsilon$ to $v$ using Lemma \ref{lem:local solution to DPD and convergence of approximation}, which is allowed by our choice of $\beta_k>k\frac{d_h-d_w}{2}$ thanks to Lemma \ref{lem: holder besov space of wick product}.
\end{proof}

\begin{corollary}[Global well-posedness of Remainder Equation]\label{cor:global_wp}
Take $\gamma>0$, $\alpha>\alpha_0$, $2\le n \le n_0$ be an odd integer that satisfies the conditions of Lemma \ref{lem:local solution to DPD and convergence of approximation}. Let $Y$ be the solution to \eqref{eq: EW eqn} with $Y_0 \in \mathfrak{C}$ and $v_0 \in \Binf{-\alpha}$. Then there exists a unique solution $ \in \mathcal C(\R_+; \Binf{\gamma})$ of equation \eqref{eq:formal eq for v intro}.
\end{corollary}
\begin{proof}
Let $T>0$ be arbitrary but fixed, and fix $p \ge 2$ sufficiently large so that $L^p \hookrightarrow \Binf{-\alpha}$ (see Proposition \ref{prop:besov_embed}). Let $\underline{\mathbf{Y}}_n:=(Y,Y^{:2:},\cdots,Y^{:n:})$. 
Recall the definition of $\Xi_{k,T}$ from Proposition \ref{prop:CDI_mollified}, and observe that
\[
\sum_{k=1}^n \Xi_{k,T}^{q_{k,p}}\leq n\norm{\underline{\mathbf{Y}}_n}_{\mathcal G_T}^{Q}+n\norm{\underline{\mathbf{Y}}_n}_{\mathcal G_T}^{q},\qquad Q:=\max_{1\leq k\leq n}q_{k,p},\, q:= \min_{1\leq k\leq n}q_{k,p}.
\] 

Hence,  Proposition \ref{prop:CDI_mollified} gives an estimate depending only on $\norm{\underline{\mathbf{Y}}_n}_{\mathcal G_T}$, which provides an a priori estimate on $\norm{v_t}_{\Binf{-\alpha}}$. Thus by Lemma \ref{lem:local solution to DPD and convergence of approximation} there exists $T^*:=T^*(\norm{v_t}_{\Binf{-\alpha}}, \norm{\underline{\mathbf Y}_n}_{\mathcal G_T}) >0$ and a unique solution up to $T^*$ in $\mathcal C([0,T^*];\Binf{\gamma})$ of \eqref{eq:formal eq for v intro}. Use again Lemma \ref{lem:local solution to DPD and convergence of approximation} to construct a solution to \eqref{eq:formal eq for v intro} on $[T^*,(2T^*)\wedge T]$ with initial condition $v_{T^*}$ which satisfies the same $L^p_M$ norm bound (and hence $\Binf{-\alpha}$ norm bound) depending again only on $\norm{\underline{\mathbf{Y}}_n}_{\mathcal G_T}$. We may iterate this process until the whole interval $(0,T]$ is covered.
\end{proof}
\begin{comment}
\begin{corollary}[Global well-posedness]\label{cor:global_wp}
Let $v_0$ be an admissible initial condition for Proposition \ref{prop:local_solution_with_weights}, and let
\[
Z\in \bigcap_{\alpha\in(\alpha_{n-1},\,d_w-\alpha_n)} C\bigl([0,T^*);B^\alpha_\infty\bigr)
\]
be the corresponding maximal local solution of \eqref{eq:formal eq for v intro}, defined up to explosion time $T^*\in(0,\infty]$. Then $T^*=\infty$ almost surely. In particular, \eqref{eq:formal eq for v intro} admits a unique global solution.
\end{corollary}

\begin{proof}
Fix $T>0$. By Proposition~\ref{prop:CDI_mollified}, for every $t\in(0,T\wedge T^*)$,
\[
\norm{Z(t)}_{L_M^{2p}}\le C_T\pr{1+t^{-1/(n-1)}},
\]
hence
\[
\sup_{\tau\le t\le T\wedge T^*}\norm{Z(t)}_{L^{2p}_M}<\infty
\qquad\text{for every }\tau>0.
\]
By Proposition~\ref{prop:local_solution_with_weights}, restarting the equation at any time $\tau\in(0,T\wedge T^*)$ yields a local solution whose lifespan depends only on the norm of the datum $v(\tau)$ and the corresponding stochastic coefficients on a short interval after $\tau$. Since the above $L_M^{2p}$ bound is finite uniformly on $[\tau,T\wedge T^*)$, and {\issue the Wick powers are controlled on finite intervals by Section~5}, the local existence time cannot degenerate as $t\uparrow T^*$.

If $T^*<\infty$, choose $\tau<T^*$ sufficiently close to $T^*$ so that the local theory started from $v(\tau)$ extends past $T^*$. This contradicts maximality of $T^*$. Therefore $T^*=\infty$ almost surely.
\end{proof}
\end{comment}
\noindent The following corollary is a direct consequence of  Proposition \ref{prop:CDI_mollified} and Lemma \ref{lem:uniform unit interval moments of wick powers}. We omit its proof.
\begin{corollary}[Positive-time moment bound]\label{cor:moment_bound}
Under the assumptions of Corollary \ref{cor:global_wp}, let $p\ge 1$ be as in Proposition~\ref{prop:CDI_mollified} and denote by $v(\cdot;v_0)$ the solution to \eqref{eq:formal eq for v intro} starting from $v_0$. Then for every $q\ge 1$,
\[
  \sup_{v_0\in \Binf{-\alpha}}\sup_{0<t<\infty}
  \pr{t^{q/(n-1)}\wedge 1}\,\E\Big[\|v(t;v_0)\|_{L_M^{2p}}^q\Big]
  <\infty.
\]
\end{corollary}

\begin{comment}
    \begin{proof}
Fix $q\ge 1$. By Proposition~\ref{prop:CDI_mollified} %(applied to $v$ itself), together with the differential inequality\eqref{eq:CDI_differential} and Lemma~\ref{lem:elementary comparison}, 
we obtain the following unit-interval estimate: for each
$a\ge 0$, there exist exponents $q_k>0$ (independent of $a$ and $v_0$) such that for all $r\in(0,1]$,
\[
  \norm{v(a+r;v_0)}_{L_M^{2p}}
  \lesssim
  \pr{1+r^{-\frac{1}{n-1}}}
  \pr{1+\sum_{k=1}^n
  \norm{Y^{:k:}}_{\mathcal C([a,a+1];\Binf{-\beta_k})}^{q_k}}.
\]
Take first $a=0$. By Lemma~\ref{lem:uniform unit interval moments of wick powers}, the random factor
\(
  1+\sum_{k=1}^n \norm{Y^{:k:}}_{\mathcal C([0,1];\Binf{-\beta_k})}^{q_k}
\)
has finite moments of every order. Hence
\[
  \sup_{v_0}\sup_{0<t\le 1}
  t^{q/(n-1)}\E\Big[\norm{v(t;v_0)}_{L_M^{2p}}^q\Big]<\infty.
\]
Now let $t\ge 1$ and set $a=t-1$, $r=1$. Then
\[
  \norm{v(t;v_0)}_{L_M^{2p}}
  \lesssim
  1+\sum_{k=1}^n
  \norm{Y^{:k:}}_{\mathcal C([t-1,t];\Binf{-\beta_k})}^{q_k}.
\]
Applying Lemma~\ref{lem:uniform unit interval moments of wick powers} again gives
\[
  \sup_{v_0}\sup_{t\ge 1}
  \E\Big[\norm{v(t;v_0)}_{L_M^{2p}}^q\Big]<\infty.
\]

Since $t^{q/(n-1)}\wedge 1=t^{q/(n-1)}$ on $(0,1]$ and equals $1$ on $[1,\infty)$, combining the two displays yields
\[
  \sup_{v_0}\sup_{0<t<\infty}
  \pr{t^{q/(n-1)}\wedge 1}\,\E\Big[\|v(t;v_0)\|_{L_M^{2p}}^q\Big]
  <\infty.
\]
\end{proof}
\end{comment}

\section{The solution as a Markov process and its invariant measure}\label{sec:inv meas}

This section constructs invariant measures for the renormalized $\Phi^{n+1}$ dynamics  as a process in $\mathfrak C$ in Definition \ref{def:c set}.
We first recall the global solution framework from \cite{chandra2025non}, then prove that the associated transition operators form a Feller semigroup, derive a positive-time moment bound, and finally apply the classical Bogoliubov--Krylov theorem (in the style of \cite[Proposition 4.4]{TW18}) to obtain existence of invariant measures.

 Recall $\alpha_0$ as in \eqref{eq:alpha_0} and $n_0 \in \N$ as in \eqref{eq:n0}. Through out this section, we assume both \eqref{ineq: dbd condition} and \eqref{ineq:global solution condition} hold for some odd integer $3 \le n \le n_0$. We shall fix $\alpha>\alpha_0$ so that \eqref{ineq:regularity condition general alpha} and \eqref{ineq:global solution condition equivalent} holds. As a consequence of Corollary \ref{cor:global_wp}, the remainder equation \eqref{eq:formal eq for v intro} admits a global solution for any $v \in \mathcal B^{-\alpha}_\infty$. 
\begin{definition}[{\cite[Definition 3.15]{chandra2025non}}]\label{def:global solution of phi n equation}
Fix an arbitrary $Y_0 \in \mathfrak{C}$.  Let $\phi_0 \in \mathcal B^{-\alpha}_\infty$ and 
$\gamma >0$ and $v \in \mathcal C((0,\infty);\Binf{\gamma})$ be the solution to \eqref{eq:formal eq for v intro} with initial condition $v_0 = \phi_0 - Y_0 \in \Binf{-\alpha}$ given by Corollary \ref{cor:global_wp}. We define the solution to \eqref{eq:phi n formal} with initial condition $\phi_0$ by
\[
  \phi_t(\phi_0):=v_t+Y_t \in \Binf{-\alpha},\qquad t\ge 0.
\]
\end{definition}

\begin{remark}
By \cite[Proposition 3.16]{chandra2025non}, the above definition is independent of the choice of $Y_0$.
\end{remark}

\begin{lemma}[{\cite[Lemma 3.17]{chandra2025non}}]\label{lem:u t in c}
%Under the setting of Definition \ref{def:global solution of phi n equation}, 
For every $t>0$, $\phi_t \in \mathfrak{C}$ almost surely.
\end{lemma}

For completeness, we first show that this solution defines a Markov process.

\begin{lemma}[Markov property of $u$]\label{lem:markov property of global solution}
Under the setting of Definition \ref{def:global solution of phi n equation}, for each $\phi_0\in\Binf{-\alpha}$ let
\(\phi=\phi(\cdot;\phi_0)=\{\phi=\phi_t(\phi_0)\}_{t\ge 0}\)
denote the corresponding global solution with initial condition $\phi_0$. For bounded Borel $f:\Binf{-\alpha}\to\R$, define
\[
  \mathbf P_tf(v):=\E\left[f\pr{\phi_t(v)}\right],\qquad t\ge 0,\,v\in\Binf{-\alpha}.
\]
Then $\{\mathbf P_t\}_{t\ge 0}$ is a Markov semigroup, and for every $s,t\ge 0$,
\begin{align}\label{eq:markov property white noise filtration}
  \E\left[f\pr{\phi_{s+t}(\phi_0)}\mid \mathcal F_s\right]
  =\mathbf P_tf\pr{\phi_s(\phi_0)}\qquad\text{a.s.}
\end{align}
\end{lemma}

\begin{proof}
Fix $\phi_0\in\Binf{-\alpha}$. By Definition \ref{def:global solution of phi n equation}, write the corresponding solution as
\[
  \phi=Y+v,
\]
where $Y$ solves \eqref{eq: EW eqn} and $v$ solves \eqref{eq:formal eq for v intro} globally by Corollary \ref{cor:global_wp}.

Fix $s\ge 0$ and define the time-shifted noise $\xi^{(s)}$ by
\[
  \xi^{(s)}(\varphi):=\xi\pr{\varphi(\cdot-s,\cdot)\1_{[s,\infty)}(\cdot)},\qquad \varphi\in L^2(\R\times M).
\]
By independent increments of white noise, $\xi^{(s)}$ is independent of $\mathcal F_s$ and has the same law as $\xi$.
%{\issue Maybe use $\tau_sY_t:=Y_{s+t}$ instead?}
Let
\[
  \tau_sY_t:=Y_{s+t},\qquad \tau_s v_t:=v_{s+t},\qquad \tau_s \phi_t:=\phi_{s+t},\qquad t\ge 0.
\]
By the mild formulations of \eqref{eq: EW eqn} and \eqref{eq:formal eq for v intro},
$(\tau_sY,\tau_sv)$ solves the same renormalized system on $[0,\infty)$, driven by $\xi^{(s)}$, with initial datum
$(\tau_sY_0,\tau_s v_0)=(Y_s,v_s)$; hence $\tau_s\phi$ is a renormalized solution started from $\phi_s$.
For $s>0$, Lemma \ref{lem:u t in c} gives $\phi_s\in\mathfrak C$ almost surely, so Definition \ref{def:global solution of phi n equation} applies at time $s$; for $s=0$, the claim is trivial. Therefore, by Corollary \ref{cor:global_wp} and independence of the splitting \cite[Proposition 3.16]{chandra2025non}, for each $t,s>0$
\[
  \phi_{s+t}(\phi_0)=\phi_t\pr{\phi_s(\phi_0);\xi^{(s)}}\qquad
  \text{almost surely,}
\]
where $\phi_t(\phi_0,\eta)$ is the solution to \eqref{eq:phi n formal} with initial condition $\phi_0$ and noise input $\eta$. 
Now let $f:\Binf{-\alpha}\to \R$ be bounded and Borel. Conditioning on $\mathcal F_s$ and using the previous identity,
\[
  \E\left[f\pr{\phi_{s+t}(\phi_0)}\mid\mathcal F_s\right]
  =\E\left[f\pr{\phi_t\pr{\phi_s(\phi_0);\xi^{(s)}}}\mid\mathcal F_s\right].
\]
Since $\xi^{(s)}$ is independent of $\mathcal F_s$ and has the same law as $\xi$, the conditional law of
$ \phi_t(\phi_s(\phi_0);\xi^{(s)})$ given $\mathcal F_s$ equals the law of $ \phi_t(\phi_s(\phi_0);\xi)$. Hence
\[
  \E\left[f\pr{\mathcal \phi_t\pr{\phi_s(\phi_0);\xi^{(s)}}}\mid\mathcal F_s\right]
  =\mathbf P_tf\pr{\phi_s(\phi_0)},
\]
which proves \eqref{eq:markov property white noise filtration}.

Taking expectation in \eqref{eq:markov property white noise filtration}, for every $v\in\Binf{-\alpha}$,
\[
  \mathbf P_{s+t}f(v)
  =\E\left[\mathbf P_tf\pr{\phi_s(v)}\right]
  =\mathbf P_s\pr{\mathbf P_tf}(v),
\]
so $\{\mathbf P_t\}_{t\ge 0}$ is a semigroup.
Also $\mathbf P_0f=f$, so $\{\mathbf P_t\}_{t\ge 0}$ is a Markov semigroup.
\end{proof}

\begin{prop}[Feller property]\label{prop:feller property}
Under the assumptions of Definition \ref{def:global solution of phi n equation}, the Markov semigroup
\(\{\mathbf P_t\}_{t\ge 0}\) from Lemma \ref{lem:markov property of global solution} is Feller on
\(\Binf{-\alpha}\): for every $t\ge 0$ and every bounded continuous
\(f:\Binf{-\alpha}\to\R\), the map
\[
  \Binf{-\alpha}\ni v\mapsto \mathbf P_tf(v)
\]
is bounded and continuous.
\end{prop}

\begin{proof}
Recall that by Corollary \ref{cor:global_wp}, the renormalized solution to \eqref{eq:phi n formal} is pathwise unique and depends continuously on the initial condition in the space $\Binf{-\alpha}$ for an appropriate $\alpha>0$.
Fix $t\ge 0$ and $f\in \mathcal C_b(\Binf{-\alpha})$,  We note that boundedness is immediate, so it remains to prove continuity. The case $t=0$ is trivial, so let $t>0$, and let
$v^{(m)}_0\to v_0$ in $\Binf{-\alpha}$.

Choose one $Y_0\in\mathfrak C$ and use the same driving noise $\xi$ and the same process $Y$ for all initial data. Write
\[
  \phi(v^{(m)}_0)=Y+v^{(m)},\qquad \phi(v_0)=Y+v,
\]
where $v^{(m)}$ and $v$ solve \eqref{eq:formal eq for v intro} with initial conditions
$v^{(m)}_0-Y_0$ and $v_0-Y_0$, respectively.

Since $v^{(m)}_0\to v_0$ in $\Binf{-\alpha}$, there exists $R>0$ such that
\(
  \sup_m \|v^{(m)}_0\|_{\Binf{-\alpha}}\vee \norm{v_0}_{\Binf{-\alpha}}\le R.
\)
Fix $\omega$ in a full-probability set where all Wick powers are defined continuously on $[0,t]$.
Set
\[
  M_t(\omega):=\max_{1\le k\le n}\sup_{0\le s\le t}
  \norm{Y^{:k:}(s,\omega)}_{\Binf{-\alpha k}}<\infty.
\]

By Proposition \ref{prop:CDI_mollified} (applied to $v$ itself) and the embedding used in Corollary \ref{cor:global_wp}, there exists a finite random constant $A_t(\omega)$, independent of the initial condition, such that
\[
  \sup_{0<s\le t}\sup_{m\ge 1}\norm{v^{(m)}(s,\omega)}_{\Binf{-\alpha}}
  \vee
  \sup_{0<s\le t}\norm{v(s,\omega)}_{\Binf{-\alpha}}
  \le A_t(\omega).
\]

Now apply Proposition \ref{prop:local_solution_with_weights} pathwise on successive time intervals. Since the coefficients are bounded by $M_t(\omega)$ on $[0,t]$ and all restart data are bounded by $A_t(\omega)$, there exists
\(
  \delta(\omega)>0
\)
(depending only on $A_t(\omega)$ and $M_t(\omega)$) such that the local solution map on any interval of length at most $\delta(\omega)$ is Lipschitz in the initial datum, with Lipschitz constant depending only on the same bounds.
Partition $[0,t]$ into at most
\(
  N(\omega):=\lceil t/\delta(\omega)\rceil
\)
subintervals and iterate the local Lipschitz estimate. This gives
\[
  \norm{v^{(m)}(t,\omega)-v(t,\omega)}_{\Binf{-\alpha}}
  \le C_{t}(\omega)\norm{v^{(m)}_0-v_0}_{\Binf{-\alpha}},
\]
for some finite $C_t(\omega)$ independent of $m$. Hence
\(
  v^{(m)}(t,\omega)\to v(t,\omega)
\)
and therefore
\(
  \phi_t(v^{(m)},\omega)\to \phi_t(v,\omega)
\)
almost surely.

Because $f$ is bounded and continuous,
\(
  f(\phi_t(v^{(m)}))\to f(\phi_t(v))
\)
almost surely, and dominated convergence yields
\[
  \mathbf P_tf(v^{(m)})=\E\left[f\pr{\phi_t(v^{(m)})}\right]
  \longrightarrow
  \E\left[f\pr{\phi_t(v)}\right]=\mathbf P_tf(v).
\]
Thus $v\mapsto \mathbf P_tf(v)$ is continuous, as desired.
\end{proof}

\begin{corollary}[Positive-time moment bound for $u$]\label{cor:moment_bound_u}
Under the assumptions of Definition \ref{def:global solution of phi n equation}, for every
\(q\ge 1\),
\[
  \sup_{\phi_0\in\Binf{-\alpha}}\sup_{t>0}
  \pr{t^{q/(n-1)}\wedge 1}
  \E\Big[\norm{\phi_t(\phi_0)}_{\Binf{-\alpha}}^q\Big]
  <\infty.
\]
\end{corollary}

\begin{proof}
 Recall $\nu$ as in \eqref{def:GFF measure} and fix $q\ge 1$. By \cite[Proposition 3.16]{chandra2025non}, the law of the solution does not depend on the chosen splitting of initial value. Thus we may choose $Y_0\sim\nu$ (independent of $\xi$), and write
\[
  \phi_t(u_0)=Y_t+v_t,
\]
where $Y$ is stationary and $v$ solves \eqref{eq:formal eq for v intro} with initial condition $\phi_0 - Y_0$.

Choose $p\ge 1$ as in Corollary \ref{cor:moment_bound}, and enlarge $p$ if needed so that
\(
  \frac{d_h}{2p}<\alpha.
\)
Then Proposition \ref{prop:besov_embed} gives
\[
  \norm{f}_{\Binf{-\alpha}}\lesssim \norm{f}_{L^{2p}_M}.
\]
Hence Corollary \ref{cor:moment_bound} implies
\[
  \sup_{\phi_0\in\Binf{-\alpha}}\sup_{t>0}
  \pr{t^{q/(n-1)}\wedge 1}
  \E\Big[\norm{v_t}_{\Binf{-\alpha}}^q\Big]
  <\infty.
\]

For the linear part, stationarity yields
\[
  \E\Big[\norm{Y_t}_{\Binf{-\alpha}}^q\Big]
  =\E_{Y_0\sim\nu}\Big[\norm{Y_0}_{\Binf{-\alpha}}^q\Big]
  <\infty,
\]
where finiteness follows from Theorem \ref{thm:Wick Power Convergence} with $n=1$.

Using \(\norm{a+b}^q\lesssim_q \norm{a}^q+\norm{b}^q\), we get
\[
  \pr{t^{q/(n-1)}\wedge 1}
  \E\Big[\norm{\phi_t(\phi_0)}_{\Binf{-\alpha}}^q\Big]
  \lesssim
  \E\Big[\norm{Y_t}_{\Binf{-\alpha}}^q\Big]
  +
  \pr{t^{q/(n-1)}\wedge 1}
  \E\Big[\norm{v_t}_{\Binf{-\alpha}}^q\Big].
\]
Taking suprema in $\phi_0$ and $t$ gives the claim.
\end{proof}

\begin{prop}[Existence of invariant measure (Bogoliubov--Krylov)]\label{prop:existence invariant measure}
Under the assumptions of Definition \ref{def:global solution of phi n equation}, the Feller semigroup
\(\{\mathbf P_t\}_{t\ge 0}\) from Lemma \ref{lem:markov property of global solution} admits at least one invariant probability measure
\(\mu_\star\) on \(\Binf{-\alpha}\). Moreover,
\[
  \mu_\star(\mathfrak C)=1,
\]
so \(\mu_\star\) is supported on \(\mathfrak C\) (by Lemma \ref{lem:u t in c}).
\end{prop}

\begin{proof}
We follow the Bogoliubov--Krylov strategy as in \cite[Proposition 4.4]{TW18}. 
Pick $\alpha'\in(\alpha_0,\alpha)$,  Corollary \ref{cor:moment_bound_u}  gives, for every $q\ge 1$,
\[
  \sup_{t\ge 1}\E\Big[\norm{\phi_t(0)}_{\Binf{-\alpha'}}^q\Big]<\infty,
\]
where $\phi_t(0)$ is the solution started from $0$.

Recall that by Theorem \ref{thm:compact embedding} (with $p=q=\infty$), the embedding
\[
  \Binf{-\alpha'}\hookrightarrow\Binf{-\alpha}
\]
is compact. For $R>0$, let
\[
  B'_R:=\set{f\in\Binf{-\alpha'}:\norm{f}_{\Binf{-\alpha'}}\le R},
\]
and let $K_R$ be its closure in $\Binf{-\alpha}$; then $K_R$ is compact in $\Binf{-\alpha}$ by Theorem \ref{thm:compact embedding}.
Hence Markov's inequality yields
\[
  \sup_{t\ge 1}\P\pr{\phi_t(0)\notin K_R}
  \le
  \sup_{t\ge 1}\P\pr{\norm{\phi_t(0)}_{\Binf{-\alpha'}}>R}
  \lesssim_q R^{-q}.
\]
Therefore \(\{\mathrm{Law} (\phi_t(0)):t\ge 1\}\) is uniformly tight in $\Binf{-\alpha'}$. Hence, the family of probability measures $\{\mu_T\}_{T \ge 1}$ on $\Binf{-\alpha'}$ given by 
\[
  \mu_T(A):=\P\pr{\frac{1}{T}\int_1^{T+1}\mathrm{Law} \pr{\phi_t(0)}\,\mathrm dt\in A},\qquad T>0.
\] for any Borel set $A\subset \Binf{-\alpha'}$,  is tight. By Prokhorov, this family is compact in the space of probability measures on $\Binf{-\alpha'}$, and there exists $T_n\uparrow\infty$ and a probability measure
\(\mu_\star\) on $\Binf{-\alpha'}$ such that \(\mu_{T_n}\Rightarrow\mu_\star\).

It remains to prove invariance. Let $s\ge 0$ and $f\in \mathcal C_b(\Binf{-\alpha'})$. By Markov property,
\[
  \int_{\Binf{-\alpha'}} \mathbf P_sf\,\mathrm d\mu_T
  =\frac{1}{T}\int_1^{T+1}\E\left[f\pr{\phi_{t+s}(0)}\right]\mathrm dt
  =\frac{1}{T}\int_{1+s}^{T+1+s}\E\left[f\pr{\phi_r(0)}\right]\mathrm dr.
\]
Hence
\[
  \abs{\int_{\Binf{-\alpha'}} \mathbf P_sf\,\mathrm d\mu_T-\int_{\Binf{-\alpha'}} f\,\mathrm d\mu_T}
  \le \frac{2s\norm{f}_{L^\infty}}{T}\xrightarrow[T\to\infty]{}0.
\]
Using the Feller property of $\mathbf P_t$ as in Proposition \ref{prop:feller property}, \(\mathbf P_sf\in \mathcal C_b(\Binf{-\alpha'})\). Passing to the limit along $T_n$ gives
\[
  \int_{\Binf{-\alpha'}} \mathbf P_sf\,\mathrm d\mu_\star
  =\int_{\Binf{-\alpha'}} f\,\mathrm d\mu_\star,
\]
for all $s\ge 0$, $f\in \mathcal C_b(\Binf{-\alpha'})$. Thus $\mu_\star$ is invariant.

Finally, by Lemma \ref{lem:u t in c}, for every $t>0$ and every initial condition $\phi_0\in\Binf{-\alpha'}$,
\[
  \P\pr{\phi_t(\phi_0)\in\mathfrak C}=1,
\]
which implies that $\mu_\star$ is supported on $\mathfrak C$.
\end{proof}

\section{Appendix}\label{sec:appendix}
\subsection{Proof of Lemma \ref{lem:long time semi-group norm bound},\ref{lem:long time heat kernel bounds}, Proposition \ref{prop:kernel_bound_dt} and Lemma \ref{lem:hk LpLq interpolation} }
\begin{proof}[Proof of Lemma \ref{lem:long time semi-group norm bound}]
Let $a >0$ be fixed. Since $\int_M g(x) \mu(\mathrm dx) = 0$, we have
\begin{align*}
  e^{-aL}g(x) &= \int_M p_a(x,y) g(y) \mu(\mathrm dy)\\
  &= c_a\int_M  g(y) \mu(\mathrm dy)+\int_M \pr{p_a(x,y)-c_a} g(y) \mu(\mathrm dy)\\
  &=\int_M \pr{p_a(x,y)-c_a} g(y) \mu(\mathrm dy).
\end{align*}
Note $p_a(x,y)-c_a \ge 0$ and  $\int_M p_a(x,y)-c_a \mu(\mathrm dy) = 1 - c_a \in(0,1)$. Hence, by Jensen's inequality, we see for each $p \ge 1$,
\begin{align*}
  \norm{e^{-aL}g}_{L^p_M} &=\pr{\int_M \pr{\int_M \pr{p_a(x,y)-c_a} g(y) \mu(\mathrm dy)}^p \mu(\mathrm{d}x)}^\frac{1}{p}\\
  &\le \pr{\int_M (1-c_a)^{p-1}\int_M \pr{p_a(x,y) - c_a}\abs{g(y)}^p \mu(\mathrm dy)\mu(\mathrm dx)}^\frac{1}{p}\\
  &\le (1-c_a) \norm{g}_{L^p_M},
\end{align*}
where we used Fubini's and symmetry of the heat kernel in the last inequality. On the other hand, if $p = \infty$, then 
\begin{align*}
    \norm{e^{-aL}g}_{L^\infty_M}&\le \sup_{x \in M}\int_M \pr{p_a(x,y) -c_a}\abs{g(y)} \mu\pr{ \mathrm dy}\\
    &\le \sup_{x \in M} \int_M p_a(x,y) -c_a\mu(\mathrm dy) \norm{g}_{L^\infty_M}\\
    &=(1-c_a)\norm{g}_{L^\infty_M}.
\end{align*}

Finally, note that for all $s >0$, $\int_M e^{-sL}g(x) \mu(\mathrm{d}x) = 0$. Hence, for $t \ge a$, we see by the semi-group property that
\begin{align*}
  \norm{e^{-tL}g}_{L^p_M} &= \norm{e^{-\left \lfloor \frac{t}{a} \right  \rfloor a L}e^{-\pr{t-\left \lfloor \frac{t}{a} \right  \rfloor a}L}g}_{L^p_M}\\
  &\le (1-c_a)^{\left \lfloor \frac{t}{a}\right \rfloor} \norm{g}_{L^p_M}.
\end{align*}
\end{proof}

\begin{proof}[Proof of Lemma \ref{lem:long time heat kernel bounds}]
Since Diam$(M)< \infty$, we see by \eqref{ineq:sgu} that there are $0<a<b<\infty$ so that $ a<p_1(x,y) <b$, hence $p_1(x,y) \asymp 1$ uniformly in $x,y \in M$. Suppose now $t>0$ and consider
\begin{align*}
  p_{1+t}(x,y)  = \int_M p_1(x,z) p_t(z,y) \mu(\mathrm dz) \asymp \int_M p_t(z,y) \mu(\mathrm dz) = 1,
\end{align*}
where we used stochastic completeness \eqref{eq:s_complete} and the symmetric property of the heat kernel.

To show \eqref{ineq:holder_kernel} holds globally in $t >0$, it is enough to assume $t> 1$. Let $g_{x,y}(u):=p_1(x,u) - p_1(y,u)$ for $x,y,u \in M$, then $ \int_M g_{x,y}(u)\mathrm d\mu(u) = 0$, and 
$$
\norm{g_{x,y}}_{L^\infty_M} \overset{\eqref{ineq:holder_kernel}}{\lesssim} d(x,y)^\Theta,\qquad \text{ uniformly in }y,z\in M.
$$
Hence by Lemma \ref{lem:long time semi-group norm bound}, there exists $c>0$ so that the following inequalities holds uniformly in $t>1$ and $x,x' \in M$ 
\begin{align*}
    \abs{p_t(x,y) - p_t(x',y)}&=\abs{\int_M g_{x,x'}(u)p_{t-1}(u,y) \mu\pr{\mathrm du}}\\
    &\lesssim e^{-2ct}\norm{g_{x,x'}}_{L^\infty_M}\lesssim e^{-2ct}d(x,x')^\Theta\\
    &\lesssim e^{-ct} \pr{\frac{d(x,x')}{t^\frac{1}{d_w}}}^\Theta\exp \pr{-\pr{\frac{d(x,y)}{t^\frac{1}{d_w}}}^\frac{d_w}{d_w-1}}
\end{align*}
where we used the fact that $t\ge 1$ and $\mathrm{diam}(M) <\infty$ in the last line. This completes the proof.
\end{proof}

\begin{proof}[Proof of Proposition \ref{prop:kernel_bound_dt}]
Suppose $k\ge 1$ is an integer, by \cite[Theorem 4, Corollary 5]{davies1997non}, we see uniformly in $(t,x,y) \in (0,1]\times M^2$ that 
\begin{align*}
    \abs{t^k\partial_t^k p_t(x,y)} \lesssim t^{-d_h/d_w}\exp \pr{-C\pr{\frac{d(x,y)}{t^{1/d_w}}}^\frac{d_w}{d_w - 1}}.
\end{align*}
By stochastic completeness \eqref{eq:s_complete}, we may use \cite[Theorem 2.27]{folland1999real} to differentiate under the integral sign to get $\int_M \partial_t^k p_t(x,\cdot) \mathrm d\mu = 0$ for all $(t,x) \in (0,\infty)\times M.$ 

By \cite[Proposition 3]{lierl2015scale}, $(P_t)_{t}$ forms a Feller semi-group on $\mathcal C_M$, hence,  we may use \cite[Lemma 2.27]{folland1999real} to see for any $f \in \mathcal C_M$ and $(t,x) \in (0,\infty)\times M$
\begin{align*}
     Q^{(k)}_tf(x) = (t\partial_t)^kP_t f(x) =\int_M t^k \partial_t^kp_t(x,y) f(y) \mu(\mathrm dy)=:\int_M  q_{k,t}(x,y) f(y) \mu(\mathrm dy).
\end{align*}
 On the other hand, using the identity $Q^{(k)}_t = (tL)^k P_t = 2^k(tL/2)^k P_{t/2} P_{t/2}$, we see by Lemma \ref{lem:long time heat kernel bounds}, there is some $c>0$ so that uniformly in $(t,x,y,y') \in (0,\infty)\times M^3$ with $y,y' \le (t/2)^{1/d_w}$,
\begin{align*}
    \abs{q_{k,t}(x,y) -q_{k,t}(x,y')}  &\lesssim \int_M \abs{q_{k,t/2}(x,z)} \abs{p_{t/2}(z,y) - p_{t/2}(z,y')} \mu(\mathrm dz)\\
    &\lesssim\int_M p_{ct}(x,z) \abs{p_{t/2}(z,y) - p_{t/2}(z,y')} \mu(\mathrm dz)\\
    &\lesssim \pr{\frac{d(y,y')}{t^{1/d_w}}}^\Theta\int_M p_{ct}(x,z) p_{ct}(z,y) \mu(\mathrm dz) e^{-ct}\\
    &\le \pr{\frac{d(y,y')}{t^{1/d_w}}}^\Theta p_{2ct}(x,y)e^{-ct}.
\end{align*}
It remains to show \eqref{ineq: q sub gaussian bound} for $t\ge 1$. However, this follows from \eqref{eq:locality integral q equals 0} and Lemma \ref{lem:long time semi-group norm bound} as in the proof of Lemma \ref{lem:long time heat kernel bounds}.
\end{proof}

\begin{proof}[Proof of Lemma \ref{lem:hk LpLq interpolation}]
Note that uniformly in $t >0$ and $x,y \in M$, $p_t(x,y) \lesssim 1+t^{- d_h/d_w}$. Hence,
\begin{align*}
  \norm{e^{-tL}}_{L^1_M \to L^\infty_M} \lesssim 1+t^{- d_h/d_w}.
\end{align*}
Then the Riesz–Thorin interpolation theorem, gives for each $p \ge 1$,
\begin{align*}
  \norm{e^{-tL}}_{L^p_M \to L^\infty_M}  \le \norm{e^{-tL}}_{L^1_M \to L^\infty_M}^\frac{1}{p} \norm{e^{-tL}}_{L^\infty_M \to L^\infty_M} ^{1-\frac{1}{p}} \lesssim 1+t^{-\frac{d_h}{d_wp}}.
\end{align*}
Apply Riesz–Thorin again to see for $q \ge p$,
\begin{align*}
  \norm{e^{-tL}}_{L^p_M \to L^q_M}  \le \norm{e^{-tL}}_{L^p_M \to L^\infty_M}^{1-\frac{p}{q}}  \norm{e^{-tL}}_{L^p_M \to L^p_M}^{\frac{p}{q}}  \lesssim 1+ t^{- \frac{d_h}{d_w}\pr{\frac{1}{p} - \frac{1}{q}}}.
\end{align*}
\end{proof}

\subsection{Proof of  Proposition \ref{prop: mod of baudoin berstein} } 
\begin{comment}
\begin{proof}[Proof of Lemma \ref{lem:elementary maximizer}]
Let us fix $s >0$ and note that it is enough to consider the maximizer of $\varphi(t):= \log F(t,s) = p\log(t)+ q\log(t+s)$.
\begin{align*}
  \varphi'(t) = \frac{p}{t}+ \frac{q}{t+s},\qquad t>0.
\end{align*}
For $p+q \ge 0$, the map $t \mapsto \varphi(t)$ (and hence $t \mapsto F(t,s)$) is non-decreasing for $t>0$, which implies
\begin{align*}
  \sup_{t \in (0,1]} F(t,s) = F(1,s)= (1+s)^q.
\end{align*}

Now suppose $p+q < 0$, then solving for $\varphi'(t) >0$ to see $F(\cdot,s)$ is increasing in $\left(0, \frac{-ps}{p+q} \right]$ and decreases for $t \ge \frac{-ps}{p+q} $. Hence
\begin{align*}
  \sup_{t >0} F(t,s) =  \pr{\frac{-p}{p+q}}^p \pr{\frac{-p}{p+q} +1}^q s^{p+q}.
\end{align*}
\end{proof}
\end{comment}
\begin{proof}[Proof of Proposition \ref{prop: mod of baudoin berstein}]
Let $n \in \mathbb N$ be positive and  $p \in \mathbb N$ is odd with $ p > n - 2$. Then for any $a,b \in \R$,
\begin{align}\label{ineq:elementary inequality}
  \abs{a^{2p+n - 2} - b^{2p+n - 2}} \le \frac{2p+n - 2}{p}\pr{\abs{a}^{p+n - 2}+\abs{b}^{p+n - 2} } \abs{a^p - b^p}.
\end{align}
Indeed, if $a,\, b$ have the same sign, i.e. $ab \ge 0$, then the inequality holds by the mean value theorem. Suppose $a,b$ have different signs, i.e. $ab < 0$. Let us assume without loss of generality that  $a>0>b = -c$ for some $c>0$,   then
\begin{align*}
  \pr{2+ \frac{n-2}{p}}\pr{\abs{a}^{p+n - 2}+\abs{b}^{p+n - 2}}\abs{a^p - b^p}&=\pr{2+ \frac{n-2}{p}}\pr{a^{p+n - 2}+c^{p+n - 2}}\pr{a^p + c^p}\\
  &\ge a^{2p+n-2}+c^{2p+n - 2},
\end{align*}
where the equality is due to the fact that $p$ is odd and $\abs{\frac{n-2}{p}}\le 1$. If $n$ is even, then we see \eqref{ineq:elementary inequality} holds. On the other hand, if $n$ is odd, then
\begin{align*}
  \abs{a^{2p+n - 2} - b^{2p+n-2}} =a^{2p+n - 2} + c^{2p+n-2},
\end{align*}
which also implies \eqref{ineq:elementary inequality} holds. Let $f: M\to \R$ be a measurable function, define $v(x):= f(x)^{p}$ and $u(x):= |f(x)|^{p+n - 2}$ for all $x \in M$. Denote $C_p := \frac{2p+n-2}{p}$ and consider for any $\alpha \ge 0$ that
\begin{align*}
  \norm{f^{2p+n - 2}}_{1,\alpha} &=\sup_{t >0} t^{-\alpha} \pr{\iint_{M^2} \abs{f(x)^{2p+n - 2} - f(y)^{2p+n - 2}}p_t(x,y)\mu(\mathrm dx) \mu(\mathrm dy)}  \\
  &\overset{\eqref{ineq:elementary inequality}}{\le}  C_p\sup_{t >0} t^{-\alpha} \iint_{M^2}\pr{ u(x)+u(y)}\abs{v(x) - v(y)}p_t(x,y) \mu(\mathrm dy) \mu(\mathrm dx)\\
  &\le 2C_p \sup_{t >0} t^{-\alpha} \pr{\iint_{M^2}u(x)p_t(x,y)\mu(\mathrm dy)\mu(\mathrm dx)}^\frac{1}{2}\pr{\iint_{M^2}\pr{v(x) - v(y)}^2p_t(x,y) \mu(\mathrm dy)\mu(\mathrm dx)}^\frac{1}{2}\\
  &\le 2C_p\norm{u^2}_{L^1_M}^\frac{1}{2} \sup_{t >0}\pr{t^{-2\alpha} \iint_{M^2}\pr{v(x) - v(y)}^2 p_t(x,y) \mu(\mathrm dy)\mu(\mathrm dx)}^\frac{1}{2}\\
  &=2C_p \norm{u^2}_{L^1_M}^\frac{1}{2}  \norm{v}_{2,2\alpha}^\frac{1}{2},
\end{align*}
where we used H\"{o}lder's inequality in the third line. This completes the proof of the first desired inequality. The second inequality follows immediately by \cite[Proposition 4.6]{UConn1}.
\end{proof}

\subsection{Proof of Lemma \ref{lem:cacciopoli} and Lemma \ref{lem:besov_norm_compare}}

To prove Lemma \ref{lem:cacciopoli}, we will first need to state some consequences of our assumptions  due to \cite[Theorem 4.9]{anttila2025approach}.

Denote for any measurable function $f$ on $M$ and open set $B \subset M$ the average of $f$ on $B$ as $f_B$, i.e.
\begin{align*}
    f_B:= \frac{1}{\mu(B)}\int_B f\mathrm d\mu.
\end{align*}
We say the \textit{Poincare inequality}  holds if there exists $\kappa\ge 1$ so that the following inequality holds uniformly in $(r,x) \in (0,\infty) \times M$ and $f \in \mathcal F$,
\begin{align}\label{ineq:pioncare}
    \int_{B(x,r)} \pr{f - f_{B(x,r)}}^2 \mathrm d\mu \lesssim r^{d_w} \int_{B(x,\kappa r)  } \mathrm{d}\Gamma(f,f).  \tag{$\mathrm{PI}_{d_w}$}
\end{align}
Let $E\subset F \subset M$ be two non-empty open subsets, we say $\varphi$ is a \textit{cutoff function} of $E \subset F$ if $\varphi$ taking values in $[0,1]$ $\mu$-almost everywhere, $\varphi\equiv 1$ on a neighborhood of $\overline{E}$ and $\varphi\equiv 0$ on $F^c$.

We say the \textit{cutoff Sobolev inequality} \eqref{ineq:cutoff sob} holds if there exists $C_S>0$ so that for every $x \in M$ and $R,r>0$, there is a cutoff function $\varphi \in \mathcal F$ for $B(x,R) \subset B(x,R+r)$ so that the following inequality holds for all  $f \in \mathcal F$,
\begin{align}\label{ineq:cutoff sob}
    \int_{M} f^2 \mathrm d \Gamma(\varphi,\varphi) \le \frac{1}{8}\int_{B(x,R+r)\backslash B(x,R)}\varphi^2\mathrm d \Gamma(f,f) + \frac{C_S}{r^{d_w}}\int_{B(x,R+r)\backslash B(x,R)}f^2 \mathrm d\mu,\tag{$\mathrm{CS}(d_w)$}
\end{align}
where we have identified $\varphi$ with its quasi-continuous version. For details about quasi-continuity, we refer the reader to \cite{FukushimaDF}.
\begin{theorem}[{\cite[Theorem 2.8]{kajino2020singularity}}]\label{thm:subg hk imply poincare and CS}
    Let $(\mathcal E,\mathcal F)$ be a strongly local regular Dirichlet form on $L^2_M$ whose heat kernel satisfies \eqref{ineq:hk_subg_compact_form}, then $(\mathcal E,\mathcal F)$ satisfies both the Poincare inequality \eqref{ineq:pioncare} and the cutoff Sobolev inequality \eqref{ineq:cutoff sob} hold.  
\end{theorem}
\begin{remark}
    See \cite{anttila2025approach} for the converse of Theorem \ref{thm:subg hk imply poincare and CS}. 
\end{remark}

\begin{proof}[Proof of Lemma \ref{lem:cacciopoli}]
    We fix arbitrary $x \in M$ and $r \in (0,\mathrm{diam}(M)]$ in this proof, and denote $\sigma B:= B(x,\sigma r)$ for any $\sigma >0$. We will let $\varphi$ be the cutoff function for $B\subset 2B$ provided by \eqref{ineq:cutoff sob} with $R=r$.
    
    Let $f \in \mathcal D_2(L)\cap \mathcal C_M$ and set $\bar{f} := f - f_{2B}$. By the strong locality of $(\mathcal E, \mathcal F)$, we see $\Gamma(f,f) = \Gamma(\bar{f},\bar{f})$. Hence, we will work with $\bar{f}$.    
    
    By the strong locality of the Dirichlet form, the Leibniz rule gives:
    \begin{align*}
        \mathrm{d}\Gamma(\bar{f}, \varphi^2 \bar{f}) = \varphi^2 \mathrm{d}\Gamma(\bar{f},\bar{f}) + 2\varphi \bar{f} \mathrm{d}\Gamma(\bar{f},\varphi).
    \end{align*}
    Integrating this over the space, and applying the definition of the generator $L$ for $f \in \mathcal D_2(L)$ (noting that $Lf = L\bar{f}$), we obtain:
    \begin{align*}
       \int_{2B} \varphi^2 \bar{f} (Lf) \mathrm{d}\mu =\int_M L(\overline{f})\varphi^2 \overline{f} \mathrm d\mu = \int_M \mathrm d \Gamma\pr{\overline{f}, \varphi^2 \overline{f}} =   \int_{2B} \varphi^2 \mathrm{d}\Gamma(f,f) + 2 \int_{2B} \varphi \bar{f} \mathrm{d}\Gamma(f,\varphi).
    \end{align*}
    Rearranging to isolate the energy of $f$ on the left side:
    \begin{align}\label{eq:direct_cacc}
        \int_{2B} \varphi^2 \mathrm{d}\Gamma(f,f) = \int_{2B} \varphi^2 \bar{f} (Lf) \mathrm{d}\mu - 2 \int_{2B} \varphi \bar{f} \mathrm{d}\Gamma(f,\varphi).
    \end{align}
    
    We bound the two terms on the right-hand side of \eqref{eq:direct_cacc} in absolute value. 
    For the first term, applying Young's inequality ($ab \le \frac{\epsilon}{2}a^2 + \frac{1}{2\epsilon}b^2$ with $\epsilon = r^{-d_w}$), and noting $\varphi \le 1$:
    \begin{align}\label{eq:bound_1}
        \left| \int_{2B} \varphi^2 \bar{f} (Lf) \mathrm{d}\mu \right| \le \frac{r^{-d_w}}{2} \int_{2B} |\bar{f}|^2 \mathrm{d}\mu + \frac{r^{d_w}}{2} \int_{2B} |Lf|^2 \mathrm{d}\mu.
    \end{align}
    
    For the second term, we apply the Cauchy-Schwarz inequality for the energy measure $\mathrm{d}\Gamma$, followed by the integral Cauchy-Schwarz inequality, and finally Young's inequality in the form $2ab \le \frac{1}{2}a^2 + 2b^2$:
    \begin{align}\label{eq:bound_2}
        2 \left| \int_{2B} \varphi \bar{f} \mathrm{d}\Gamma(f,\varphi) \right| 
        &\le 2 \left( \int_{2B} \varphi^2 \mathrm{d}\Gamma(f,f) \right)^\frac{1}{2} \left( \int_{2B} |\bar{f}|^2 \mathrm{d}\Gamma(\varphi,\varphi) \right)^\frac{1}{2} \nonumber \\
        &\le \frac{1}{2} \int_{2B} \varphi^2 \mathrm{d}\Gamma(f,f) + 2 \int_{2B} |\bar{f}|^2 \mathrm{d}\Gamma(\varphi,\varphi).
    \end{align}
    Substituting \eqref{eq:bound_1} and \eqref{eq:bound_2} back into \eqref{eq:direct_cacc} and subtracting $\frac{1}{2}\int_{2B} \varphi^2 \mathrm{d}\Gamma(f,f)$ from both sides leaves:
    \begin{align}\label{eq:intermediate_cacc}
        \frac{1}{2} \int_{2B} \varphi^2 \mathrm{d}\Gamma(f,f) \le \frac{r^{-d_w}}{2} \int_{2B} |\bar{f}|^2 \mathrm{d}\mu + \frac{r^{d_w}}{2} \norm{Lf}_{L^2_{2B}}^2 + 2 \int_{2B} |\bar{f}|^2 \mathrm{d}\Gamma(\varphi,\varphi).
    \end{align}
    By Theorem \ref{thm:subg hk imply poincare and CS} and \eqref{ineq:cutoff sob}, we see 
    \begin{align*}
        2\int_{2B} \abs{\overline{f}}^2 \mathrm d\Gamma(\varphi,\varphi) \le \frac{1}{4}\int_{2B } \varphi^2\mathrm{d}\Gamma(f,f) + \frac{2C_S}{r^{d_w}}\int_{2B\backslash B} \abs{\overline{f}}^2 \mathrm d\mu.
    \end{align*}
    Inserting this back into \eqref{eq:intermediate_cacc}, and unitize the fact that $\varphi$ is bounded between $0$ and $1$ to see the desired result. 
\end{proof}

\begin{proof}[Proof of Lemma \ref{lem:besov_norm_compare}]
For (1): suppose $\alpha \le \beta$, then
\begin{align*}
  \norm{u}_{\mathcal B^\alpha_{\infty}} = \norm{P_1u}_{L_M^\infty}+ \sup_{t \in (0,1]} t^{-\frac{\alpha}{d_w}} \norm{Q^{(b)}_tu}_{L^\infty_M}\le \norm{P_1u}_{L_M^\infty}+ \sup_{t \in (0,1]} t^{-\frac{\beta}{d_w}} \norm{Q^{(b)}_tu}_{L^\infty_M}=\norm{u}_{\Binf{\beta}}.
\end{align*}
For (2), we note if $\alpha \le 0$,
\begin{align*}
  \norm{u}_{\mathcal B^\alpha_{\infty}} = \norm{P_1u}_{L_M^\infty}+ \sup_{t \in (0,1]} t^{-\frac{\alpha}{d_w}} \norm{Q^{(b)}_tu}_{L^\infty_M}\lesssim\norm{u}_{L^\infty_M}+ \sup_{t \in (0,1]} t^{-\frac{\alpha}{d_w}} \norm{u}_{L^\infty_M}\lesssim \norm{u}_{L^\infty_M}.
\end{align*}
For (3), we take $\beta \in (0, \alpha\wedge \Theta)$, then
\begin{align*}
  \norm{u}_{L^\infty_M}\le \norm{u}_{{\mathcal C}^\beta} \lesssim \norm{u}_{\mathcal B^\beta_{\infty}}\lesssim \norm{u}_{\mathcal B_\infty^\alpha}.
\end{align*}
\end{proof}

\subsection{Proof of Lemma \ref{lem: holder regularity of mollifed y} and Lemma \ref{lem:imcomplete beta function bounds}}
\begin{proof}[Proof of Lemma \ref{lem: holder regularity of mollifed y}]
Since for $n \ge 2$, the Wick powers $Y^{:n:}_\varepsilon$ of $Y_\varepsilon$ are defined via Hermit polynomials and the space of H\"{o}lder continuous functions forms an algebra, it is enough to show both $Y_\varepsilon \in \mathcal C^{\delta_1}( \mathcal C_M^{\delta_2})$ for some  $\delta_1,\delta_2>0$, and the renormalization counterterm $C_\varepsilon:M\to \R$ given by Definition \ref{def:wick renormaliztion of distributions} is in $\mathcal C^{\delta_2}_M$ for some sufficiently small $\delta_2$. 

Suppose $Y$ is the solution to \eqref{eq: EW eqn} with $Y_0 \in \Binf{-\alpha}$. Then for every $\varepsilon>0$ and $(t,x) \in \R_+\times M$,
\begin{align*}
    Y_\varepsilon(t,x)&= P_{t}P_{\varepsilon }\pr{Y_0}(x) + \int_0^t \int_M p_{t-s+\varepsilon}(x,y) \xi\pr{\mathrm dy, \mathrm ds}=:P_t P_\varepsilon\pr{Y_0}(x)+\widetilde Y_\varepsilon(t,x).
\end{align*}

Recall for $G_\varepsilon:M^2 \to \R$ from \eqref{eq:green function}. It holds uniformly in $\varepsilon \in (0,1]$  and $y,z ,x\in M$ that 
\begin{align*}
  \abs{ G_\varepsilon(x,y) - G_\varepsilon(x,z)}=& 2\abs{\int_0^\infty e^{-t} \pr{p_{t+2\varepsilon}(x,y) - p_{t+2\varepsilon}(x,z) }\mathrm dt}\\
  &= 2\abs{\int_0^\infty e^{-t} \int_M p_{t}(x,w)\pr{p_{2\varepsilon}(w,y) - p_{2\varepsilon}(w,z)}\mu(\mathrm dw) \mathrm dt}\\
    &\overset{\eqref{ineq:holder_kernel}}{\lesssim}\varepsilon^{-\frac{d_h}{d_w} - \Theta} d(x,y)^\Theta.
\end{align*}
Hence uniformly in $\varepsilon>0$ and $x,x' \in M$,
\begin{align*}
    \abs{C_{\varepsilon}(x) - C_\varepsilon(x')} &= \abs{G_\varepsilon(x,x) - G_\varepsilon(x,x') +G_\varepsilon(x,x') - G_\varepsilon(x',x')}\lesssim \varepsilon^{-\frac{d_h}{d_w}-\Theta }d(x,x')^\Theta.
\end{align*}
Also, by Lemma \ref{lem:long time heat kernel bounds}, for each $\varepsilon \in (0,1]$
\begin{align*}
    \sup_{x\in M} C_\varepsilon(x) &\overset{\eqref{ineq:sgu}}{\lesssim}\int_0^1 (t+2\varepsilon)^{-\frac{d_h}{d_w}} \mathrm dt + \int_1^\infty e^{-t} \mathrm dt,
\end{align*}
which is finite. Hence, for all $\varepsilon \in (0,1],\,C_\varepsilon\in \mathcal C^\frac{\Theta}{2}_M.$

Next, we show that for each $\varepsilon >0$ and $T>0$, the process $Y_\varepsilon$ is jointly continuous on $[0,T]\times M$ via the Kolmogorov's continuity lemma. Let $t \in [0,T]$ and $x,y \in M$, consider
\begin{align*}
  \norm{\widetilde Y_\varepsilon(t,x) - \widetilde Y_\varepsilon(t,x')}_{L^{2p}_\Omega}^2&\lesssim \int_0^t \int_M \int_M \pr{p_{t-s}(x',y_1) - p_{t-s}(x,y_1)}\\
  &\qquad\qquad \times \pr{p_{t-s}(x',y_2) - p_{t-s}(x,y_2)}p_\varepsilon(y_1,y_2)\mu(\mathrm dy_1) \mu(\mathrm dy_2) \mathrm ds\\
  &=\int_0^t\int_M \pr{p_{t-s+\varepsilon}(x',y_1) - p_{t-s+\varepsilon}(x,y_1)}\pr{p_{t-s}(x',y_1) - p_{t-s}(x,y_1)}\mu(\mathrm dy_1)\\
  &=\int_0^t p_{2(t-s)+\varepsilon}(x',x') - 2p_{2(t-s)+\varepsilon}(x',x)+p_{2(t- s)+\varepsilon}(x,x)\mathrm ds\\
  &\overset{\eqref{ineq:holder_kernel}}{\lesssim}\int_0^t \pr{\frac{d(x,x')}{(t-s+\varepsilon)^\frac{1}{d_w}}}^\Theta (t-s+\varepsilon)^{- \frac{d_h}{d_w}}\mathrm ds\\
  &\lesssim d(x,x')^\Theta \varepsilon^{1 - \frac{d_h+\Theta}{d_w}}.
\end{align*}

Similarly, we have
\begin{align*}
  \norm{\widetilde Y_\varepsilon(t,x) - \widetilde Y_\varepsilon(t',x)}_{L^{2p}_\Omega}^2&\lesssim \int_0^t\int_M\pr{p_{t'-s+\varepsilon}(x,y) - p_{t-s+\varepsilon}(x,y)} \pr{p_{t'-s}(x,y) - p_{t-s}(x,y)} \mu(\mathrm dy) \mathrm ds\\
  &\qquad +\int_t^{t'} p_{2(t'-s)+\varepsilon}(x,x) \mathrm ds\\
  &\lesssim\int_0^t \pr{p_{2(t'-s)+\varepsilon}(x,x) - 2p_{t+t'-2s+\varepsilon}(x,x) -p_{2(t-s)+\varepsilon}(x,x)} \mathrm ds\\
  &\qquad+\int_t^{t'} p_{2(t'-s)+\varepsilon}(x,x) \mathrm ds.
\end{align*}
Let us call the two terms on the right hand side $I_1$ and $I_2$ respectively. We have
\begin{align*}
  I_1 &\lesssim \int_0^t (2t-2s+\varepsilon)^{- \frac{d_h}{d_w}-1}\mathrm ds \abs{t-t'}\lesssim \varepsilon^{- \frac{d_h}{d_w} } \abs{t-t'}
\end{align*}
and
\begin{align*}
  I_2 &\overset{\eqref{ineq:sgu}}{\lesssim }\int_t^{t'}(2t-2s+\varepsilon)^{-\frac{d_h}{d_w}} \mathrm ds\lesssim\varepsilon^{-\frac{d_h}{d_w}}\abs{t-t'}.
\end{align*}
Therefore, by triangular inequality,
\begin{align*}
  \norm{\widetilde Y_\varepsilon(t,x) -\widetilde Y_\varepsilon(t',x')}_{L^{2p}} \lesssim \pr{d(x,x')^\frac{\Theta}{2} + \abs{t-t'}^\frac{1}{2}} \varepsilon^{- \frac{d_h}{2d_w}}.
\end{align*}
Hence by the Kolmogorov's continuity lemma (c.f.\cite[Theorem 1.1]{kratschmer2023kolmogorov}), we see for all $q \ge 1$ and  sufficiently small $\delta_1,\delta_2>0$,
\begin{align}\label{ineq: y epsilon joint holder expectation bound}
  \norm{ \sup_{(t,x),\, (t',x') \in [0,T]\times M \atop (t,x) \neq (t',x') } \frac{\abs{\widetilde Y_\varepsilon(t,x) - \widetilde Y_\varepsilon(t',x')}}{\abs{t-t'}^{2\delta_1} + d(x,x')^{2\delta_2}}}_{L^{q}_\Omega} \lesssim 1.
\end{align}
In addition, we see from \cite[Theorem 11.18]{ledoux2013probability} that for each $T>0$ , $ q \ge 1$ and $\varepsilon >0$,
\begin{align}\label{ineq: y epsilon sup bound}
  \norm{\sup_{(t,x) \in [0,T]\times M} \widetilde Y_\varepsilon(t,x)}_{L^q_\Omega} \lesssim1.
\end{align}
Hence, we see from a sequence of elementary inequalities that for any $\varepsilon >0$, $q \ge 1$ and $T>0$, there are sufficiently small (and possibly different) $\delta_1,\delta_2>0$
\begin{align*}
  \norm{\norm{\widetilde Y_\varepsilon}_{\mathcal C_T^{\delta_1} \mathcal C_M^{\delta_2}} }_{L^q_\Omega}&= \norm{\sup_{t \in [0,T]} \norm{\widetilde Y_\varepsilon(t,\cdot)}_{\mathcal C_M^{\delta_2}} +\sup_{t,t' \in [0,T]\atop t \neq t'} \frac{\norm{ \widetilde Y_\varepsilon(t,\cdot) - \widetilde Y_\varepsilon(t',\cdot)}_{\mathcal C_M^{\delta_2}}}{\abs{t-t'}^{\delta_1}}}_{L^q_\Omega}< \infty.
\end{align*}
Finally, we see from Proposition \ref{prop:holder=besov}, for $\delta_2 \in (0,\Theta)$, we have
\begin{align*}
  \norm{\norm{\widetilde Y_\varepsilon}_{\mathcal C_T \mathcal B^{\delta_2}_{\infty}}}_{L^q_\Omega} < \infty.
\end{align*}
Finally, by Lemma \ref{lem:q_regulariz} and Lemma \ref{lem:strong continuity of heat semi group}, we see for each $\varepsilon \in (0,1]$, $P_tP_\varepsilon\pr{ Y_0} \in \mathcal C^{\delta_1}\mathcal C^{\delta_2}_M$. Collecting all terms above to see the desired result.  
\end{proof}
\begin{proof}[Proof of Lemma \ref{lem:imcomplete beta function bounds}]
We split the integral at $s = u$:
\begin{align*}
  \mathrm I(u) = \int_0^u s^{-a}(u+s)^{-b} \mathrm ds+\int_u^1 s^{-a}(u+s)^{-b} \mathrm ds =: \mathrm I_1(u) + \mathrm I_2(u).
\end{align*}
For $0< s \le u$, we have $u < u+s \le 2 u$ so $(2u)^{-b} \le  (u+s)^{-b} < u^{-b}$. Hence uniformly in $u \in (0,1]$,
\begin{align*}
  \mathrm I_1(u) \asymp u^{-b} \int_0^u s^{-a} \mathrm ds \asymp u^{1-b-a}.
\end{align*}
Now for $u \le s \le 1$, we have $s \le u+s \le 2s$, so $(2s)^{-b} \le (u+s)^{-b} \le s^{-b}$. Hence uniformly in $u \in (0,1]$,
\begin{align*}
  \mathrm I_2(u) \asymp \int_{u}^1 s^{-b-a} \mathrm ds \asymp
  \begin{cases}
    \abs{1-u^{1-a-b}},& \text{ for } a+b \neq 1, \\
    \log \pr{1/u},&\text{ for }a+b = 1.
  \end{cases}
\end{align*}
Collecting all terms above we see
\begin{align*}
  \mathrm I(u) = \mathrm I_1(u) + \mathrm I_2(u) \asymp
  \begin{cases}
    1,& \text{ if }a+b <1,\\
    1+\log(1/u),&\text{ if }a+b = 1,\\
    u^{1-a-b},&\text{ if }a+b >1,
  \end{cases}
\end{align*}
which is the desired result.
\end{proof}

\section*{Declaration of competing interest}
The authors declare that they have no known competing financial interests or personal relationships that could have appeared to influence the work reported in this paper.

\section*{Data availability}
Data sharing is not applicable to this article, as no datasets were generated or analyzed during the current study.

\newpage
\bibliographystyle{alpha}
\bibliography{ref}

\end{document}